\documentclass[11pt]{article}
\usepackage{amsmath,amsthm,amssymb}
\usepackage[usenames,dvipsnames]{xcolor}
\usepackage{enumerate}
\usepackage{graphicx}
\usepackage{cite}
\usepackage{comment}
\usepackage{oands}
\usepackage{tikz}
\usepackage{changepage}
\usepackage{bbm}
\usepackage{mathtools}
\usepackage[margin=1in]{geometry}

\usepackage[pagewise,mathlines]{lineno}
\usepackage{appendix}
\usepackage{stmaryrd} 
\usepackage{multicol}
\usepackage{microtype}
\usepackage{subcaption}
\usepackage[colorinlistoftodos,textsize=footnotesize]{todonotes}

\definecolor{dgreen}{RGB}{0,180,0}

\usepackage[pdftitle={Conformal invariance of non-simple CLE on the Riemann sphere},
  pdfauthor={Ewain Gwynne, Jason Miller, and Wei Qian},
colorlinks=true,linkcolor=NavyBlue,urlcolor=RoyalBlue,citecolor=PineGreen,bookmarks=true,bookmarksopen=true,bookmarksopenlevel=2,unicode=true,linktocpage]{hyperref}

\theoremstyle{plain}
\newtheorem{thm}{Theorem}[section]
\newtheorem{cor}[thm]{Corollary}
\newtheorem{lem}[thm]{Lemma}
\newtheorem{prop}[thm]{Proposition}

\theoremstyle{plain}
\newtheorem{theorem}{Theorem}[section] 
\newtheorem{lemma}[theorem]{Lemma}
 
\newtheorem{corollary}[theorem]{Corollary}

\def\@rst #1 #2other{#1}
\newcommand\MR[1]{\relax\ifhmode\unskip\spacefactor3000 \space\fi
  \MRhref{\expandafter\@rst #1 other}{#1}}
\newcommand{\MRhref}[2]{\href{http://www.ams.org/mathscinet-getitem?mr=#1}{MR#2}}

\theoremstyle{definition}
\newtheorem{defn}[thm]{Definition}
\newtheorem{remark}[thm]{Remark}

\newtheorem{definition}[theorem]{Definition}

\numberwithin{equation}{section}

\newcommand{\dsb}{\begin{adjustwidth}{2.5em}{0pt}
\begin{footnotesize}}
\newcommand{\dse}{\end{footnotesize}
\end{adjustwidth}}

\newcommand{\ssb}{\begin{adjustwidth}{2.5em}{0pt}}
\newcommand{\sse}{\end{adjustwidth}}

\newcommand{\aryb}{\begin{eqnarray*}}
\newcommand{\arye}{\end{eqnarray*}}
\def\alb#1\ale{\begin{align*}#1\end{align*}}
\def\allb#1\alle{\begin{align}#1\end{align}}
\newcommand{\eqb}{\begin{equation}}
\newcommand{\eqe}{\end{equation}}
\newcommand{\eqbn}{\begin{equation*}}
\newcommand{\eqen}{\end{equation*}}

\newcommand{\BB}{\mathbbm}
\newcommand{\ol}{\overline}
\newcommand{\ul}{\underline}
\newcommand{\op}{\operatorname}

\newcommand{\re}{\operatorname{Re}}

\newcommand{\eqD}{\overset{d}{=}}
\newcommand{\ep}{\varepsilon}
\newcommand{\rta}{\rightarrow}

\newcommand{\wt}{\widetilde}
\newcommand{\wh}{\widehat} 
\newcommand{\mcl}{\mathcal}

\newcommand{\bdy}{\partial}

\newcommand{\CLE}{{\mathrm {CLE}}}
\newcommand{\SLE}{{\mathrm {SLE}}}

\newcommand{\eps}{\varepsilon}

\newcommand{\CK}{\mathcal {K}}
\newcommand{\CL}{\mathcal {L}}

\let\originalleft\left
\let\originalright\right
\renewcommand{\left}{\mathopen{}\mathclose\bgroup\originalleft}
\renewcommand{\right}{\aftergroup\egroup\originalright}

\title{Conformal invariance of $\CLE_\kappa$ on the Riemann sphere for $\kappa \in (4,8)$}
\date{  }
\author{
\begin{tabular}{c} Ewain Gwynne \ \ \ Jason Miller \ \ \ Wei Qian\\ {\it University of Cambridge} \end{tabular}
}

\begin{document}

\maketitle

\begin{abstract}
The conformal loop ensemble ($\CLE$) is the canonical conformally invariant probability measure on non-crossing loops in a simply connected domain in $\BB C$ and is indexed by a parameter $\kappa \in (8/3,8)$.  We consider $\CLE_\kappa$ on the whole-plane in the regime in which the loops are self-intersecting ($\kappa \in (4,8)$) and show that it is invariant under the inversion map $z \mapsto 1/z$.  This shows that whole-plane $\CLE_\kappa$ for $\kappa \in (4,8)$ defines a conformally invariant measure on loops on the Riemann sphere. The analogous statement in the regime in which the loops are simple ($\kappa \in (8/3,4]$) was proven by Kemppainen and Werner and together with the present work covers the entire range $\kappa \in (8/3,8)$ for which $\CLE_\kappa$ is defined.  As an intermediate step in the proof, we show that $\CLE_\kappa$ for $\kappa \in (4,8)$ on an annulus, with any specified number of inner-boundary-surrounding loops, is well-defined and conformally invariant. 
\end{abstract}

\tableofcontents

\section{Introduction}
\label{sec-intro}

The Schramm-Loewner evolution (SLE$_\kappa$) is the canonical conformally invariant probability measure on non-crossing curves in a simply connected domain in $\BB C$.  It was originally introduced by Schramm~\cite{schramm0} in 1999 as a candidate to describe the scaling limits of discrete planar lattice models from statistical mechanics.  The parameter $\kappa > 0$ controls the ``windiness'' of the curve.  For $\kappa \in (0,4]$ the curve is simple (i.e., does not have self-intersections), for $\kappa \in (4,8)$ the curve is self-intersecting but not space-filling and for $\kappa \geq 8$ it is space-filling \cite{schramm-sle}.

Since its invention by Schramm, SLE$_\kappa$ has been shown to be the scaling limit of various discrete random curves arising in statistical mechanics, both on deterministic lattices and on random planar maps.  Examples of such models include loop-erased random walk~\cite{lsw-lerw-ust} ($\kappa=2$), Ising model interfaces~\cite{smirnov-ising} ($\kappa=3$), percolation interfaces~\cite{smirnov-cardy,camia-newman-sle6} ($\kappa=6$).   

There are several different flavors of SLE$_\kappa$.  The most common of these are chordal, radial, and whole-plane.  Chordal SLE$_\kappa$ describes a curve connecting two distinct boundary points in a simply connected domain, radial describes a curve connecting a boundary point to an interior point, and whole-plane describes a curve connecting two points in the Riemann sphere.  A key property of SLE$_\kappa$ is \emph{conformal invariance}: if $\eta$ is an SLE$_\kappa$ from $x$ to $y$ in $D$ and $f : D\rta\wt D$ is a conformal map, then the law of $f(\eta)$ is that of an SLE$_\kappa$ from $f(x)$ to $f(y)$ in $\wt D$.

The focus of the present work is on the conformal loop ensemble ($\CLE_\kappa$), introduced by Sheffield \cite{shef-cle}, which is the loop version of SLE$_\kappa$. A CLE$_\kappa$ consists of a random countable collection of non-crossing loops in a simply connected domain $D\subset \BB C$, each of which locally looks like an $\SLE_\kappa$ curve. 
Just like $\SLE_\kappa$ arises as the scaling limit of a single interface in a number of discrete models, $\CLE_\kappa$ arises as the scaling limit of the full collection of interfaces: see, e.g.,~\cite{smirnov-cardy,camia-newman-sle6,smirnov-ising,kemp-smirnov-fk-bdy,benoist-hongler-cle3} for models on deterministic lattices and~\cite{shef-burger,gms-burger-cone,bhs-site-perc} for models on random planar maps. 

CLE$_\kappa$ is defined only for $\kappa \in (8/3,8)$. 
 When $\kappa \in (8/3,4]$, the loops of a $\CLE_\kappa$ are simple, do not intersect each other, and do not intersect the domain boundary.  When $\kappa \in (4,8)$, the loops are self-intersecting (but not self-crossing) and intersect (but do not cross) each other and the domain boundary.  The boundary cases $\kappa = 8/3$ and $\kappa=8$ correspond to an empty loop ensemble and the loop ensemble consisting of a single space-filling SLE$_8$-type loop, respectively.  In this paper we will primarily be interested in the case when $\kappa \in (4,8)$.  

The original definition of CLE is for $D\not=\BB C$.
This version of CLE is conformally invariant: if $f : D\rta \wt D$ is a conformal map and $\Gamma$ is a CLE$_\kappa$ in $D$, then $f(\Gamma)$ is a CLE$_\kappa$ in $\wt D$.  As we will describe in more detail below, unlike SLE$_\kappa$, the conformal invariance property of CLE$_\kappa$ is not built into its definition and requires a non-trivial proof.
 
When one speaks of $\CLE$ in $D$, one can either refer to its nested or non-nested versions.  The latter is obtained from the former by taking the outermost loops and the former is obtained from the latter by sampling an independent non-nested $\CLE_\kappa$ in each of the connected components of the complement of the loops and then iterating this procedure.  
In this article, we will be interested in a variant of CLE$_\kappa$ which is defined in the whole plane.
For this setting, only the nested version makes sense. Roughly speaking (and we will come back to this later), the whole plane CLE is the limit of a nested CLE in $D$ when $D$ tends to the whole plane. 

The construction of $\CLE_\kappa$ is based on a so-called \emph{branching SLE$_\kappa(\kappa-6)$ exploration tree} introduced in \cite{shef-cle}. For $\kappa\in(8/3,4]$, it was shown in \cite{shef-werner-cle} that one can also construct CLE$_\kappa$ using Brownian loop-soups. However for $\kappa\in(4,8)$, the branching SLE$_\kappa(\kappa-6)$ exploration tree remains the only method to construct $\CLE_\kappa$.
We will describe this process and its relationship to CLE$_\kappa$ in detail in Section~\ref{sec-branching}; see also \cite[Section~2]{msw-gasket} for a concise review in the case $\kappa \in (4,8)$. 
For now, we give a brief summary. 
SLE$_\kappa(\kappa-6)$ is a variant of chordal SLE$_\kappa$ which is target invariant in the following sense. 
Two SLE$_\kappa(\kappa-6)$ curves in a simply connected domain $D\subset\BB C$ with the same starting point $x\in\bdy D$ and different target points (either in the interior or the boundary of the domain) can be coupled together to agree until the first time that the two target points lie in different complementary connected components of the curve~\cite{sw-coord}. 
To define branching SLE$_\kappa(\kappa-6)$, one fixes a countable dense set $\{x_n\}_{n\in\BB N}$ in $D$ and constructs, using the target invariance property of SLE$_\kappa(\kappa-6)$, a ``tree" of $\SLE_\kappa(\kappa-6)$ processes starting from $x$ and targeted at the points $\{x_n\}_{n\in\BB N}$ with the following property.
The SLE$_\kappa(\kappa-6)$ ``branches" targeted at $x_i,x_j$ are the same until $x_i,x_j$ lie in different complementary connected components of the curve other and then evolve independently thereafter. 
It is shown in~\cite{shef-cle} that CLE$_\kappa$ can be constructed from branching SLE$_\kappa(\kappa-6)$ in such a way that the branch targeted at any given point corresponds to the exploration that one would obtain if one were to explore the loops of the $\CLE_\kappa$ starting from $x$, then follow the loops of the $\CLE_\kappa$ with the rule that whenever this process divides the domain into two parts, one continues exploring in the subdomain which contains the target point.  

Many of the important properties of $\CLE_\kappa$ are not obvious from its definition, including the fact that the collection of loops does not depend on the choice of root $x \in\bdy D$ and that the loops defined are in fact continuous paths.  In the case that $\kappa \in (8/3,4]$, these facts were established in \cite{shef-werner-cle} by showing that the outermost loops agree in law with the boundaries of so-called Brownian loop-soup clusters.  For $\kappa \in (4,8)$, these properties were established in \cite{shef-cle} conditionally on certain results for $\SLE_\kappa(\kappa-6)$ curves, which were later proved in \cite{ig1,ig3,ig4} using the connection between $\SLE$ and the Gaussian free field (GFF).  (See also \cite{cle-percolations} for a treatment of the case $\kappa \in (8/3,4]$ based on the GFF.)

The focus of the present work is on the whole-plane version of $\CLE_\kappa$.  This can be constructed by taking an increasing sequence of simply connected domains $D_n$ with $\bigcup_{n=1}^\infty D_n = \BB C$, for each $n \in \BB N$ letting $\Gamma_n$ be a $\CLE_\kappa$ on $D_n$, and then taking $\Gamma$ to be the limit of $\Gamma_n$ as $n \to \infty$ (see \cite{mww-nesting} for a detailed proof that the limit exists and does not depend on the sequence $(D_n)$).
Whole-plane CLE$_\kappa$ can equivalently be constructed by means of a whole-plane analog of the above branching $\SLE_\kappa(\kappa-6)$ construction (see Section~\ref{sec-branching}).  It is immediate from the construction that whole-plane $\CLE_\kappa$ is invariant under rescalings, rotations, and translations.  That is, whole-plane $\CLE_\kappa$ is invariant under conformal transformations $\BB C \to \BB C$ which fix $\infty$.  The purpose of the present work is to show that whole-plane $\CLE_\kappa$ for $\kappa \in (4,8)$ is also invariant under the inversion map $z \mapsto 1/z$ and therefore defines a conformally invariant family of loops on the Riemann sphere.

\begin{thm}
 \label{thm:main_result}
 Fix $\kappa \in (4,8)$ and suppose that $\Gamma$ is a whole-plane $\CLE_\kappa$.  Then the law of $\Gamma$ is invariant under inversion.  In particular, the law of $\Gamma$ is invariant under all M\"obius transformations of the Riemann sphere.	
\end{thm}

The analog of Theorem~\ref{thm:main_result} in the case $\kappa \in (8/3,4]$ was proved by Kemppainen and Werner \cite{werner-sphere-cle} using the Brownian loop-soup representation of $\CLE_\kappa$ \cite{shef-werner-cle}.  The argument that we give to prove Theorem~\ref{thm:main_result} will be based on the exploration tree construction from \cite{shef-cle}.  We expect that the arguments here could be generalized using the tools of \cite{cle-percolations} to establish the inversion symmetry for $\kappa \in (8/3,4]$ as well, but for simplicity we will focus on the case $\kappa \in (4,8)$.

Theorem~\ref{thm:main_result} is similar in spirit to reversibility results for SLE$_\kappa$, which say that time-reversing the curve does not change its law~\cite{zhan-reversibility,ig2,ig3,ig4}. The CLE$_\kappa$ analog of this is that swapping ``inside" and ``outside" for the origin-surrounding loops does not change the law of the CLE.

Theorem~\ref{thm:main_result} is natural from the perspective of loop models considered on random planar maps with the sphere topology.  It has been conjectured that the loops associated with many such models, after conformally embedding into the Riemann sphere, converge in the scaling limit to $\CLE_\kappa$.  Inverting the embedded loop ensemble (and hence the limiting $\CLE_\kappa$) corresponds to choosing a different collection of marked points to define the conformal embedding of the random planar map.  It would in principle be possible to deduce Theorem~\ref{thm:main_result} from the convergence of such a loop model to $\CLE_\kappa$ in a sufficiently strong topology, however the proof we will present here is directly based solely on continuum theory.

In the same vein, Theorem~\ref{thm:main_result} has applications to the continuum theory of Liouville quantum gravity (LQG)~\cite{shef-kpz,shef-zipper,wedges}.  For example, it follows from the results of \cite{sphere-constructions,wedges} that the following is true.  If one considers an independent $\CLE_\kappa$ for $\kappa \in (4,8)$ on top of a $(\gamma=4/\sqrt{\kappa})$-LQG sphere marked by the points $x,y$ and explores the loops which separate $x$ and $y$ from $x$ towards $y$, then the quantum surface parameterized by the component which contains $y$ is that of a quantum disk (weighted by its quantum area).  Since the definition of $\CLE_\kappa$ on the sphere \emph{a priori} depends on the choice of a marked point (in this case $x$), it is not obvious that if one explores these same loops in the reverse direction (i.e., from $y$ to $x$), then the quantum surface parameterized by the component which contains $x$ is also a quantum disk (weighted by its quantum area).  Theorem~\ref{thm:main_result}, however, supplies the missing symmetry to deduce this statement.

At a first glance, one might guess that Theorem~\ref{thm:main_result} follows from the branching $\SLE_\kappa(\kappa-6)$ construction and the reversibility of whole-plane $\SLE_\kappa(\kappa-6)$ for $\kappa \in (4,8)$ established in \cite{ig4}. However, this is not the case since the whole-plane $\CLE_\kappa$ which is associated with a whole-plane $\SLE_\kappa(\kappa-6)$ is \emph{not the same} as the whole-plane $\CLE_\kappa$ associated with the time-reversal of the whole-plane $\SLE_\kappa(\kappa-6)$.  We will explain this point in more detail in see Section~\ref{sec-branching}.

\subsection*{Overview of proof strategy: inverting CLE in an annulus}

The basic idea of the proof of Theorem~\ref{thm:main_result} is as follows.  Suppose that $\Gamma$ is a whole-plane $\CLE_\kappa$ and that $\{\gamma_n\}_{n\in\BB Z}$ is the sequence of loops of $\Gamma$ which surround $0$, numbered from outside in, where a loop is said to surround $0$ if its winding number around $0$ is non-zero.  We can choose the normalization of the indices so that $\gamma_0$ is the largest loop which intersects $\ol{\BB D}$.  For each $n \in \BB Z$, let $D_n$ be the connected component of $\BB C \setminus \mcl \gamma_n$ which contains $0$.  We will then fix $M \in \BB N$ and let $A_{M}$ be the annular connected component of $D_0 \setminus \gamma_{M+1}$. The main step of the proof is to show that the conditional law of the restriction of $\Gamma$ to $A_M$, given $A_M$, is invariant under the inversion map of $A_M$. Note that we already know that the restriction of $\Gamma$ to $D_0$ has the law of a $\CLE_\kappa$ on $D_0$ (see Lemma~\ref{lem-cle-markov} below) so by conformal invariance this can be thought of as a problem about $\CLE_\kappa$ on the disk. 
 
To accomplish this, we will show that $\CLE_\kappa$ on an annulus with any fixed number  of inner-boundary-surrounding loops is well-defined and conformally invariant (including inversion invariant) for\footnote{$\CLE_\kappa$ on an annulus for $\kappa \in (8/3,4]$ with no inner-boundary-surrounding loops is constructed in~\cite{sww-cle-doubly} using the Brownian loop soup. We learned from Wendelin Werner [private communication] that one can deduce from the results of~\cite{werner-sphere-cle} that also $\CLE_\kappa$ on an annulus with any fixed number of inner-boundary-surrounding loops is well-defined and conformally invariant for $\kappa \in (8/3,4]$.}
 $\kappa \in (4,8)$, and that the law of the restriction of $\Gamma$ to $A_M$ is that of a $\CLE_\kappa$ on $A_M$. 

For $\rho \in (0,1)$, we define the open annulus
\eqb \label{eqn-annulus-def}
\BB A_\rho := \BB D\setminus \ol{B_\rho(0)}, \quad \forall \rho \in (0,1). 
\eqe
The following theorem gives a way to define $\CLE_\kappa$ on $\BB A_\rho$ with a specified number of loops which surround the inner boundary.

\begin{thm}[$\CLE_\kappa$ on an annulus] \label{thm-cle-annulus0}
Let $\kappa \in (4,8)$ and $M\in\BB N_0$. 
Let $\Gamma$ be a $\CLE_\kappa$ on $\BB D$ and let $\gamma_{M+1}$ be the $(M+1)$st outermost loop in $\Gamma$ surrounding 0. 
On the event $\{\gamma_{M+1} \cap \bdy\BB D = \emptyset\}$ (which has probability 1 if $M \geq 1$), let $\mcl A_M$ be the non-simply connected component of $\BB D\setminus \gamma_{M+1}$ and let $f_M : \mcl A_M \rta \BB A_\rho$ for some $\rho > 0$ be the conformal map which fixes 1 (note that $\rho$ is random and determined by $\gamma_{M+1}$). 
Let $\Gamma_{\BB A_\rho}$ be the image under $f_M$ of the restriction of $\Gamma$ to $\mcl A_M$. 
Almost surely, the conditional law of $\Gamma_{\BB A_\rho}$ given $\gamma_{M+1}$ depends only on $\rho$, and this conditional law is invariant under rotations of $\BB A_\rho$ and under the inversion map $z\mapsto \rho/z$.
\end{thm}

In the setting of Theorem~\ref{thm-cle-annulus0}, it is easily seen that the support of the law of $\rho$ is all of $(0,1)$, and the conditional law of $\Gamma_{\BB A_\rho}$ depends continuously on $\rho$, which allows us to define this conditional law for each fixed $\rho\in (0,1)$.  We call a loop ensemble sampled according to this law \emph{$\CLE_\kappa$ on $\BB A_\rho$ with $M$ inner-boundary-surrounding loops}.
It is an interesting open problem to determine the law of the conformal modulus $\rho$ in the setting of Theorem~\ref{thm-cle-annulus0}.
  
We also remark that, for $\kappa\in(4,8)$, since the loops in $\CLE_\kappa$ are non-simple, a $\CLE_\kappa$ on $\BB A_\rho$ with $M$ inner-boundary-surrounding loops is allowed to have more than $M$ loops which disconnect the inner and outer boundaries. 

In Section~\ref{sec-cle-annulus}, we will give an alternative definition of $\CLE_\kappa$ on $\BB A_\rho$ with $M$ inner-boundary-surrounding loops in terms of the so-called \emph{annulus Markov property}, which is analogous to the domain Markov property of $\SLE_\kappa$ (see Definition~\ref{def-annulus-markov}).

The main steps in the proof of Theorem~\ref{thm:main_result} consist of proving that (a) the loop ensemble $\Gamma_{\BB A_\rho}$ described in Theorem~\ref{thm-cle-annulus0} satisfies the annulus Markov property and (b) there is at most one law on loop ensembles on $\BB A_\rho$ which satisfies this Markov property.
Since both $\Gamma_{\BB A_\rho}$ and its image under inversion satisfy the annulus Markov property, their laws must be the same. This implies that the law of the whole-plane $\CLE_\kappa$ restricted to $A_M$ is invariant under the inversion map.  We will then deduce the inversion invariance of the whole-plane $\CLE_\kappa$ by looking at its restriction to annular regions that tend to the whole-plane.
 
The proof that $\Gamma_{\BB A_\rho}$ satisfies the annulus Markov property is given in Section~\ref{sec-markov}, building on the basic Markov property for $\CLE_\kappa$ established in~\cite{shef-cle}. 
The proof of the uniqueness statement is given in Section~\ref{sec-resampling} using re-sampling arguments similar to those used to prove various reversibility and uniqueness statements for SLE in~\cite{ig2,ig4,msw-sle-range}. 
Unlike the arguments of~\cite{ig2,ig4,msw-sle-range}, however, we will \emph{not} directly use the Gaussian free field (although various results from~\cite{ig3,ig4} are implicitly used in our arguments since they are needed to show that $\CLE_\kappa$ is well-defined). 

Appendix~\ref{sec-sle-cle-lemmas} contains the proofs of several basic facts about SLE and CLE which are used elsewhere in the paper and are collected here to avoid interrupting the main argument.

\medskip

\noindent\textbf{Acknowledgements.} 
We thank an anonymous referee for helpful comments on an earlier version of this paper.
We thank Wendelin Werner for helpful discussions.
EG was supported by a Herchel Smith fellowship and a Trinity College junior research fellowship.
WQ acknowledges the support of an Early Postdoc Mobility grant of the SNF, EPSRC grant EP/L018896/1 and a JRF of Churchill college.

\section{Preliminaries}
\label{sec-prelim}

We first introduce some basic notation and terminology in Sections~\ref{sec-basic} and~\ref{sec-loop-prelim}.
In Section~\ref{sec-branching}, we review the construction of whole-plane $\CLE_\kappa$ via branching $\SLE_\kappa(\kappa-6)$.
In Section~\ref{sec-cle-annulus}, we state the Markov property which characterizes $\CLE_\kappa$ on an annulus and state a more precise version of Theorem~\ref{thm-cle-annulus0}.
In Appendix~\ref{sec-sle-cle-lemmas}, we record some elementary lemmas for SLE and CLE. 

\subsection{Basic notation}
\label{sec-basic}

\noindent
We write $\BB N$ for the set of positive integers and $\BB N_0 = \BB N\cup \{0\}$. 
\vspace{6pt}

\noindent
For $a,b \in \BB R$ with $a<b$, we define the discrete interval $[a,b]_{ \BB Z} := [a, b]\cap \BB Z $.
\vspace{6pt}
    
\noindent
For a collection $\mcl A$ of subsets of $\BB C$ (which will typically be loops) we write $\bigcup \mcl A$ for the union of the elements of $\mcl A$. 
\vspace{6pt}

\subsection{Basic definitions for loops and loop configurations}
\label{sec-loop-prelim}

In this subsection, we will define loops and loop configurations as well as some basic properties thereof.
We will also define complete separable metrics on the space of loops and on the space of locally finite loop configurations.
Most of the definitions in this subsection are standard, so the reader may want to skim it. 

\subsubsection{Loops}
\label{sec-loop-def}

A \emph{parameterized loop} is a continuous function $\wh\gamma : \bdy\BB D  \rta \BB C$. 
A \emph{loop} is an equivalence class $\gamma$ of parameterized loops, with two parameterized loops declared to be equivalent if they differ by pre-composition with an orientation-preserving homeomorphism $\bdy\BB D\rta\bdy\BB D$.
A \emph{parameterization} of $\gamma$ is a choice of equivalence class representative.
We define a metric on the set of loops $\gamma$ in $\BB C$ by
\eqb \label{eqn-loop-metric}
\BB d^{\op{Loop}}(\gamma_1,\gamma_2) := \inf_{\wh\gamma_1 ,\wh\gamma_2} \sup_{u\in \bdy\BB D} |\wh\gamma_1(u) - \wh\gamma_2(u)|,
\eqe
where the infimum is over all choices of parameterizations $\wh\gamma_1$ for $\gamma_1$ and $\wh\gamma_2$ for $\gamma_2$. This defines a complete metric on the space of loops in $\BB C$ (see, e.g.,~\cite[Lemma 2.1]{ab-random-curves}, which treats the case of curves). It is also easily seen that the space of parameterized loops is separable with respect to $\BB d^{\op{Loop}}$. 
 
\begin{defn} 
An \emph{arc} of a loop $\gamma$ is a curve (viewed modulo increasing re-parameterization of time) which admits a parameterization of the form $\alpha(t) = \wh\gamma(e^{i t})$ for $t\in [a,b]$, where $[a,b] \subset [0,2\pi]$ is a non-trivial interval and $\wh\gamma$ is a parameterization of $\gamma$.
We say that an arc $\alpha$ is \emph{proper} if it is not all of $\gamma$.
\end{defn}

\subsubsection{Loop configurations}
\label{sec-loop-config}

A \emph{loop configuration} on a domain $D\subset \BB C$ is a countable multiset $\Gamma$ of loops which are each contained in $\ol D$ (we say ``multiset" instead of ``set" since we need to allow multiple copies of the same loop to make our metric on loop configurations complete). 
For $A\subset \ol D$, we write
\eqb \label{eqn-loop-restrict}
\Gamma|_A := \left\{\gamma \in \Gamma : \gamma \subset \ol A \right\} \quad \op{and}\quad 
\Gamma(A) := \left\{\gamma \in \Gamma : \gamma \cap \ol A \not=\emptyset \right\} .
\eqe

\begin{defn} \label{def-locally-finite}
A loop configuration $\Gamma$ is called \emph{locally finite} if for each $\ep >0$ and each compact set $A\subset\ol{D}$, the number of loops in $\Gamma$ of Euclidean diameter greater than $\ep$ which intersect $A$ is finite. 
\end{defn}

We will now define a metric on the space of locally finite loop configurations on $D$ whereby, roughly speaking, two loop configurations are close if their large loops can be ``matched up" in such a way that the corresponding loops are close with respect to $\BB d^{\op{Loop}}$.  
We need to be somewhat careful about the definition since we want to ensure that our metric is complete (see Lemma~\ref{lem-lc-metric} below).
This prevents us from using, e.g., the $\BB d^{\op{Loop}}$-Hausdorff distance on discrete subsets of the space of loops as in~\cite{shef-cle} since a sequence of discrete sets of loops can converge to a non-discrete set of loops with respect to this metric. 

We first define our metric on finite loop configurations.  
If $\Gamma^1,\Gamma^2$ are two such loop configurations, we define $\BB d^{\op{LC}}(\Gamma^1,\Gamma^2)$ to be 1 if $\#\Gamma^1 \not=\#\Gamma^2$ and otherwise we define
\eqb \label{eqn-lc-metric0}
\BB d^{\op{LC}}(\Gamma^1,\Gamma^2) := 1\wedge\left\{ \min_{\psi : \Gamma^1\rta\Gamma^2} \sum_{\gamma \in \Gamma^1} \BB d^{\op{Loop}}(\gamma , \psi(\gamma) ) \right\}  
\eqe
where the minimum is over all bijections $\psi : \Gamma^1 \rta \Gamma^2$.

We next consider the case when our domain $\ol D$ is compact. 
For $\ep  > 0$ and a locally finite loop configuration $\Gamma$ on $D$, we write $\Gamma_\ep$ for the (multi)set of loops in $\Gamma$ which have diameter greater than $\ep $ (which is finite).  
We then define
\eqb \label{eqn-lc-metric}
\BB d_D^{\op{LC}}(\Gamma^1,\Gamma^2) := \int_0^1    \BB d^{\op{LC}}(\Gamma_\ep^1,\Gamma_\ep^2)  \,d\ep  ,
\eqe 
so that for loop configurations $\{\Gamma^n\}_{n\in\BB N}$ and $\Gamma$, we have $\BB d_D^{\op{LC}}(\Gamma^n,\Gamma ) \rta 0$ if and only if the $\BB d^{\op{LC}}$-distance between $\Gamma_\ep^n$ and $\Gamma_\ep$ tends to zero for Lebesgue-a.e.\ $\ep  > 0$. 

Finally, if $\ol D$ is not necessarily compact, we define the \emph{localized loop configuration metric} (using the notation~\eqref{eqn-loop-restrict}) by
\eqb \label{eqn-lc-metric-loc}
\BB d_D^{\op{LC},\op{loc}}(\Gamma^1,\Gamma^2) := \int_1^\infty e^{-R}   \BB d_{B_R(0)}^{\op{LC}}\left(\Gamma^1(B_R(0)) ,\Gamma^2(B_R(0)) \right) \,dR .
\eqe
Henceforth, whenever we talk about a random loop configuration on $D$ we will use the Borel $\sigma$-algebra with respect to $\BB d_D^{\op{LC}}$ if $\ol D$ is compact or the Borel $\sigma$-algebra with respect to $\BB d_D^{\op{LC},\op{loc}}$ if $\ol D$ is not compact.

\begin{lem} \label{lem-lc-metric}
For any domain $D\subset\BB C$, the metric $\BB d_D^{\op{LC}}$ defined just above (or the metric $\BB d_D^{\op{LC},\op{loc}}$ in the case when $\ol D$ is not compact) is complete and separable on the space of locally finite loop configurations.
\end{lem}
\begin{proof}
Trivially, the space of all loops is separable with respect to the metric~\eqref{eqn-loop-metric} and the space of finite loop configurations on $D$ is dense in the space of all loop configurations on $D$ with respect to the metric~\eqref{eqn-lc-metric} or~\eqref{eqn-lc-metric-loc}. This gives separability.

To check completeness, consider a Cauchy sequence of loop configurations $\{\Gamma^n\}_{n\in\BB N}$. 
First assume that all of the loop configurations $\Gamma^n$ are finite, with the same cardinality $N$.
We will prove convergence with respect to the metric~\eqref{eqn-lc-metric0} by induction on $N$. 
The case $N=1$ just follows from the completeness of the metric~\eqref{eqn-loop-metric}.

Now suppose that $N\geq 2$ and we have proven the convergence of all Cauchy sequences of loop configurations which all have $N'\leq N-1$ loops. 
To prove the convergence of $\{\Gamma^n\}_{n\in\BB N}$, it suffices to show convergence along a subsequence. 
By the definition of $\BB d^{\op{LC}}$~\eqref{eqn-lc-metric}, after possibly passing to a subsequence we can arrange that for each $n\in\BB N$, there is a bijection $\psi_n : \Gamma^n\rta \Gamma^{n+1}$ such that 
\eqb \label{eqn-cauchy-loop-sum}
\sum_{\gamma \in \Gamma^n} \BB d^{\op{Loop}}(\gamma , \psi_n(\gamma) ) \leq 2^{-n} . 
\eqe
Now fix a loop $\gamma^1 \in \Gamma^1$ and for $n\geq 2$, let $\gamma^n := (\psi_{n-1}\circ\psi_{n-2}\circ\dots\circ\psi_1)(\gamma^1)  \in \Gamma^n$.
By~\eqref{eqn-cauchy-loop-sum}, the sequence of loops $\{\gamma^n\}_{n\in\BB N}$ is Cauchy with respect to the metric~\eqref{eqn-loop-metric}, so converges to a limiting loop $\gamma$. On the other hand, the loop configurations $\Gamma^n\setminus \{\gamma^n\}$ each have $N-1$ loops and are Cauchy with respect to~\eqref{eqn-lc-metric0}. 
Combining these statements with the inductive hypothesis concludes the proof in the case of $N$ loops. 

We now assume that our loop configurations $\Gamma^n$ are all on a domain $D\subset \BB C$ with $\ol D$ compact and prove convergence with respect to the metric~\eqref{eqn-lc-metric}. 
For each $\ep > 0$, the sequence of finite loop ensembles $\{\Gamma_\ep^n\}_{n\in\BB N}$ (as defined in~\eqref{eqn-lc-metric}) is Cauchy with respect to the metric~\eqref{eqn-lc-metric0}.  
By the definition of $\BB d^{\op{LC}}$ in~\eqref{eqn-lc-metric0}, for each $\ep > 0$ and each large enough $n,m\in\BB N$, we have $\# \Gamma_\ep^n = \#\Gamma_\ep^m$.
Therefore, the case of finite loop configurations shows that there is a limiting loop configuration $\Gamma_\ep'$ such that $\Gamma_\ep^n\rta \Gamma_\ep'$. 
For Lebesgue-a.e.\ $\wt\ep \leq \ep$ (i.e., every such pair of $\wt\ep,\ep$ for which there is not a loop of $\Gamma_{\wt\ep}'$ or $\Gamma_\ep'$ of diameter exactly $\wt\ep$ or $\ep$), the set of the loops in $\Gamma_{\wt\ep}'$ with diameter greater than $\ep$ coincides with $\Gamma_\ep'$. Therefore, there is a unique loop configuration $\Gamma$ with $\Gamma_\ep = \Gamma_\ep'$ for each $\ep > 0$ and $\Gamma^n \rta \Gamma$.

The case when $\ol D$ is not compact and we work with the local metric~\eqref{eqn-lc-metric-loc} is treated similarly. 
\end{proof}

In addition to local finiteness, the other important condition which we will typically impose on our loop configurations is that the loops do not cross or trace themselves or each other, in a rather strong sense. Let us first define the condition for a single loop.

\begin{defn} \label{def-non-crossing}  
A loop $\gamma$ in $D\subset\BB C$ is \emph{non-crossing} if for each arc $\alpha$ of $\gamma$, the following is true.
\begin{itemize}
\item $\alpha$ does not trace the complementary arc $\wt\alpha$ of $\alpha$ in $\gamma$ for any non-trivial interval of time.   
\item $\alpha$ is contained in the closure of a single connected component $U$ of $  D \setminus  \wt\alpha $.
\item If $f : U\cup \bdy U \rta \ol{\BB D}$ is a conformal map (with $\bdy U$ viewed as a collection of prime ends), then $f(\alpha)$ is a continuous curve. 
\end{itemize}    
\end{defn}

See Figure~\ref{fig-non-crossing} for an example of a non-crossing loops and three examples of loops which are not non-crossing. 

\begin{figure}[t!]
 \begin{center}
\includegraphics[scale=1]{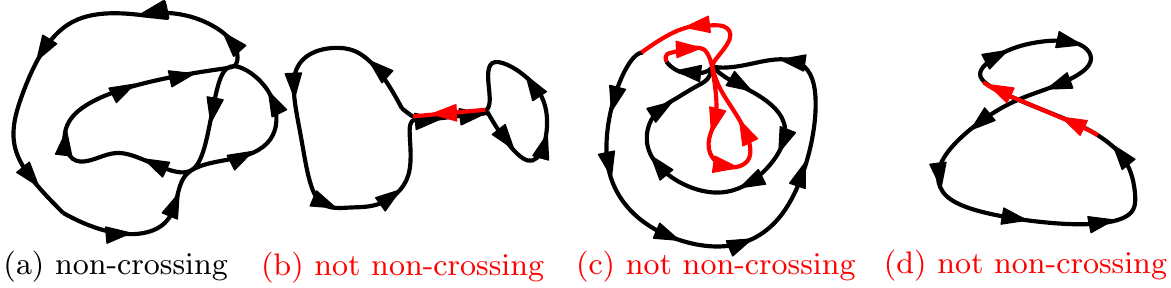}
\vspace{-0.01\textheight}
\caption{A non-crossing loop and three loops which are not non-crossing. For loops (b)-(d), an arc $\alpha$ which violates one of the conditions in Definition~\ref{def-non-crossing} is shown in red. In particular, loops (b), (c), and (d), resp., violate the first, second, and third conditions in the definition. 
}\label{fig-non-crossing}
\end{center}
\vspace{-1em}
\end{figure}

\begin{defn} \label{def-non-crossing-ensemble}  
A loop ensemble $\Gamma$ is \emph{non-crossing} if for any finite collection $\gamma_0,\gamma_1,\dots,\gamma_n$ of loops in $\Gamma$ and any arc $\alpha_0$ of $\gamma_0$, the following is true. 
\begin{itemize}
\item If we let $\wt\alpha_0$ be the complementary arc of $\alpha_0$ in $\gamma_0$, then $\alpha_0$ does not trace the set $\ol{\wt\alpha_0 \cup \gamma_1\cup\dots\cup\gamma_n }$ for any non-trivial interval of time.
\item $\alpha_0$ is contained in the closure of a single connected component $U$ of $  D \setminus \ol{\wt\alpha_0 \cup \gamma_1\cup\dots\cup\gamma_n }$.
\item If $f : U\cup \bdy U \rta \ol{\BB D}$ is a conformal map (with $\bdy U$ viewed as a collection of prime ends), then $f(\alpha_0)$ is a continuous curve. 
\end{itemize}    
\end{defn}

Each of the loops in a non-crossing loop ensemble is non-crossing, as can be seen by applying the definition in the case of a single loop. 
Furthermore, the loops in a non-crossing loop ensemble are necessarily distinct (so such a loop ensemble is a set, non a multi-set): indeed, if $\gamma_0 = \gamma_1$ then the first condition fails for any arc $\alpha_0$ of $\gamma_0$. 
$\CLE_\kappa$ for $\kappa \in (4,8)$ on a domain bounded by a curve is a.s.\ non-crossing since each arc of a $\CLE_\kappa$ loop is an $\SLE_\kappa$-type curve and the loops do not cross or trace each other.

\subsection{Construction of whole-plane CLE using branching SLE}
\label{sec-branching}
 
Sheffield~\cite{shef-cle} constructed $\op{CLE}_\kappa$ for each $\kappa \in (8/3,8)$ using a branching $\op{SLE}_\kappa(\kappa-6)$ process in a proper simply connected subdomain of $\BB C$. Here we will describe the analogous construction for whole-plane $\op{CLE}_{\kappa}$ for $\kappa \in (4,8)$. Throughout, we assume that $\kappa \in (4,8)$ is fixed.

\begin{remark}\label{rmk:whole_plane_CLE}
\cite[Appendix A]{mww-nesting} gives a different definition of whole-plane $\CLE_\kappa$, based on taking limits of CLE on large domains.  It is easy to see using the Markov property of whole-plane $\CLE_\kappa$ (Lemma~\ref{lem-cle-markov}) that this construction gives the same object as our construction. 
\end{remark}

\subsubsection{Whole-plane branching $\SLE_\kappa(\kappa-6)$}

Let us first recall the definition of whole-plane $\SLE_\kappa(\kappa-6)$ from~\cite[Section~2.1]{ig4}.
 Whole-plane $\SLE_\kappa(\kappa-6)$ from 0 to $\infty$ is the curve $\eta$ generated by the whole-plane Loewner evolution with driving process $W$, where $(W,O) : \BB R\rta \bdy\BB D\times\bdy\BB D$ is the unique stationary solution to the following SDE:  
\eqb \label{eqn-sde}
dW_t = -\frac{\kappa}{2} W_t \, dt + i \sqrt\kappa \, dB_t + \left( \frac{\kappa-6}{2} \right) \frac{O_t + W_t}{O_t-W_t} W_t \,dt \quad\text{and}\quad
dO_t =  \frac{O_t + W_t}{ W_t - O_t} O_t \,dt .
\eqe
More precisely, if we let $g_t$ be the conformal map from the unbounded connected component of $\BB C\setminus \eta([0,t])$ onto $\BB C\setminus \ol{\BB D}$ such that $g_t(z)=e^{-t}z+o_{|z|}(|z|)$ as $z\to \infty$, then $W_t=g_t(\eta(t))$ and for Lebesgue-a.e.\ $t\in\BB R$, $O_t$ is the image under $g_t$ of the unique point on the outer boundary of $ \eta([0,t])$ other than $\eta(t)$ at which the left and right outer boundaries of $\eta([0,t])$ meet. 
The existence and uniqueness of this solution is proven in~\cite[Proposition 2.1]{ig4}. 

For distinct $z,w\in \BB C\cup \{\infty\}$, whole-plane $\SLE_\kappa(\kappa-6)$ from $z$ to $w$ is defined to be the image of whole-plane $\SLE_\kappa$ from 0 to $\infty$ under a M\"obius transformation taking 0 to $z$ and $\infty$ to $w$. 
We will typically consider whole-plane $\SLE_\kappa(\kappa-6)$ started from $\infty$. 

By the Schramm-Wilson coordinate change formula~\cite[Theorem 3]{sw-coord}, whole-plane $\SLE_\kappa(\kappa-6)$ is target invariant in the sense that the law of whole-plane $\SLE_\kappa(\kappa-6)$ from $\infty$ to $z$ and from $\infty$ to $w$ agree up until the first time that the curve separates $z$ from $w$. 
This allows us to find a coupling $\{\eta_z\}_{z\in\BB Q^2}$ where each $\eta_z$ is a whole-plane $\SLE_\kappa(\kappa-6)$ from $\infty$ to $z$ and for $z,w\in\BB Q^2$, the curves $\eta_z$ and $\eta_w$ agree, modulo time parameterization, until the first time that $z$ and $w$ lie in different complementary connected components of the curve and evolve in a conditionally independent manner thereafter. We call $\{\eta_z\}_{z\in\BB Q^2}$ the \emph{whole-plane branching $\SLE_{\kappa}(\kappa-6)$ process}.

\subsubsection{Construction of whole-plane $\CLE_\kappa$}

Now let $\{\eta_z\}_{z\in\BB Q^2}$ be a branching $\SLE_\kappa(\kappa-6)$ process started from $\infty$, where for all $z\in\BB Q^2$, $ \eta_z $ is the branch from $\infty$ to $z$. If we apply a M\"obius transformation that sends $0, \infty$ to $\infty, z$, then the image of $\eta_z$ is a whole-plane $\SLE_\kappa(\kappa-6)$ from $0$ to $\infty$, generated by a Loewner driving pair $(W^{z } , O^{z }) : \BB R\rta \BB R^2$, as in~\eqref{eqn-sde}.  
Let $\theta^z $ be the continuous version of $ \op{arg} W^z - \op{arg} O^z$. 
For $z\in \BB Q^2$ let $\wt{\mcl T}_z$ be the set of times $t \in \BB R$ such that the following is true. We have $\theta^z_t \in 2\pi \BB Z$ and the last time $s < t$ such that $\theta_s^z\in 2\pi \BB Z$ satisfies $\theta_s^z \not=\theta_t^z$.  Since $\theta^z$ is continuous, the set $\wt{\mcl T}_z$ is discrete, so we can write $\wt{\mcl T}_z = \{t_{z,i}  \}_{i\in \wt{\mcl N}_z}$, where $\wt{\mcl N}_z$ is the intersection of $\BB Z$ with an interval in $\BB R$ (possibly empty or all of $\BB R$), and the enumeration is chosen so that $t_{z,i}  < t_{z,i+1} $ for each $i \in \wt{\mcl N}_z$.  
  
\begin{lem} \label{lem-branching-law}
Let $z\in\BB Q^2$. Almost surely, we have $\wt{\mcl N}_z = \BB Z$. Furthermore, the law of each $\eta_z|_{[t_{z,i}  , t_{z,i+1} ]}$ is that of a radial $\op{SLE}_{\kappa}(\kappa-6)$ process from $\eta_z(t_{z,i} )$ to $z$ in the connected component of $\BB C\setminus \eta_z((-\infty, t_{z,i} ])$ containing $z$, stopped at the first time it disconnects the boundary of this component from $z$. If $\theta^z_{t_{z,i} } - \theta^z_{t_{z,i-1} } = 2\pi$, the force point is located to the right of $\eta_z(t_{z,i} )$, and if $\theta^z_{t_{z,i} } - \theta^z_{t_{z,i-1} } = -2\pi$, the force point is located to the left of $\eta_z(t_{z,i} )$. 
\end{lem}
\begin{proof}
By harmonic measure considerations, $t_{z,i} $ is the first time at which $\eta_z$ disconnects $\eta_z ( (-\infty, t_{z,i-1} ])$ from $z$. 
The lemma is immediate from this together with the Markov property of whole plane $\op{SLE}_{\kappa}(\kappa-6)$~\cite[Proposition 2.2]{ig4}.  
\end{proof}  

Let $\{\tau_{z,j} \}_{j\in \BB Z}$ be the times $t_{z,i} \in \wt{\mcl T}_z$ such that $\theta^z_{t_{z,i} } - \theta^z_{t_{z,i-1} } = 2\pi $, enumerated in increasing order.

We define a sequence of loops $\{\gamma_{z,j}\}_{j\in\BB Z}$ surrounding $z$ (enumerated from outside in) as follows.
For each $j \in \BB N$, let $\sigma_{z,j}$ be the last time $s < \tau_{z,j}$ such that $\theta_s^z \in 2\pi \BB Z$, so that by the definition of $\mcl T_z$ we have $\theta^z_{\tau_{z,j} } - \theta^z_{\sigma_{z,j}} = 2\pi$. 

The curve $\eta_z|_{[\sigma_{z,j} , \tau_{z,j}]}$ traces part (but not all) of a loop surrounding $z$. 
To describe the rest of this loop, we let $\wh\eta_{z,j}$ be the branch of the branching $\SLE_\kappa(\kappa-6)$ process from $\eta_z(\tau_{z,j})$ to $\eta_z(\sigma_{z,j})$.
This process can be described as the limit of the segment of $\eta_w$ from $\eta_z(\tau_{z,j})$ to $w$ as $w\rta \eta_z(\sigma_{z,j})$ along sequences of rational points in the connected component of $\BB C\setminus \eta_z((-\infty,\tau_{z,j}])$ with $\eta_z(\sigma_{z,j})$ and $\eta_z(\tau_{z,j})$ on its boundary. Its conditional law given $\eta_z|_{(-\infty,\tau_{z,j}]}$ is that of an $\SLE_\kappa$ in the appropriate connected component of $\BB C\setminus \eta_z((-\infty,\tau_{z,j}])$. 

Let $\gamma_{z,j}$ be the loop obtained by concatenating the curves $\eta_z|_{[\sigma_{z,j} ,\tau_{z,j}]}$ and $\wh\eta_{z,j}$. 
We define the whole-plane $\op{CLE}_{\kappa}$ by
\eqb \label{eqn-Gamma-def}
\Gamma := \{\gamma_{z,j} : z\in \BB Q^2 ,\, j\in \BB Z\}  .
\eqe 
Then $\Gamma$ is a non-crossing, locally finite collection of loops in $\BB C$ (Definitions~\ref{def-locally-finite} and~\ref{def-non-crossing}).
Indeed, the fact that $\Gamma$ is non-crossing follows from the fact that the curves $\eta_z$ for $z\in\BB Q^2$ do not cross or trace themselves or each other.
The fact that $\Gamma$ is locally finite follows from the local finiteness of $\CLE_\kappa$ on a bounded Jordan domain~\cite[Theorem 1.17]{ig4} and the Markov property of whole-plane $\CLE_\kappa$ which is stated and proven just below.

\begin{figure}[ht!] 
 \begin{center}
\includegraphics[scale=.75]{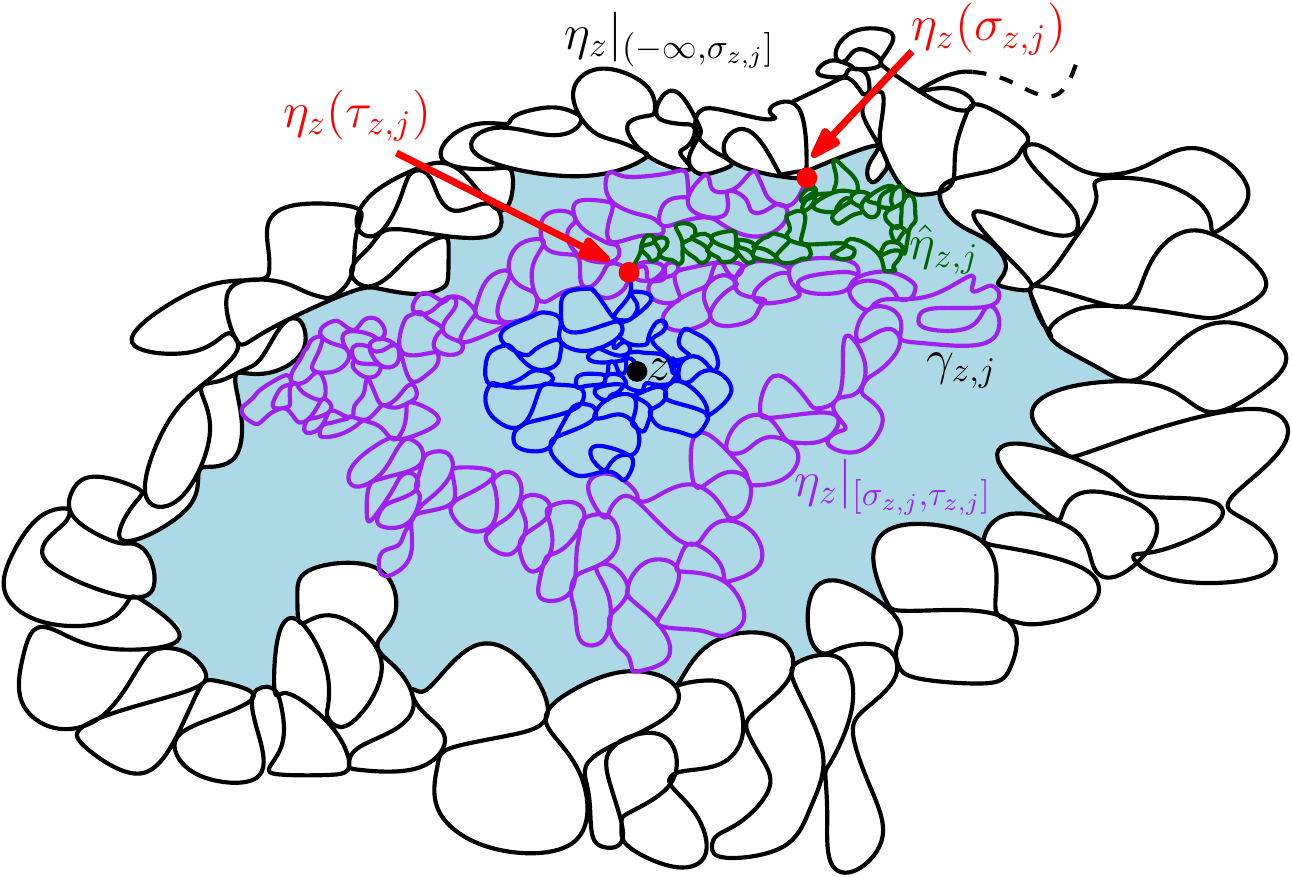} 
\caption{Illustration of the construction of whole-plane $\CLE_{\kappa}$ from branching $\SLE_{\kappa}(\kappa-6)$. Shown is the branch $\eta_z$ targeted at $z$ for some $z\in\BB C$ and a single loop $\gamma_{z,j}$ which it traces part of. The curve $\eta_z $ traces the black segment, then the purple segment, then the blue segment (it does not trace the green segment). The loop $\gamma_{z,j}$ is the concatenation of the purple segment $\eta_z|_{[\sigma_{z,j},\tau_{z,j}]}$ and the green segment $\wh\eta_{z,j}$. The loop $\gamma_{z,j+1}$ (not shown) is the concatenation of part of the blue segment of $\eta_z$ plus an additional curve segment not traced by $\eta_z$.}\label{fig-cle-def}

\end{center}
\end{figure}

\begin{remark} \label{remark-branching-extend}
The curves $\eta_z$ above are defined only for points $z$ in a countable dense subset of $\BB C$. However, one can define $\eta_z$ as a continuous curve for Lebesgue-a.e.\ point $z\in\BB C$ as follows. Suppose $z\in \BB C$ is surrounded by arbitrarily small loops in $\Gamma$ (which is a.s.\ the case for fixed $z$). Let $\{\gamma_{z,j}\}_{j\in\BB Z}$ be the bi-infinite sequence of loops surrounding $z$, numbered from outside in. For $j\in \BB Z$, let $z_j \in \BB Q^2$ such that $z_j$ lies in the same connected component of $\BB C\setminus \gamma_{z,j}$ as $z$. Let $s_j$ be the time at which $\eta_{z_j}$ finishes tracing the boundary of this complementary connected component. For each $j < j'$, the curve $\eta_{z_j}$ agrees with $\eta_{z_{j'}}$ until time $s_j$. Furthermore, the diameters of the loops $\gamma_{z,j}$ tend to zero as $j\rta\infty$. It follows that the curves $\eta_{z_j}$ converge to a limiting curve $\eta_z$ as $j\rta\infty$, which agrees with each $\eta_{z_j}$ until time $s_j$, viewed modulo monotone re-parameterization.
\end{remark}

\subsubsection{Markov property of whole-plane $\CLE_\kappa$}

\begin{lem} \label{lem-cle-markov}
Let $\Gamma$ be a whole-plane $\CLE_\kappa$. 
Also let $z\in \BB C$ and let $\gamma_z^*$ be a random loop in $\Gamma$ which surrounds $z$ with the following property. If $\eta_z$ is the branch targeted at $z$ of the branching $\SLE_{\kappa }(\kappa -6)$ process which traces the loops in $\Gamma$, then the time $\tau_z^*$ at which $\eta_z$ finishes tracing the part of $\gamma_z^*$ that it should trace is a stopping time for $\eta_z$. 
If we condition on $\gamma_z^*$ and the set of loops in $\Gamma$ which are contained in the unbounded connected component of $\BB C\setminus \gamma_z^*$, then the conditional law of the rest of $\Gamma$ is that of an independent $\CLE_{\kappa }$ in each bounded connected component of $\BB C\setminus \gamma_z^*$. 
\end{lem}
\begin{proof}
Let $\wh\eta_z^* $ be the segment of $\gamma_z^*$ which is not traced by $\eta_z$, as above. 
As explained just above~\eqref{eqn-Gamma-def}, the conditional law of $\wh\eta_z^*$ given $\eta_z((-\infty,\tau_z^*])$ is that of a chordal $\SLE_{\kappa }$ in the appropriate connected component of $\BB C\setminus \eta_z((-\infty,\tau_z^*])$. 
The loop $\gamma_*^z  $ is contained in $\eta_z((-\infty,\tau_z^*]) \cup \wh\eta_z^*$ and every bounded connected component of $\BB C\setminus \gamma_z^*$ is also a connected component of $\BB C\setminus (\eta_z((-\infty,\tau_z^*]) \cup \wh\eta_z^*)$ whose boundary is entirely traced by either the left boundaries of $\eta_z$ and $\wh\eta_z^*$ or the right boundaries of these two curves. 
By the renewal property of branching $\SLE_{\kappa }(\kappa -6)$ (which follows from Lemma~\ref{lem-branching-law} applied to the branches targeted at points in the components) and the construction of $\CLE_\kappa$ on a proper simply connected domain from branching $\SLE_\kappa(\kappa-6)$~\cite{shef-cle}, the conditional law given $\eta_z|_{(-\infty,\tau_z^*]}$ and $ \wh\eta_z^*$ of the set of loops in $\Gamma$ which are contained in the bounded connected components of $\BB C\setminus \gamma_z^*$ is that of an independent $\CLE_\kappa$ in each of these components. 
Furthermore, the set of loops of $\Gamma$ which are contained in the unbounded connected component of $\BB C\setminus \gamma_z^*$ is determined by $\eta_z|_{(-\infty,\tau_*^z]}$, $\wh\eta_z^*$, and the radial branching $\SLE_{\kappa }(\kappa -6)$ processes in the connected components of $\BB C\setminus (\eta_z((-\infty,\tau_z^*]) \cup \wh\eta_z^*)$ which are not bounded connected components of $\BB C\setminus \gamma_z^*$. Since these processes are conditionally independent given $\eta_z|_{(-\infty,\tau_*^z]}$, $\wh\eta_z^*$ from the set of loops of $\Gamma$ which are contained in the unbounded connected component of $\BB C\setminus \gamma_z^*$, we get the statement of the lemma.
\end{proof}

\subsubsection{CLE$_\kappa$ loops intersecting a set}

At several places in the paper, we will need the following basic property of CLE$_\kappa$.

\begin{lem} \label{lem-cle-cover}
Let $\kappa \in (4,8)$, let $D \subset\BB C$ be simply connected, and let $\Gamma$ be a $\CLE_\kappa$ on $D$. 
Suppose $U \subset D$ is open and $K\subset U$ is a connected Borel set. 
Almost surely,
\eqb \label{eqn-cle-cover} 
K\subset \ol{\bigcup \left\{\gamma\in\Gamma : \gamma\cap K\not=\emptyset , \gamma\subset U \right\}} . 
\eqe
\end{lem}
\begin{proof}
For each $z\in K$ and $\ep > 0$, a.s.\ there is a loop in $\Gamma$ with Euclidean diameter at most $\ep$ which disconnects $z$ from $\infty$. 
Since $z\in K\subset U$ and $K$ is connected, for small enough $\ep > 0$, this loop must be contained in $U$ and must intersect $K$. 
Hence a.s.\ $z$ belongs to the closure on the right side of~\eqref{eqn-cle-cover}. 
Since $z\in K$ is arbitrary, we get that this closure a.s.\ contains a dense subset of $K$, hence it a.s.\ contains $K$. 
\end{proof}

\subsection{CLE on an annulus}
\label{sec-cle-annulus}
 
In this subsection, we will state a result to the effect that for each $\kappa \in (4,8)$ and $M\in\BB N_0$, there is a unique law on locally finite, non-crossing loop configurations on an annulus which has exactly $M$ loops with non-trivial winding number and which satisfies a certain Markov property (Theorem~\ref{thm-cle-annulus}). 
The results stated in this section are proven in Sections~\ref{sec-markov} and~\ref{sec-resampling}. 
We define this law to be $\CLE_\kappa$ on the annulus for $\kappa \in (4,8)$. 
We will also state a result which says that restricting $\CLE_\kappa$ on the disk to the non-simply-connected complementary connected component of a simple loop gives a $\CLE_\kappa$ on the annulus, under our definitions (Theorem~\ref{thm-cle-markov}), which shows that the definition in this subsection is equivalent to the one in Theorem~\ref{thm-cle-annulus0}.
As discussed in Section~\ref{sec-intro}, the relevance to the proof of our main result is that the re-sampling property which characterizes $\CLE_\kappa$ on the annulus is invariant under inversion, so the law of $\CLE_\kappa$ on the annulus is invariant under inversion (see Corollary~\ref{cor-annulus-inversion}).

Recall the annulus $\BB A_\rho$ for $\rho\in (0,1)$ from~\eqref{eqn-annulus-def}.  
The idea behind the Markov property which characterizes $\CLE_\kappa$ on $\BB A_\rho$ is to choose a subset of loops in $\BB A_\rho$ such that each connected component of the complement of their closed union is simply connected (e.g., the set of loops which intersect a line segment from the inner boundary to the outer boundary). We then require that the law of the restriction of the $\CLE_\kappa$ to the complement of this closed union is that of a $\CLE_\kappa$ in each of these simply connected components. 
By itself, such a property is not enough to characterize $\CLE_\kappa$ on $\BB A_\rho$ since the set of loops which intersect a path between the inner and outer boundaries of $\BB A_\rho$ will always include all of the loops which disconnect the inner and outer boundaries. So, we also need to impose a re-sampling condition on these loops. 
To state this re-sampling condition, we first introduce some notation, which is illustrated in Figure~\ref{fig-stubs}. 

\begin{defn}[$P$-excursions of loops] \label{def-excursion}
Let $P\subset \BB C$ be a compact set and let $U\subset \BB C$ be an open set containing $P$. 
Also let $\gamma $ be a loop in $\BB C$ (with some arbitrary choice of parameterization).
We say that an arc $\alpha$ of $\gamma$ is a \emph{$P$-excursion of $\gamma$ into $U$} if $\alpha \subset \ol U$, $\alpha \cap P\not=\emptyset$, and $\alpha$ is not properly contained in any larger arc of $\gamma$ with these properties.  
We say that $\alpha$ is \emph{proper} if $\alpha\notin \{\gamma,\emptyset\}$ (equivalently, $\gamma \cap P \not=\emptyset$ and $\gamma\not\subset \ol U$).
An arc $ \alpha'$ is called a \emph{complementary $P$-excursion of $\gamma$ out of $U$} if $\alpha'$ does not overlap with any $P$-excursion of $\gamma$ into $U$ and $\alpha'$ is not contained in any larger arc of $\gamma$ with this property. 
\end{defn}
 
By definition, a loop is the concatenation of its $P$-excursions into $U$ and its complementary $P$-excursions out of $U$, and these arcs overlap only at their endpoints.

\begin{defn}[Sets of loops and excursions] \label{def-stub}
Let $\Gamma$ be a locally finite collection of non-crossing loops in a domain $D\subset \BB C$. For a compact set $P \subset \ol D$ and an open set $U\subset D$ with $P\subset U$, we write $\Gamma(P ; U)$ for the set of loops in $\Gamma$ which intersect $P$ and are contained in $U$. 
We write $\Gamma(P) := \Gamma(P; \ol D)$ for the set of all loops in $\Gamma$ which intersect $P$.
We also write $\mcl S_\Gamma(P; U)$ for the set of all proper $P$-excursions into $U$ of loops in $\Gamma$. Note that $\mcl S_\Gamma(P; U)$ is a finite set since $\Gamma$ is locally finite.
\end{defn}

\begin{figure}[t!]
 \begin{center}
\includegraphics[scale=.85]{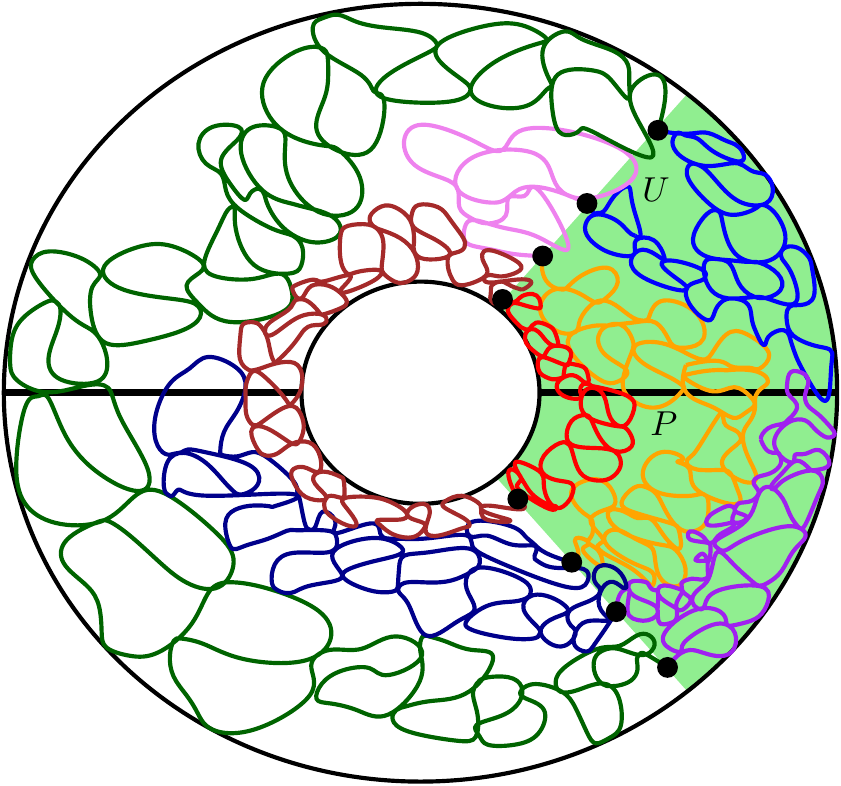}
\vspace{-0.01\textheight}
\caption{Illustration of a possible configuration of $P $-excursions into $U $ (red, orange, purple, blue) and complementary $P $-excursions out of $U $ (dark green, dark blue, pink, brown) for a loop configuration on $\BB A_\rho$. Here there are two loops which disconnect the inner and outer boundaries of $\BB A_\rho$ (union of purple, dark green, blue, pink, orange, and dark blue arcs; and union of red and brown arcs). The loop consisting of the red and brown arcs has winding number 1 around the inner boundary, and the other loop has winding number zero. 
}\label{fig-stubs}
\end{center}
\vspace{-1em}
\end{figure}

See Figure~\ref{fig-stubs} for an illustration of Definitions~\ref{def-excursion} and~\ref{def-stub}. 
Each $P$-excursion $\alpha \in \mcl S_\Gamma(P;U)$ has two endpoints in $\bdy U$, which we call its initial and terminal endpoints.
There is a distinguished bijection between the set of initial endpoints of elements of $\mcl S_\Gamma(P;U)$ and the set of terminal endpoints of elements of $\mcl S_\Gamma(P;U)$ whereby a terminal endpoint $x$ corresponds to an initial endpoint $y$ if there is a complementary $P $-excursion out of $U $ of a loop of $\Gamma$ which has $x$ and $y$ as its initial and terminal endpoints, respectively. In this case we write $y = x^*$. Note that the initial and terminal endpoints of a $P$-excursion into $U $ need not correspond to each other under this bijection (although they sometimes do).  
In fact, the bijection $x\leftrightarrow x^*$ is \emph{not} determined by $\mcl S_\Gamma(P;U)$ since it depends on how loops behave when they are outside of $\ol U$. 
  
\begin{defn}[Annulus Markov property] \label{def-annulus-markov}
Let $\rho \in (0,1)$ and let $\Gamma$ be a random locally finite, non-crossing collection of loops on the annulus $\BB A_\rho$. 
We say that $\Gamma$ satisfies the \emph{annulus Markov property} if the following is true. 
Let $P = [\rho e^{i\theta} , e^{i\theta}]$ be a radial line segment between the inner and outer boundaries of $\BB A_\rho$ and let $U = \{r e^{i s} : r\in (\rho,1) , s \in (\theta-\pi/4 ,\theta+ \pi/4 )\}$ be the annular slice centered at $P$ of opening angle $\pi/2$. 
If $\mcl S_\Gamma(P;U) \not=\emptyset$, choose a $P$-excursion into $U$ from $\mcl S_\Gamma(P;U)$ in a manner which is measurable with respect to $\sigma( \Gamma(P;U) , \mcl S_\Gamma(P;U) )$. Let $x$ be its terminal endpoint and let $\eta_x$ be the complementary $P $-excursion out of $U$ from $x$ to $x^*$. 
\begin{enumerate}
\item Almost surely, $P\subset \ol{\bigcup \Gamma(P;U)}$. \label{item-annulus-markov-contain}
\item If $\mcl S_\Gamma(P;U) \not=\emptyset$ and we condition on $\Gamma(P;U)$, $\mcl S_\Gamma(P;U)$, and all of the complementary $P $-excursions of loops in $\Gamma$ out of $U $ except for $\eta_x$, then the conditional law of $\eta_x$ is that of a chordal $\SLE_\kappa$ from $x$ to $x_*$ in the connected component of $\BB A_\rho\setminus \ol{\bigcup \Gamma(P)\setminus \eta_x}$ with $x$ on its boundary.   \label{item-annulus-markov-sle}
\item If we further condition on $\eta_x$ (equivalently, we condition on $\Gamma(P)$) then the conditional law of $\Gamma|_{\BB A_\rho\setminus \ol{\bigcup \Gamma(P)}}$ is that of a collection of independent $\CLE_\kappa$'s in the connected components of $\BB A_\rho\setminus \ol{\bigcup \Gamma(P)}$.  \label{item-annulus-markov-cle}
\end{enumerate}
\end{defn}

Since $P$ intersects the inner and outer boundaries of $\BB A_\rho$, condition~\ref{item-annulus-markov-contain} in Definition~\ref{def-annulus-markov} implies that each of the sets $\bdy \BB A_\rho \cup \ol{\bigcup \Gamma(P)}$ and $\bdy\BB A_\rho\cup \ol{\bigcup \Gamma(P) \setminus \eta_x}$ is connected.  
Hence, a.s.\ each of the connected components of each of $\BB A_\rho\setminus \ol{\bigcup \Gamma(P)\setminus \eta_x}$ and $\BB A_\rho\setminus \ol{\bigcup \Gamma(P) }$ is simply connected, so it makes sense to talk about CLE$_\kappa$ on these connected components.

\begin{thm} \label{thm-cle-annulus}
For each $\kappa \in (4,8)$, $\rho  \in (0,1)$, and $M\in\BB N_0$, there is a unique law on non-crossing, locally finite collections of loops in $\BB A_\rho$ which satisfies the annulus Markov property and a.s.\ includes exactly $M$ loops whose winding number around the inner boundary is non-zero (the winding number of each such loop must be 1 since the loops do not cross themselves).  
\end{thm}

Theorem~\ref{thm-cle-annulus} is an annulus analog of the Markovian characterization for CLE$_\kappa$ on a simply connected domain given in~\cite[Theorem 5.3]{shef-cle}.
If $\Gamma$ satisfies the condition of Definition~\ref{def-annulus-markov}, then so does its images under $z\mapsto \rho/z$ and $z\mapsto e^{i\theta} z$ for any $\theta  \in [0,2\pi]$. We therefore obtain the following corollary.

\begin{cor} \label{cor-annulus-inversion}
If $\kappa \in (4,8)$, $\rho  \in (0,1)$, $M\in\BB N_0$, and $\Gamma$ is distributed according to the unique law of Theorem~\ref{thm-cle-annulus}, then the law of $\Gamma$ is invariant under conformal automorphisms (inversion and rotation) of $\BB A_\rho$. 
\end{cor}

Corollary~\ref{cor-annulus-inversion} will be a key tool in the proof of our main theorem. It also allows us to make the following definition. 

\begin{defn} \label{def-cle-annulus}
Let $\mcl A\subset \BB C$ be an open domain with the topology of an annulus and let $f : \mcl A\rta \BB A_\rho$ be a conformal map into an annulus for some $\rho > 0$. For $\kappa \in (4,8)$ and $M\in\BB N_0$, we define \emph{$\CLE_\kappa$ on $\mcl A$ with $M$ inner-boundary-surrounding loops} to be the image under $f^{-1}$ of a loop configuration on $\BB A_\rho$ distributed according to the unique law of Theorem~\ref{thm-cle-annulus}. 
\end{defn}

The choice of $f$ in Definition~\ref{def-cle-annulus} does not matter due to Corollary~\ref{cor-annulus-inversion}. 
The following theorem combines with Theorem~\ref{thm-cle-annulus} and Corollary~\ref{cor-annulus-inversion} to give Theorem~\ref{thm-cle-annulus0}. 

\begin{thm} \label{thm-cle-markov}
Let $M\in\BB N_0$, let $\Gamma$ be a $\CLE_\kappa$ on $\BB D$, and let $\gamma_{M+1}$ be the $(M+1)$st outermost loop in $\Gamma$ surrounding 0. 
On the event $\{\gamma_{M+1} \cap \bdy\BB D = \emptyset\}$ (which has probability 1 if $M \geq 1$), let $\mcl A_M$ be the non-simply connected component of $\BB D\setminus \gamma_{M+1}$ and let $f_M : \mcl A_M \rta \BB A_\rho$ for some $\rho > 0$ be the conformal map which fixes 1 (note that $\rho$ is random and determined by $\gamma_{M+1}$). 
Almost surely, the conditional law given $\gamma_{M+1}$ of the loop ensemble $f_M(\Gamma|_{\mcl A_M})$ on the event $\{\gamma_{M+1} \cap \bdy\BB D = \emptyset\}$ satisfies the annulus Markov property of Definition~\ref{def-annulus-markov}, so has the law of a $\CLE_\kappa$ on $\BB A_\rho$ with $M$ inner-boundary-surrounding loops.  
\end{thm}

\section{CLE satisfies the annulus Markov property}
\label{sec-markov}

In this section we will prove Theorem~\ref{thm-cle-markov}.
To accomplish this, we will need to establish several versions of the Markov property for $\CLE_\kappa$ on $\BB D$ which build on the basic Markov property established in~\cite[Theorem~5.3]{shef-cle} (restated as Lemma~\ref{lem:Sheffield} below).  
We start in Section~\ref{sec-general-markov} by proving a Markov property for the conditional law of the rest of the $\CLE_\kappa$ when we condition on all of the loops which intersect a fixed compact set $K$ with $K\cap\bdy\BB D\not=\emptyset$.  This property for $\kappa \in (4,8)$ is the analog of the restriction property of $\CLE_\kappa$ for $\kappa \in (8/3,4]$, which was used to characterize the simple  $\CLE$s \cite{shef-werner-cle}.  Since the annulus Markov property requires us to condition on only \emph{part} of some loops, we will also need a suitable Markov property for an $\SLE_\kappa$ coupled with a $\CLE_\kappa$, which we establish in Section~\ref{sec-sle-cle-markov}. 
In Section~\ref{sec:nk}, we combine the preceding two sections to establish a variant of the annulus Markov property for $\CLE_\kappa$ on the disk (when we do not condition on one of the origin-surrounding loops). 
In Section~\ref{sec-annulus-markov}, we conclude the proof of Theorem~\ref{thm-cle-markov}.

\subsection{General Markov property}
\label{sec-general-markov}

Fix $\kappa\in(4,8)$. Let $\Gamma$ be a $\CLE_\kappa$ in the unit disk $\BB D$.
By \cite[Theorem 5.3]{shef-cle} and the local finiteness of $\CLE_\kappa$ \cite{ig4}, it is known that $\CLE_\kappa$ has the following Markovian property.

\begin{lemma}[\!\!\cite{shef-cle, ig4}]\label{lem:Sheffield}
Suppose that $I\subset\partial\BB D$ is a deterministic arc.  Given the set $\Gamma(I)$ of loops in $\Gamma$ that intersect $I$, the conditional law of the rest of $\Gamma$ is that of an independent $\CLE_\kappa$ in each connected component of $ \BB D\setminus \ol{\bigcup \Gamma(I)} $. 
\end{lemma}

Now, we want to extend this Markov property to a more general version where the arc $I$ is replaced by a more general set.

\begin{lemma}[General Markov property]\label{lem:general-markov}
Let $K\subset \overline{\BB D}$ be a deterministic compact \emph{connected} set which intersects $\partial\BB D$.  
Then given $\Gamma(K)$, the conditional law of the rest of $\Gamma$ is that of an independent $\CLE_\kappa$ in each connected component of $\BB D\setminus \ol{\bigcup \Gamma(K)}$.
\end{lemma}

Note that Lemma~\ref{lem-cle-cover} implies that in the setting of Lemma~\ref{lem:general-markov}, a.s.\ $\ol{\bigcup \Gamma(K)}$ contains $K$, and hence is connected. 
In particular, connected components of $\BB D\setminus \ol{\bigcup \Gamma(K)}$ are simply connected, so it makes sense to talk about CLE$_\kappa$ in these components.
Lemma~\ref{lem:general-markov} is the $\kappa \in (4,8)$ analog of the spatial Markov property of $\CLE_\kappa$ for $\kappa \in (8/3,4]$ which was established in~\cite{shef-werner-cle}. For $\kappa \in (8/3,4]$, it is shown in~\cite{shef-werner-cle} that this property together with conformal invariance uniquely characterizes the law of $\CLE_\kappa$ (hence showing that $\CLE_\kappa$ loops are distributed as the outer boundaries of Brownian loop-soup clusters).

\begin{proof}[Proof of Lemma~\ref{lem:general-markov}]
Due to conformal invariance of CLE, the origin plays no particular role. Therefore, we can assume that origin is not in $K$ and denote by $O$ the connected component of $\BB D \setminus\ol{\bigcup \Gamma(K)}$ that contains the origin. It then suffices to prove that conditionally on $\Gamma(K)$, $\Gamma$ restricted to $O$ is a $\CLE_\kappa$ in $O$ which is independent of the restriction of $\Gamma$ to any other connected component of $\BB D \setminus \ol{\bigcup \Gamma(K)}$.

We will explore the $\CLE$ from the boundary towards the origin. In order to use Lemma~\ref{lem:Sheffield}, we first consider the $\eps$-neighborhood of $K$ and denote it by $K^\eps$ (e.g., this is particularly necessary when $K$ is a line). Let $O^\eps$ be the connected component of ${\BB D}\setminus \ol{\bigcup \Gamma(K^\eps)}$ containing the origin. 
Let us first prove the following statement:
\eqb \label{eqn-cle-markov0}
\text{Conditionally on $\Gamma(K^\eps)$, $\Gamma$ restricted to $O^\eps$ is an independent $\CLE_\kappa$ in $O^\eps$.}
\eqe

Note that $K^\eps\cap\partial{\BB D}$ is the union of countably many arcs.
 Hence, as a first step, we can discover all the loops  in $\Gamma$ that intersect $K^\eps\cap\partial \BB D$ and we denote the closure of their union by $\CK_1^\eps$. Let $O_1^\eps$ be the connected component of $\BB D \setminus \CK_1^\eps$ that contains the origin.  We know by Lemma~\ref{lem:Sheffield} that conditionally on the loops which make up $\CK_1^\eps$, $\Gamma|_{O_1^\eps}$ is a $\CLE_\kappa$ in $O_1^\eps$ which is independent from the restriction of $\Gamma$ to each other connected component of $\BB D \setminus \CK_1^\eps$.
If $O_1^\eps \cap K^\eps=\emptyset$, then we would have proved~\eqref{eqn-cle-markov0}. Otherwise, we continue to explore $\Gamma$ restricted to $O_1^\eps$ and discover all the loops that intersect $\partial O_1^\eps \cap K^\eps$. 
This process can be iterated:
Suppose that at step $n$, we let $\CK_n^\eps$ be the closure of the union of all of the loops we have discovered so far and let $O_n^\eps$ be the component of $\BB D \setminus \CK_n^\eps$ which contains the origin.  Then Lemma~\ref{lem:Sheffield} and induction shows that conditionally on the loops which make up $\CK_n^\eps$, $\Gamma|_{O_n^\eps}$ is a $\CLE_\kappa$ in $O_n^\eps$ which is independent from the restriction of $\Gamma$ to each other connected component $\BB D \setminus \CK_n^\eps$.
If $O_n^\eps \cap K^\eps=\emptyset$, then we would have proved~\eqref{eqn-cle-markov0}. Otherwise, we continue to explore $\Gamma$ restricted to $O_n^\eps$ and discover all the loops that intersect $\partial O_n^\eps \cap K^\eps$. 

If the above exploration process ends in finitely many steps, then we would have already proved~\eqref{eqn-cle-markov0}.  Otherwise, we need to prove that $\ol{\cup_n \CK_n^\eps} = \ol{\cup \Gamma(K^\epsilon)}$.  It is clear that $\ol{\cup_n \CK_n^\eps} \subseteq \ol{\cup \Gamma(K^\epsilon)}$.  If the containment were strict, then it means that there is a loop $\gamma$ which intersects $K^\epsilon$ which is not discovered by the exploration process.  In other words, $\gamma \subseteq O_n^\eps$ for all $n$.
Since $K^\eps$ is open, $\gamma$ must intersect the interior of $K^\eps$ , so the intersection of the component containing $0$ of $\BB D \setminus \ol{\cup_n \CK_n^\eps}$ and $K^\eps$ has non-empty interior.  We will prove that this is impossible.  It suffices to show that for any point $z\in K^\eps$, there is a.s.\ a finite $n$ for which $z \not\in O_n^\eps$.

Now fix  $z\in K^\eps$. 
For each step $n$ such that $z\in O_n^\eps$, the harmonic measure seen from $z$ inside $O_n^\eps$ of $K^\eps \cap \partial O_n^\eps$ is greater than the harmonic measure seen from $z$ inside $K^\eps \cap O_n^\eps$ of $K^\eps\cap \partial O_n^\eps$, which is again greater than the harmonic measure seen from $z$ inside $K^\eps$ of $K^\eps\cap \partial{\BB D}$, which is equal to some $p>0$ which depends on $\eps$, but not $n$.
On the other hand, $z$ and the origin are relatively far away from each other inside $O_n^\eps$ in the sense that if one starts two independent Brownian motions from $z$ and the origin, then the probability that they meet before hitting $\partial O_n^\eps$ is decreasing in $n$, hence bounded above by some $q<1$.
This means that if one maps $O_n^\eps$ to the unit disk by some conformal map $\varphi$ sending $z$ to the origin, then $\varphi(K^\eps\cap \partial O_n^\eps)$ will have Lebesgue measure at least $p$ and the distance from $\varphi(0)$ to the origin is greater than some $d\in(0,1)$.
By combining this with Lemma~\ref{lem-cle-hit} at each step such that $z\in O_n^\eps$, the probability that one discovers in the $(n+1)$st step a loop that encircles $z$ and disconnects $z$ from the origin is bounded below by some constant $c>0$.
Therefore, the number of steps that it takes to have $z\not\in O_n^\eps$ is stochastically dominated by a geometric random variable, thus a.s.\ finite.
We have thus proved $\ol{\cup_n \CK_n^\eps} = \ol{\cup \Gamma(K^\epsilon)}$, hence have also proved~\eqref{eqn-cle-markov0}. 

Now it only remains to let $\eps$ go to $0$.  Let us first show that $\ol{\bigcup \Gamma(K^\eps)}$ converges a.s.\ to $\ol{\bigcup \Gamma(K)}$ w.r.t.\ the Hausdorff distance. Note that $\ol{\bigcup \Gamma(K^\eps)}$ is decreasing in $\eps$, i.e., $\ol{\bigcup \Gamma(K^{\eps_1})} \subset \ol{\bigcup \Gamma(K^{\eps_2})}$ if $\eps_1<\eps_2$, hence it converges to the limit $ \cap_{\eps>0} \ol{\bigcup \Gamma(K^\eps)}$.
For any set $A\subset \overline{\BB D}$ which is at positive distance away from $K$, there are a.s.\ at most finitely many loops in $\Gamma$ that intersect both $A$ and $K^\eps \setminus K$, provided that $\eps$ is small enough, due to the fact that CLE is locally finite.
This implies that if $\ol{\bigcup \Gamma(K)} \cap A=\emptyset$, then $\ol{\bigcup \Gamma(K^\eps)} \cap A=\emptyset$ when $\eps$ is small enough. Hence $ \cap_{\eps>0} \ol{\bigcup \Gamma(K^\eps)}$ is contained in $\ol{\bigcup \Gamma(K)}$. Since it obviously also contains $\ol{\bigcup \Gamma(K)}$, the two sets are actually equal. 
Therefore $O^\eps$ a.s.\ increases to $O$.
Therefore, $\Gamma$ restricted to $O$ is distributed as the limit of $\Gamma$ restricted to $O^\eps$ which is an independent $\CLE$ in $O$.
\end{proof} 

Lemma~\ref{lem:general-markov} allows us to explore the CLE along any given simple path $P: [0,1]\to \ol{\BB D}$ such that $P(0)\in\partial \BB D$ by taking $K=P[0,t]$ for $t\in [0,1]$. This exploration also satisfies the following strong Markov property.

\begin{lemma}[Exploring along a line]\label{lem:strong_markov}
If $T$ is a stopping time for the filtration generated by $\{ \Gamma(P[0,t]) \}_{t\in [0,1]}$, then conditionally on $\Gamma(P[0,T])$, the rest of $\Gamma$ is distributed as an independent $\CLE_\kappa$ in each connected component of $\BB D\setminus \ol{\bigcup \Gamma(P[0,T])}$.
\end{lemma}
\begin{proof}
For any stopping time $T$, we define $T_n$ to be the smallest real number greater than $T$ which is in the set $D_n:=\{k 2^{-n}, k\in\BB Z\} \cap [0,1]$. Then $T_n$ is  a stopping time for the exploration process and can take on only finitely many possible values. Lemma~\ref{lem:general-markov} implies that the strong Markov property holds with $T_n$ in place of $T$. As $n\to\infty$, $T_n$ decreases towards $T$ and $\ol{\bigcup\Gamma(P[0,T])}$ decreases towards $\ol{\bigcup\Gamma(P[0,T_n])}$ a.s. Therefore, the strong Markov process also holds at time $T$.
\end{proof}

\subsection{Markov property for SLE decorated CLE}
\label{sec-sle-cle-markov}
 
In this subsection we prove a variant of Lemma~\ref{lem:general-markov} which applies to a $\CLE_\kappa$ coupled with an $\SLE_\kappa$.  This variant is needed to describe the conditional law of the complementary $P$-excursion out of $U$ in the setting of Definition~\ref{def-annulus-markov}.

Let $B \subset\overline{\BB D}$ be a closed set which does not disconnect $-i, i$.  Let $\eta$ be an $\SLE_\kappa$ from $-i$ to $i$ in $\BB D$ conditioned to avoid $B$. Conditionally on $\eta$, let $\Gamma$ be a $\CLE_\kappa$ in $\BB D\setminus \eta$ (i.e., a collection of independent $\CLE_\kappa$'s in each of the connected components of $\BB D\setminus \eta$). Also let $I\subset\bdy{\BB D}$ be a connected arc and let $\Gamma(I)$ be the set of loops in $\Gamma$ which intersect $I$.  We define $D_0$ to be the connected component of $\BB D\setminus \ol{\bigcup \Gamma(I) }$ with $-i$ and $i$ on its boundary. Let $B_0:=B\cap D_0$.

\begin{lemma} \label{lem-sle-cle-markov}
The conditional law of $\eta$ given $\Gamma(I)$ and the event $\{\eta\cap I=\emptyset\}$ is that of an $\SLE_\kappa$ in $D_0$ conditioned to avoid $B_0$. 
\end{lemma}

\begin{figure}[t!]
 \begin{center}
\includegraphics[scale=1.4]{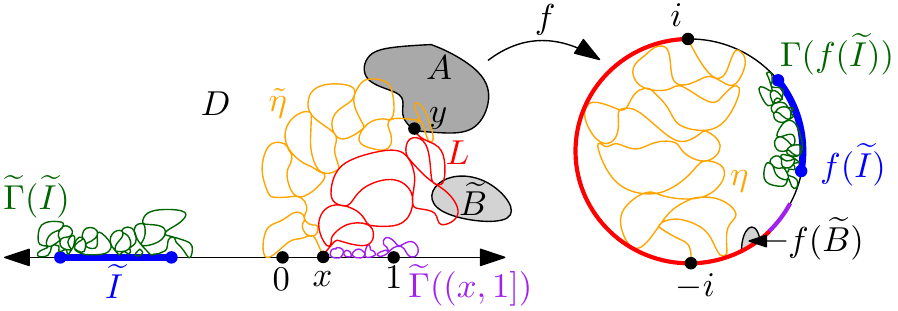}
\vspace{-0.01\textheight}
\caption{Illustration of the proof of Lemma~\ref{lem-sle-cle-markov}.
Starting with a $\CLE_\kappa$ $\wt\Gamma$ on $\BB H$, we construct an $\SLE_\kappa$ curve $\wt\eta$ in a random domain $D$ (the complementary connected component of the red and purple sets containing the orange curve) with the property that the conditional law of $\wt\Gamma|_{D\setminus \wt\eta}$ given $D$ and $\wt\eta$ is that of a collection of independent $\CLE_\kappa$'s in the connected components of $D\setminus \wt\eta$. We show using~\cite[Theorem 5.4, Properties 4 and 5]{shef-cle} that the pair $(\wt\eta , \wt\Gamma|_{D\setminus\wt\eta})$ satisfies the Markov property in the statement of the lemma, then conformally map $D$ to $\BB D$ to conclude.
}\label{fig-sle-cle-markov}
\end{center}
\vspace{-1em}
\end{figure}

\begin{proof}
The idea of the proof is to start with a $\CLE_\kappa$ process $\wt\Gamma$ (we will take this $\CLE_\kappa$ to be on $\BB H$, for convenience) and express the pair $(\eta,\Gamma)$ in the statement of the lemma as a functional of $\wt\Gamma$. We will then be able to compute the conditional law of $\eta$ given $\Gamma(I)$ using~\cite[Theorem 5.4, Properties 4 and 5]{shef-cle}. See Figure~\ref{fig-sle-cle-markov} for an illustration. 

Let $\wt\Gamma$ be a $\CLE_\kappa$ on $\BB H$. Fix a non-trivial connected closed set $A \subset \BB H$ which is disjoint from $[0,1] $. On the event $\{\ol{\bigcup \wt\Gamma([0,1])}\cap A\not=\emptyset\}$, let $x$ be the rightmost point of $[0,1]$ which lies on one of the finitely many loops of $\wt\Gamma([0,1])$ which intersect $A$. Let $ \gamma$ be the loop of $\wt\Gamma$ with $x\in  \gamma$ and let $ L$ be the counterclockwise arc of $ \gamma$ from $x$ to the first point $y$ at which $ \gamma$ (traversed counterclockwise) hits $A$. Let $\wt\eta$ be the counterclockwise arc of $ \gamma$ from $y$ to $x$ (so that $\wt\eta$ traverses $ \gamma\setminus L$). By~\cite[Theorem 5.4, Property 4]{shef-cle}, if we condition on $L$ and $\wt\Gamma((x,1])$, then the conditional law of $\wt\eta$ is that of an $\SLE_\kappa$ from $y$ to $x$ in the connected component of $\BB H \setminus \ol{L \cup \bigcup \wt\Gamma((x,1])}$ with $y$ and $x$ on its boundary. Call this connected component $D$. 
Fix a closed set $\wt B \subset \ol{\BB H}$ which is disjoint from $[0,1]$ and does not disconnect $x$ and $y$, chosen in a measurable way w.r.t.\ $L$ and $\wt\Gamma((x,1])$.

On the event $\{\ol{\bigcup \wt\Gamma([0,1])}\cap A\not=\emptyset\} $, if we condition on $L$, $\wt\Gamma((x,1])$, and the event $ \{ \wt\eta\cap \wt B=\emptyset \}$, then the conditional law of $\wt\eta$ is that of an $\SLE_\kappa$ from $y$ to $x$ in $D$ conditioned to avoid $\wt B$. If we further condition on $ \wt\eta$ then the conditional law of $\wt\Gamma|_{D\setminus\wt\eta}$ is that of a collection of independent $\CLE_\kappa$'s in the connected components of $D\setminus \wt\eta$. 
In other words, the conditional law of $(\wt\eta , \wt\Gamma)$ given $L$, $\wt\Gamma([x,1])$, and the event $ \{ \wt\eta\cap \wt B=\emptyset \}$ is as in the statement of the lemma but with $(D , \wt B)$ in place of $(\BB D , B)$. 

Let $\wt I \subset (-\infty,0]$ be a connected arc, also chosen in a measurable way w.r.t.\ $L$ and $\wt\Gamma((x,1])$. We will describe the conditional law of $\wt\eta$ given $L$, $\wt\Gamma((x,1])$, and $\wt\Gamma(\wt I)$. By~\cite[Theorem 5.4, Property 5]{shef-cle}, if we condition on $\wt\Gamma(\wt I)$, then the conditional law of $\wt\Gamma$ is that of a collection of independent $\CLE_\kappa$'s in the connected components of $\BB H\setminus \ol{\bigcup  \wt\Gamma(\wt I) }$. If $  \gamma \notin \wt\Gamma(\wt I)$, equivalently $ \gamma \cap \wt I = \emptyset$, then $\wt\eta$ is contained in the closure of a single connected component $U$ of $\BB H\setminus \ol{\bigcup\wt\Gamma(\wt I)}$. 
We observe that the event $\{ \gamma \cap \wt I = \emptyset \}$ is the same as the event that the loop arc $L$ is not a subset of a loop in $\wt\Gamma(\wt I)$, so this event is determined by $L$ and $\wt\Gamma(\wt I)$. 
With positive probability, $\ol{L \cap \bigcup \wt\Gamma((x,1])} \cap (-\infty,0] = \emptyset$, in which case $D$ is unbounded and $(-\infty,0] \subset \bdy D$. 
By~\cite[Theorem 5.4, Property 4]{shef-cle} applied to the $\CLE_\kappa$ $\wt\Gamma|_U$ and the boundary arc $\bdy U \cap [0,1]$ of $\bdy U$, we find that the conditional law of $\wt\eta$ given $L$, $\wt\Gamma((x,1])$, and $\wt\Gamma(\wt I)$ on the event 
\eqb \label{eqn-sle-cle-markov-event}
\{\wt \eta \cap \wt I = \emptyset\} \cap E \quad \text{where} \quad E := \{ L \cap \wt I = \emptyset \} \cap \{ (-\infty,0]\subset \bdy D \} \cap \{\ol{\bigcup \wt\Gamma([0,1])}\cap A\not=\emptyset\}
\eqe 
is that of a chordal $\SLE_\kappa$ from $x$ to $y$ in $U$. 
This implies that the conditional law of $\wt\eta$ given $L$, $\wt\Gamma((x,1])$, $\wt\Gamma(\wt I)$ and  $\{\wt\eta \cap \wt B=\emptyset\}$ on the event $\{\wt \eta \cap \wt I = \emptyset\}  \cap E$ is  that of a chordal $\SLE_\kappa$ from $x$ to $y$ in $U$ conditioned to avoid $\wt B$. 
We emphasize that the event $E$ is determined by $L$ and $\wt\Gamma((x,1])$. 
 
We now want to transfer from the random doubly marked domain $(D,x,y)$ to the deterministic domain $(\BB D , -i ,i)$.

On the event $E$ of~\eqref{eqn-sle-cle-markov-event}, let $f : D\rta \BB D$ be the conformal map which takes $y$ to $-i$, $x$ to $i$, and the right endpoint of $\wt I$ to the upper endpoint of $I$. Note that $f$ is measurable w.r.t.\ $L$ and $\wt\Gamma((x,1])$. Define
\eqbn
\eta := f(\wt\eta)    \quad \text{and} \quad  \Gamma := f(\wt\Gamma|_D) .
\eqen
Now, let $B$ and $I$ be as in the statement of the lemma. Choose $\wt B= f^{-1}(B)$ and $\wt I= f^{-1}(I)$. 

Since $f$ is measurable w.r.t.\ $L$ and $\wt\Gamma((x,1])$, we can apply the results of the third paragraph of the proof: If we condition on $L$ and $\wt\Gamma((x,1])$, which determine $f$, and the event $\{\wt\eta \cap \wt B=\emptyset\}$, then on $E $, the conditional law of the pair $(\eta,\Gamma)$ is as in the statement of the lemma.  

By the fourth paragraph of the proof, if we further condition on $\Gamma(I)$, then on the event $\{\wt \eta \cap \wt I = \emptyset\} \cap E$, the conditional law of $\eta$ is that of a chordal $\SLE_\kappa$ from $-i$ to $i$ in $D_0$ conditioned to avoid $B_0$, which proves the lemma. 
\end{proof}

\subsection{The $(n,k)$-exploration process}
\label{sec:nk}

To verify Definition~\ref{def-annulus-markov} for $\CLE_\kappa$, we want to use the Markov properties of $\CLE_\kappa$ as described in the preceding two subsections. For this purpose, we first need to define and analyze a ``Markovian" way of picking out a complementary $P$-excursion out of $U$. Since it takes no extra effort, we will allow for a slightly more general choice of $P$ and $U$ than the ones in Definition~\ref{def-annulus-markov}.
 
Let $\Gamma$ be a $\CLE_\kappa$ on $\BB D$ and let $P : [0,1] \rta \ol{\BB D}$ be a simple path (either deterministic or random but independent of $\Gamma$). Also let $U \subset \BB D$ be an open set with $P \cap \BB D\subset U$.

For $n, k \in \BB N  $, we define the \emph{$(n,k)$-exploration process} of $\Gamma$ along $P$, relative to $U$, to be the pair $(\alpha_{n,k} , \mcl L_n)$ defined as follows. See Figure~\ref{fig:nk-exploration} for an illustration of the definitions.

\begin{figure}[h!]
 \begin{center}
\includegraphics[scale=.65]{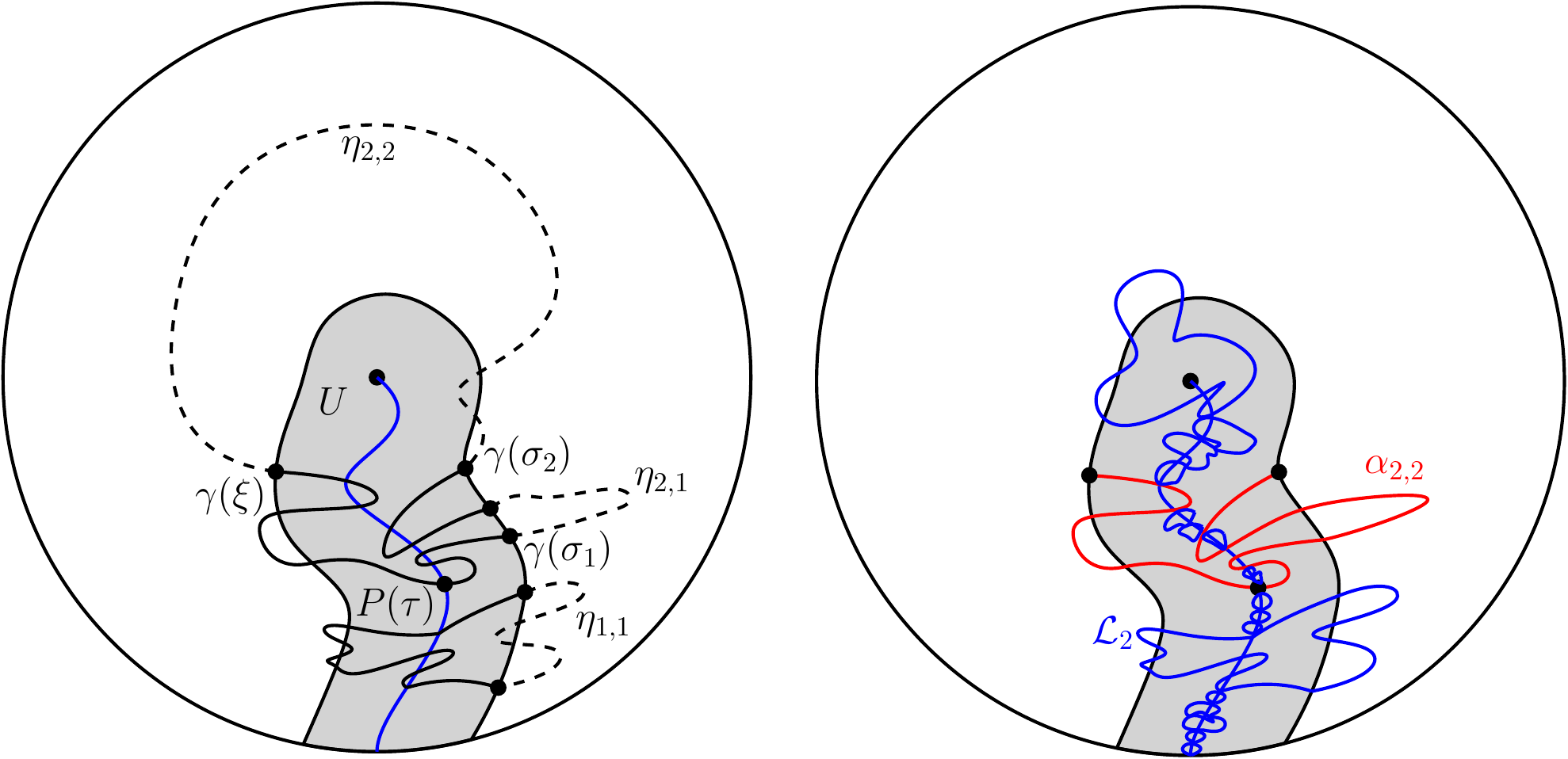}
\vspace{-0.01\textheight}
\caption{Illustration of the $(n,k)$-exploration process. On the left, we show $\eta_{1,1}$, $\eta_{2,1}$, and $\eta_{2,2}$ and the times $\sigma_1$, $\sigma_2$, and $\xi$ involved in the definition of $\eta_{2,2}$. Only the loops which exit $U$ are drawn in this picture. On the right, we show $\alpha_{2,2}$ and $\CL_2$ for the same loop configuration as on the left. We show simple loops for clarity, but in actuality the loops will intersect themselves and each other. }\label{fig:nk-exploration}
\end{center}
\vspace{-1em}
\end{figure} 

\begin{enumerate}
\item Let $t_n$ be the $n$th smallest $t \in [0,1]$ for which the following is true: there is some loop $\gamma \in \Gamma$ such that $\gamma \not\subset \ol U$ and $P$ hits $\gamma$ for the first time at time $t$; or let $t_n = 1$ if there are fewer than $n$ such times $t\in [0,1]$. Note that by the local finiteness of $\Gamma$, there are at most finitely many loops which intersect both $P$ and $\BB D\setminus \ol U$, so $t_n$ is well-defined.
\item If $t_n < 1$, we parameterize the loop $\gamma$ in the counterclockwise direction by $[0,1]$ so that $\gamma(0)=\gamma(1)=P(t_n)$.

Let $\sigma_0 = 0$. If $j \in \BB N_0$, inductively let $\sigma_{j+1}$ be the first time after $\sigma_j$ at which $\gamma$ completes a crossing from $P$ to $\BB D\setminus U$, i.e., the smallest $s \in (\sigma_j,1]$ for which $\gamma(s) \not\in U$ and there is an $s' \in(\sigma_j, s)$ for which $\gamma(s') \in P$. Let $\sigma_{j+1} = 1$ if $\gamma$ does not make any crossings from $P$ to $\BB D\setminus \ol U$ after time $\sigma_j$.
 
\item Let $\ol\xi$ be the \emph{last} time that the time-reversal of $\gamma|_{[\sigma_k ,1]}$ completes a crossing from $P$ to $\BB D\setminus U$, or let $\ol\xi$ be the starting time for this time-reversal if it does not make any such crossings. Let $\xi$ be the time for $\gamma$ corresponding to $\ol\xi$ and let $\eta_{n,k}:= \gamma|_{ [\sigma_k ,\xi]}$. Note that by definition, $\eta_{n,k}$ does not hit $P$. 

\item Let $\alpha_{n,k}$ be the concatenation of the time-reversal of $\gamma|_{[\xi,1]}$ and $\gamma|_{[0,\sigma_k]}$.  That is, $\alpha_{n,k}$ is the part of $\gamma$ not traced by $\eta_{n,k}$. Let $\mcl L_n$ be the set of loops of $\Gamma$ which intersect $P$ other than $\gamma$ 
\end{enumerate}

To make the connection to the setting of Section~\ref{sec-cle-annulus}, we observe that (in the terminology of Definition~\ref{def-excursion} and~\ref{def-stub}), the curve $\eta_{n,k}$ is a complementary $P$-excursion of the loop $\gamma$ out of $U$. 
In fact, $\{\eta_{n,k} : n,k \in \BB N\}$ is precisely the set of all complementary $P$-excursions of loops in $\Gamma$ out of $U$. 
The $\sigma$-algebra generated by $\alpha_{n,k}$ and $\mcl L_n$ is the same as the $\sigma$-algebra generated by the set $\Gamma(P;U)$ of loops which intersect $P$ and are contained in $U$, the set $\mcl S_\Gamma(P;U)$ of $P$-excursions into $U$ of loops in $\Gamma$, and the set of all $P$-excursions out of $U$ of loops in $\Gamma$ other than $\eta_{n,k}$.

\begin{lemma}\label{lem:nk}
In the above setting, if we condition on $\alpha_{n,k}$ and $\mcl L_n$, then either $\eta_{n,k}$ is empty ($\alpha_{n,k}$ is itself a loop), or otherwise the conditional law of $\eta_{n,k}$ is that of a chordal $\SLE_\kappa$ between the two endpoints of $\alpha_{n,k}$ in the connected component of $\BB D\setminus \left( \alpha_{n,k} \cup \ol{\bigcup \mcl L_n} \right)$ with these two endpoints on its boundary.  Furthermore, if we condition on $\alpha_{n,k}$, $\mcl L_n$, and $\eta_{n,k}$, then the conditional law of the rest of $\Gamma$ is that of an independent $\CLE_\kappa$ in each of the connected components of $\BB D\setminus \left( \eta_{n,k} \cup \alpha_{n,k} \cup \ol{\bigcup \mcl L_n} \right)$. 
\end{lemma}
 
\begin{proof}

As in the proof of Lemma~\ref{lem:general-markov}, due to conformal invariance, the origin plays no special role. We can therefore assume that $P$ does not contain the origin and denote by $C$ the connected component of $\BB D\setminus \left( \alpha_{n,k} \cup \ol{\bigcup \mcl L_n} \right)$ which contains the origin.
It then suffices to prove that, conditionally on  $\alpha_{n,k}$ and $\mcl L_n$, one has the following. If $C$ does not contain $\eta_{n,k}$ or if $\eta_{n,k}$ is empty, then the conditional law of $\Gamma|_{C}$ is that of a $\CLE_\kappa$ in $C$ independent from $\eta_{n,k}$ and the restrictions of $\Gamma$ to the other connected components of $\BB D\setminus \left( \alpha_{n,k} \cup \ol{\bigcup \mcl L_n} \right)$. Otherwise if $\eta_{n,k} \not=\emptyset$ and $\eta_{n,k} \subset C$, then the conditional law of $\eta_{n,k}$ is that of a chordal $\SLE_\kappa$ between the two endpoints of $\alpha_{n,k}$ in $C$, and if we
further condition on $\eta_{n,k}$, then $\Gamma$ restricted to $\BB D\setminus \left(\eta_{n,k} \cup \alpha_{n,k} \cup \ol{\bigcup \mcl L_n} \right)$ is an independent $\CLE_\kappa$ in each of its connected components.

By Lemma~\ref{lem:strong_markov}, we can explore the loops that intersect $P$ in the order that $P$ intersects them. Define $T$ to be the $(n-1)$st time that $P[0,t]$ intersects a loop that exits $U$ (if $n=1$, then let $T:=0$). Then $T$ is a stopping time for the filtration generated by $\Gamma(P[0,t])$ and therefore Lemma~\ref{lem:strong_markov}  implies that conditionally on $\Gamma(P([0,T]))$, the rest of $\Gamma$ is distributed as an independent $\CLE_\kappa$ in each connected component of $\BB D\setminus \ol{\bigcup \Gamma(P([0,T]))}$ (note that these connected components are simply connected, as explained just after the statement of Lemma~\ref{lem:general-markov}).
Let $O$ be the connected component of $\BB D\setminus \ol{\bigcup \Gamma(P([0,T]))}$ containing the origin.
Then conditionally on $O$ and the restriction of $\Gamma$ in $\BB D\setminus O$, the conditional law of the restriction of $\Gamma$ to $O$ is that of a $\CLE_\kappa$ in $O$.
We will now focus on explaining how to continue exploring $\Gamma|_{\ol O}$.
The basic idea is similar to the proof of Lemma~\ref{lem:general-markov}: we use an inductive procedure (based on Lemmas~\ref{lem:Sheffield} and~\ref{lem-sle-cle-markov}) to define for each $\ep > 0$ a collection of loops and a curve which satisfy the desired Markov property and converge, in an appropriate sense, to $\mcl L_n$ and $\eta_{n,k}$ as $\ep\rta 0$.   

For  $\eps>0$, let $P^\eps$ be the $\eps$-neighborhood of $P$. 
Let $\CL^\eps$ be the collection of all the loops in $\Gamma|_{\ol O}$ that intersect $P^\eps$. Let $\eta_{n,k}^\eps$ be the part of the excursion that we will eventually leave out if it is non-empty (we will give the precise definition of $\eta_{n,k}^\eps$ later on). Let $\alpha^\eps_{n,k}$ be the complement of $\eta_{n,k}^\eps$ in the loop that $\eta_{n,k}^\eps$  is tracing.
Let $\CL^\eps_n$ be $\CL^\eps$ minus the loop containing $\eta_{n,k}^\eps$ (if $\eta_{n,k}^\eps\not=\emptyset$).
Let $C^\eps$ be the connected component of $O\setminus \left( \alpha^\eps_{n,k} \cup \ol{\bigcup \mcl L^\eps_n} \right)$  which contains the origin. 

We would like to first  prove the following statement:
\medskip

\noindent($*$) \emph{Suppose we condition on $\CL_n^\eps$ and $\alpha^\eps_{n,k}$. If $C^\eps$ does not contain $\eta^\eps_{n,k}$ or if $\eta^\eps_{n,k}$ is empty, then the conditional law of $\Gamma|_{C^\eps}$ is that of a $\CLE_\kappa$ in $C^\eps$ independent from  $\eta_{n,k}^\eps$ and the restrictions of $\Gamma$ to the other connected components of $O\setminus \left( \alpha_{n,k}^\eps \cup \ol{\bigcup \mcl L^\eps_n} \right)$. Otherwise if $\eta_{n,k}^\eps \not=\emptyset$ and $\eta_{n,k}^\eps \subset C^\eps$, then the conditional law of $\eta^\eps_{n,k}$ is that of a chordal $\SLE_\kappa$ between the two endpoints of $\alpha^\eps_{n,k}$ in $C^\eps$, and if we
further condition on $\eta^\eps_{n,k}$, then $\Gamma$ restricted to $C^\eps\setminus \left(\eta_{n,k} \cup \alpha_{n,k} \cup \ol{\bigcup \mcl L^\eps_n} \right)$ is an independent $\CLE_\kappa$ in each of its connected components.}
\medskip

We will later prove that as $\eps$ goes to zero,  the sets $\eta^\eps_{n,k}, \alpha^\eps_{n,k}, \ol{\bigcup \CL^\eps_n}, \ol{\bigcup\CL^\eps}$ respectively converge to $\eta_{n,k}, \alpha_{n,k}, \ol{\bigcup\CL_n \cap O}, \ol{\bigcup \CL\cap O}$. This will imply the present lemma.

We will prove ($*$) by performing an exploration process w.r.t.\ $P^\eps$.
Let us define our exploration process by induction on a parameter $i$. Let $O^\eps_0 := O$.
Now suppose $i\in\BB N_0$ and we have completed the $i$ first steps which enable us to define the domain $O^\eps_i\subset O$ with the property that conditionally on $O^\eps_i$, $\Gamma$ restricted to $O^\eps_i$ is a CLE in $O^\eps_i$ which is independent from the restriction of $\Gamma$ to the complement of $O^\eps_i$. Let us explain how to carry out the $(i+1)$st step. See Figure~\ref{fig:sec3-proof} for an illustration.

\begin{figure}[b!]
 \begin{center}
\includegraphics[trim = 0mm 0mm 0mm 45mm, clip, scale=.75]{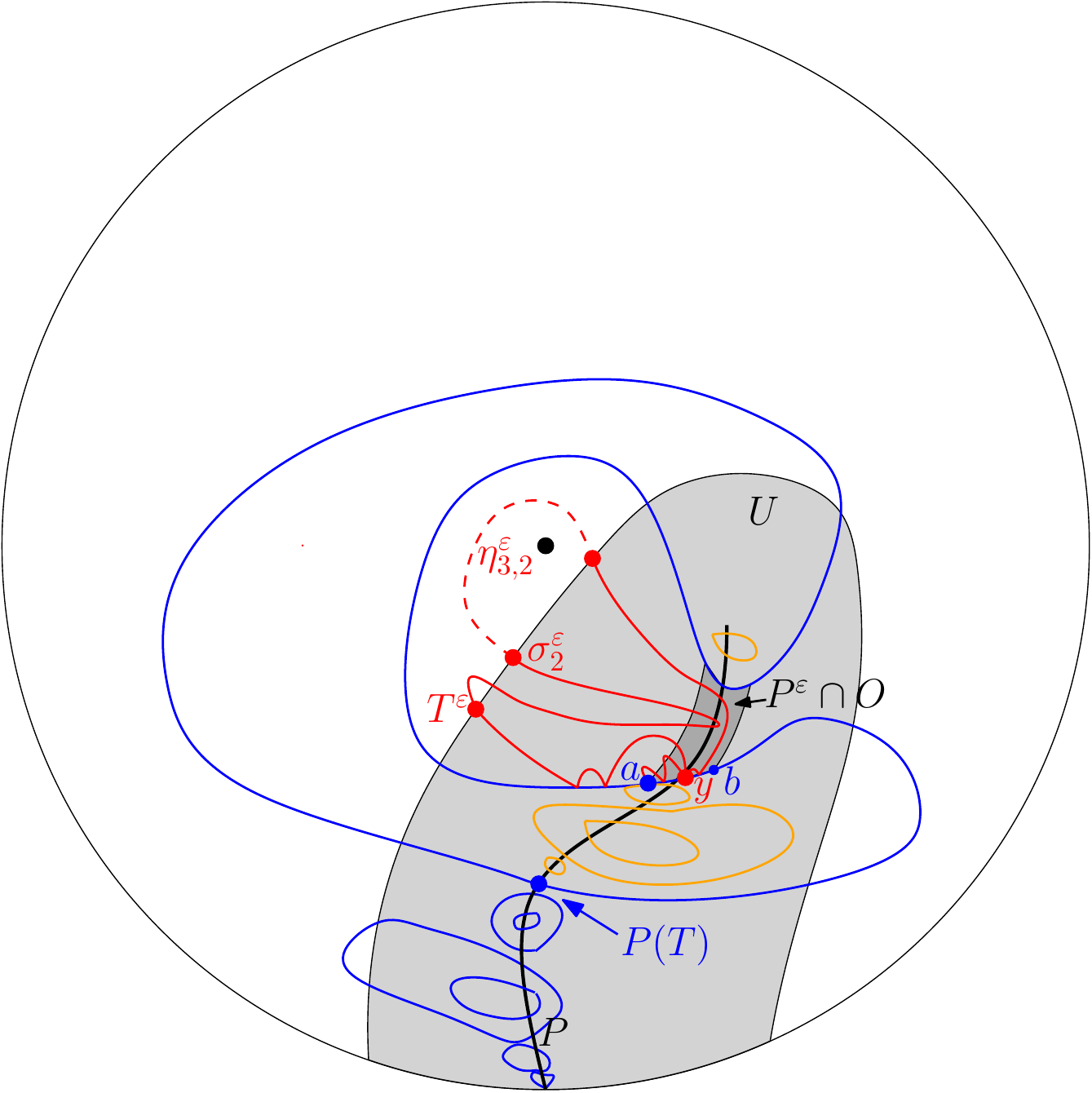}
\vspace{-0.01\textheight}
\caption{We illustrate the exploration process for $n=3, k=2$ in the case where $\eta^\eps_{n,k}\not=\emptyset$ and $\eta^\eps_{n,k}\subset C^\eps$. We first explore along $P$ until it hits the $(n-1)$th loop exiting $U$ at time $T$. All the loops discovered in this step are shown in blue. Then we continue exploring $\Gamma$ restricted to the connected component $O$ of the complement of the blue loops containing the origin. Conditionally on $O$, we are allowed to further condition on  $\Gamma|_{\BB D\setminus O}$ so that the conditional law of $\Gamma|_{\ol O}$ is still a CLE in $O$.
In particular, it is important to acquire the knowledge of all the loops in  $\Gamma|_{\BB D\setminus O}$  that intersect $P$ which have not been previously discovered (drawn in orange).  For the arc $I$ (with endpoints $a,b$) which is the first arc among $\partial O\cap P^\eps$  hit by $P$, if the number of loops exiting $U$ that $P$ has hit has not reached $n$ upon hitting $I$, then we explore along an SLE$_\kappa(\kappa-6)$ process from $a$ to $b$ (shown in red). We depict a case where  $\eta^\eps_{3,2}$ is contained in this step and show it in dashed line.
We show simple loops for clarity, but in actuality the loops will intersect themselves and each other.}

\label{fig:sec3-proof}
\end{center}
\vspace{-1em}
\end{figure}

We call a connected component of $\partial O^\eps_i\cap P^\eps$ an \emph{arc}. There can be countably many arcs of $\partial O_i^\epsilon \cap P^\eps$  but since $\bdy O^\eps_i$ is a continuous curve at most finitely many of them intersect $P$. We can order these finitely many arcs according to the first point on the arc hit by $P$. Let $I$ be the first such arc hit by $P$. 
Given the loops of $\Gamma$ in $\BB D\setminus O^\eps_i$, we know the exact number of loops exiting $U$ that $P$ has hit before hitting $I$. If this number is at least $n$, then it means that $C^\eps$ does not contain $\eta^\eps_{n,k}$ or $\eta^\eps_{n,k}$ is empty. Then we can continue to explore $\Gamma|_{O_i^\eps}$ using the procedure defined in Lemma~\ref{lem:general-markov} w.r.t.\ $P^\eps$. This proves ($*$).
Otherwise, if this number is equal to $n-1$, then let $a$ and $b$ be the endpoints of $I$ (in the counterclockwise direction).

We explore along an $\SLE_\kappa(\kappa-6)$ process $\eta^\eps$ in $O^\eps_i$ from $a$ to $b$, with a marked point immediately to the right of $a$, namely the one constructed by concatenating certain arcs of loops in $\Gamma$ which intersect the arc from $a$ to $b$ as in Lemma~\ref{lem-cle-concatenate}.

\begin{enumerate}[(a)]

\item If $\eta^\eps$ never exits $U$, then define $O^\eps_{i+1}$ to be the connected component containing the origin of the complement in $O^\eps_i$ of all the loops that $\eta^\eps$ has traced.
By the induction hypotheses and Lemma~\ref{lem:Sheffield}, conditionally on $O^\eps_{i+1}$, $\Gamma$ restricted to $O^\eps_{i+1}$ is a CLE in $O^\eps_{i+1}$ which is independent from the restriction of $\Gamma$ to the complement of $O^\eps_{i+1}$.  We can then go on to the $(i+1)$st step.

If in the successive steps $i\in \BB N$, we always end up in situation (a) (hence we can go on infinitely), then it means that $C^\eps$ does not contain $\eta^\eps_{n,k}$ or that $\eta^\eps_{n,k}$ is empty. Therefore, we are in the same situation as in Lemma~\ref{lem:general-markov} and hence ($*$) is true.

\item Otherwise,  let $T^\eps$ be the first time that $\eta^\eps$ exits $U$. Let $\gamma$ be the loop that $\eta^\eps$ is tracing at time $T^\eps$.  When  $\eps$ is small enough, then $\gamma$ is exactly the $n$th loop exiting $U$ that $P$ encounters. To see this, it is enough to show that $\gamma$ is the first loop exiting $U$ that $P$ encounters after hitting $I$. See Figure~\ref{fig:p_eps} for illustration.
Note that there are a.s.\ finitely many loops that intersect $P^\eps$ and exit $U$, hence if $\eps$ is small enough, all the loops intersecting $P^\eps$ exiting $U$ also intersect $P$. Moreover, they a.s.\ all cross $P$. For each of these loops $\omega$, let $\partial\omega$ denote the outer boundary of $\omega$, which is a simple loop. Let $z_\omega$ be the first point that $P$ intersects $\omega$. Then when $\eps$ is small enough, the connected component $I_\omega$ of $\partial\omega\cap P^\eps$ containing $z_\omega$ cuts the tube $P^\eps$, in the sense that it disconnects $P(0)$ and $P(1)$ within $P^\eps$. Therefore, the order in which we discover the different loops that intersect both $P^\eps$ and $\BB D\setminus U$ is the same as the order in which $P$ encounters the corresponding arcs $I_\omega$. In particular, $\gamma$ is indeed the first loop exiting $U$ that $P$ encounters after hitting $I$. 
 
 \begin{figure}[b!]
 \begin{center}
\includegraphics[scale=.85]{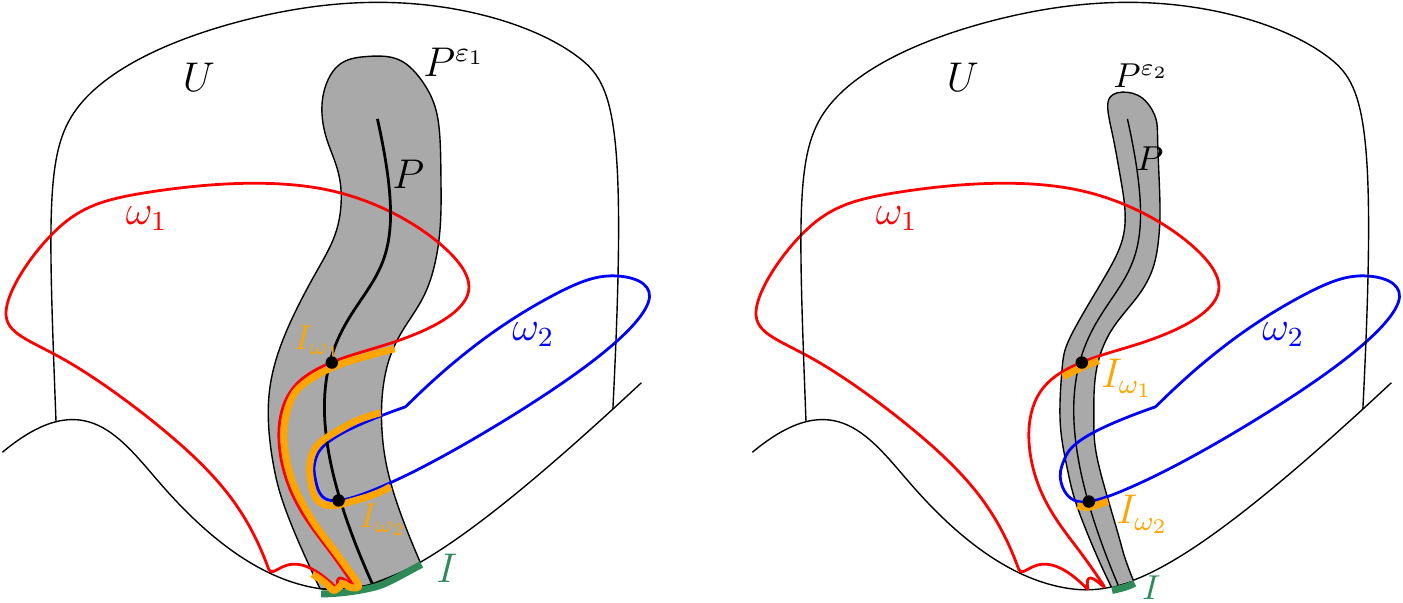}
\vspace{-0.01\textheight}
\caption{\textbf{Left:} $\eps_1$ is not small enough and $I_{\omega_2}$ does not cut $P^{\eps_1}$. When we explore along the green arc $I$, we discover $\omega_1$ before $\omega_2$, which is not the right order. \textbf{Right:} $\eps_2$ is small enough and both $I_{\omega_1}, I_{\omega_2}$  cut $P^{\eps_2}$, hence we will discover $\omega_1$ and $\omega_2$ in the same order as $P$ encounters them.
}
\label{fig:p_eps}
\end{center}
\vspace{-1em}
\end{figure}

If $\eta^\eps[0, T^\eps]$ disconnects the origin from $b$ inside of $O^\eps_i$, then it again means that $C^\eps$ does not contain $\eta^\eps_{n,k}$ or $\eta^\eps_{n,k}$ is empty. Let $O^\eps_{i+1}$ be the connected component containing the origin of $O^\eps_i\setminus \eta^\eps[0, T^\eps]$.  We are again in the same situation as Lemma~\ref{lem:general-markov}, hence ($*$) holds. 

\item Otherwise, if $\eta^\eps[0, T^\eps]$ does not disconnect the origin from $b$, then let $y$ be the marked point of the $\SLE_\kappa(\kappa-6)$ process $\eta^\eps$ at time $T^\eps$. Equivalently, $y$ is also the left-most point on $I$ at which $\gamma$ intersects $I$, where $I$ is the arc of $\bdy O_i^\eps \cap \bdy P^\eps$ which we are currently exploring. Let $\eta^\eps_\gamma$ be the clockwise part of $\gamma$ from $\eta^\eps(T^\eps)$ to $y$, i.e., $\eta^\eps_\gamma$ is the remaining part of $\gamma$ that $\eta^\eps[0,T^\eps]$ has not yet discovered. According to the construction of CLE in~\cite{shef-cle} using branching $\SLE_\kappa(\kappa-6)$ processes, $\eta^\eps_\gamma$ is distributed as an $\SLE_\kappa$ in $O^\eps_i\setminus \eta^\eps[0,T^\eps]$ from $\eta^\eps(T^\eps)$ to $y$. Moreover, conditionally on  $O^\eps_i, \eta^\eps[0, T^\eps]$, and $\eta^\eps_\gamma$, $\Gamma$ restricted to each of the connected components of $O^\eps_i \setminus \!\left(\eta^\eps[0, T^\eps] \cup \eta^\eps_\gamma\right)$  is an independent $\CLE_\kappa$ in that component. We denote by $\Gamma^\eps_{n,k}$ the restriction of $\Gamma$ to 
$O^\eps_i \setminus \left(\eta^\eps[0, T^\eps] \cup \eta^\eps_\gamma\right)$.

We parameterize $\eta^\eps_\gamma$ in a way such that $\eta^\eps_\gamma(0)=\eta^\eps(T^\eps)$ and $\eta^\eps_\gamma(1)=y$.
If $\eta^\eps_\gamma$ makes at most $k-2$ crossings from $P^\eps$ to $\BB D\setminus U$, then let $\eta^\eps_{n,k} = \emptyset$. 
Then we are again in the same situation as Lemma~\ref{lem:general-markov}, hence ($*$) holds. 
Otherwise, let $\sigma^\eps_k$ be the $(k-1)$st time that $\eta^\eps_\gamma$ completes a crossing from $P^\eps$ to $\BB D \setminus U$. 
Then conditionally on $O^\eps_i$ and on $\eta^\eps|_{[0,T^\eps]} , \eta^\eps_\gamma|_{[0,\sigma^\eps_k]}$, the process $\eta^\eps_\gamma|_{[\sigma^\eps_k,1]}$ is an $\SLE_\kappa$ in $O^\eps_i \setminus \!\left(\eta^\eps[0, T^\eps] \cup \eta^\eps_\gamma[0, \sigma^\eps_k]\right)$.
Let $\wt \eta^\eps_\gamma$ be the time-reversal of $\eta^\eps_\gamma$. Let $\ol\xi^\eps$ be the last time that $\wt \eta^\eps_\gamma$ completes a crossing from $P$ to $\BB D\setminus U$. Let $\xi^\eps$ be the time for $\eta^\eps_\gamma$ corresponding to $\ol\xi^\eps$ and let $\eta^\eps_{n,k}:= \eta^\eps_\gamma|_{[\sigma^\eps_k,\xi^\eps]}$.
Let $O^\eps_{i+1}:=O^\eps_i \setminus  \left(\eta^\eps[0, T^\eps] \cup \eta^\eps_\gamma[0, \sigma^\eps_k] \cup \eta^\eps_\gamma[\xi^\eps,1]\right)$.
Then conditionally on $O^\eps_{i+1}$,  the curve $\eta^\eps_{n,k}$ is an $\SLE_\kappa$ in $O^\eps_{i+1}$ conditioned to avoid $P^\eps$.
Therefore, $(\eta^\eps_{n,k}, \Gamma^\eps_{n,k})$ is distributed as an $\SLE_\kappa$ decorated $\CLE_\kappa$ in $O^\eps_{i+1}$ where the $\SLE_\kappa$ curve is conditioned to avoid $P^\eps$ i.e., $(\eta^\eps_{n,k}, \Gamma^\eps_{n,k})$ has the law considered in Section~\ref{sec-sle-cle-markov} with $B = P^\ep$.

Using Lemma~\ref{lem-sle-cle-markov}, we can then continue to explore $(\eta^\eps_{n,k}, \Gamma^\eps_{n,k})$ along any arc on the boundary of $O^\eps_{i+1}$. 
More precisely, if $\partial O^\eps_{i+1}\cap P^\eps$ is non-empty, then we can discover all the loops in $O^\eps_{i+1}$ that intersect $\partial O^\eps_{i+1}\cap P^\eps$ and denote by $O_{i+2}^\epsilon$ the connected component containing the origin of the complement in $O^\eps_{i+1}$ of all the newly discovered loops.
Since $\eta^\eps_{n,k}$ is conditioned to avoid $P^\eps$, it will also avoid all the arcs of $\partial O^\eps_{i+1}\cap P^\eps$.
By Lemma~\ref{lem-sle-cle-markov}, conditionally on $O_{i+2}^\epsilon$, the restriction of $(\eta^\eps_\gamma, \Gamma_{n,k})$ to $O_{i+2}^\epsilon$ is still an $\SLE_\kappa$ decorated $\CLE_\kappa$, where the $\SLE_\kappa$ is conditioned to avoid $P^\eps$.  We can then iterate this process until some step $N\in\BB N \cup \{\infty\}$ such that  $\partial O_N^\epsilon \cap P^\eps$ is empty. If $N=\infty$, then we define $O_\infty^\epsilon$ to be the interior of $\bigcap_{n=1}^\infty O_n^\epsilon$. It then follows that conditionally on $O_N$, the curve $\eta^\eps_{n,k}$ is an $\SLE_\kappa$ in $O_N^\epsilon$ conditioned to avoid $P^\eps$ and that if we further condition on $\eta^\eps_{n,k}$, then $\Gamma$ restricted to each of the connected components of $O_N\setminus\eta^\eps_{n,k}$ is distributed as an independent $\CLE_\kappa$ in that component.
It is clear that when $N<\infty$, $O_N^\epsilon$ is exactly the connected component containing the origin of $\BB D\setminus \left(\alpha^\eps_{n,k} \cup \ol{\bigcup \CL^\eps_n}\right)$. Similar arguments as in Lemma~\ref{lem:general-markov} imply that the same is true when $N=\infty$.
\end{enumerate}

Now that we have proved ($*$), we will send $\eps$ to $0$. The fact that $\ol{\bigcup\Gamma(P^\ep ) \cap O}$ converges to $\ol{\bigcup \Gamma(P) \cap O}$ follows from the same arguments as in Lemma~\ref{lem:general-markov}.  
We have also argued in (b) that for $\eps$ small enough, the $n$th loop $\gamma$ that exits $U$ in the $P^\eps$ exploration process indeed coincides with the $n$th loop that exits $U$ that $P$ encounters. Therefore $\ol{\bigcup \CL^\eps_n}$ converges to $\ol{\bigcup \CL_n \cap O}$.
For similar reasons, $\eta^\eps_{n,k}$ will also coincide with $\eta_{n,k}$ for $\eps$ small enough, since any loop a.s.\ makes finitely many crossings from $P^\eps$ to $\BB D\setminus U$ and any such crossing that intersects $P$ also crosses $P$.  This completes the proof.
\end{proof}

As a consequence of Lemma~\ref{lem:nk}, we obtain the following variant of the annulus Markov property for $\CLE_\kappa$ in $\BB D$. For the statement, we recall the notation from Section~\ref{sec-cle-annulus}.

\begin{definition}\label{def:markovUP}
Define the path $P$ and the open set $U$ as in the beginning of this subsection.
Choose a $P$-excursion into $U$ from $\mcl S_\Gamma(P;U)$ in a manner which is measurable w.r.t.\ $\sigma(\Gamma(P;U), \mcl S_\Gamma(P;U))$. Let $x$ be its terminal endpoint and let $\eta_x$ be the complementary $P $-excursion out of $U$ from $x$ to the corresponding endpoint $x^*$. 
Let $\Sigma_x$ be the $\sigma$-algebra generated by $\Gamma( P; U)$, $\mcl S_\Gamma(P; U)$, and all of the complementary $ P$-excursions of loops in $\Gamma$ out of $ U$ except for $\eta_x$.
We say that $\Gamma$ satisfies \emph{Markov property w.r.t. $(P, U)$} if the following is true: 
\begin{enumerate}
\item \label{itm0} Almost surely, $\ol{\bigcup \Gamma(P;U)}$ is connected.

\item\label{itm1} If $\mcl S_\Gamma(P;U) \not=\emptyset$ and we condition on $\Sigma_x$, then the conditional law of $\eta_x$ is that of an independent chordal $\SLE_\kappa$ from $x$ to $x_*$ in the connected component of $\BB D\setminus \ol{\bigcup \Gamma(P)\setminus \eta_x}$ with $x$ on its boundary.  

\item\label{itm2} If we further condition on $\eta_x$ (equivalently, we condition on $\Gamma(P)$) then the conditional law of $\Gamma|_{\BB D\setminus \ol{\bigcup \Gamma(P)}}$ is that of a collection of independent $\CLE_\kappa$'s in the connected components of $\BB D\setminus \ol{\bigcup \Gamma(P)}$. 
\end{enumerate}

\end{definition}

As in Definition~\ref{def-annulus-markov}, the purpose of condition~\ref{itm0} in Definition~\ref{def:markovUP} is to ensure that the connected components involved in conditions~\ref{itm1} and~\ref{itm2} are simply connected, so it makes sense to talk about CLE$_\kappa$ in these connected components.

\begin{corollary}\label{cor:markov}
Let $\Gamma$ be a $\CLE_\kappa$ on $\BB D$. Then $\Gamma$ satisfies the Markov property w.r.t.\ $(P, U)$.
 \end{corollary}
\begin{proof}
By Lemma~\ref{lem-cle-cover}, condition~\ref{itm0} of Definition~\ref{def:markovUP} is satisfied.
We observe that if $n,k \in \BB N$ and $\eta_{n,k}$ is the curve defined in the $(n,k)$-exploration process for $(P,U)$, then $\eta_{n,k}$ is a complementary $P$-excursion out of $U$ for some loop in $\Gamma$. Furthermore, for any $\sigma(\Gamma(P;U), \mcl S_\Gamma(P;U))$-measurable choice of $x$ as in Definition~\ref{def:markovUP}, the event $\{\eta_x = \eta_{n,k}\}$ is $\Sigma_x$-measurable.  
If we fix $n,k \in \BB N$, then by Lemma~\ref{lem:nk}, if we condition on $\Sigma_x$ and the event $\{\eta_x = \eta_{n,k}\}$, then the properties~\ref{itm1} and~\ref{itm2} in Definition~\ref{def:markovUP} are satisfied.
Since each complementary $P$-excursion of $\Gamma$ out of one of $U$ is one of the $\eta_{n,k}$'s, this concludes the proof. 
\end{proof}

\subsection{Annulus Markov property: proof of Theorem~\ref{thm-cle-markov}}
\label{sec-annulus-markov}

\begin{figure}[h]
\centering
\includegraphics[width=0.85\textwidth]{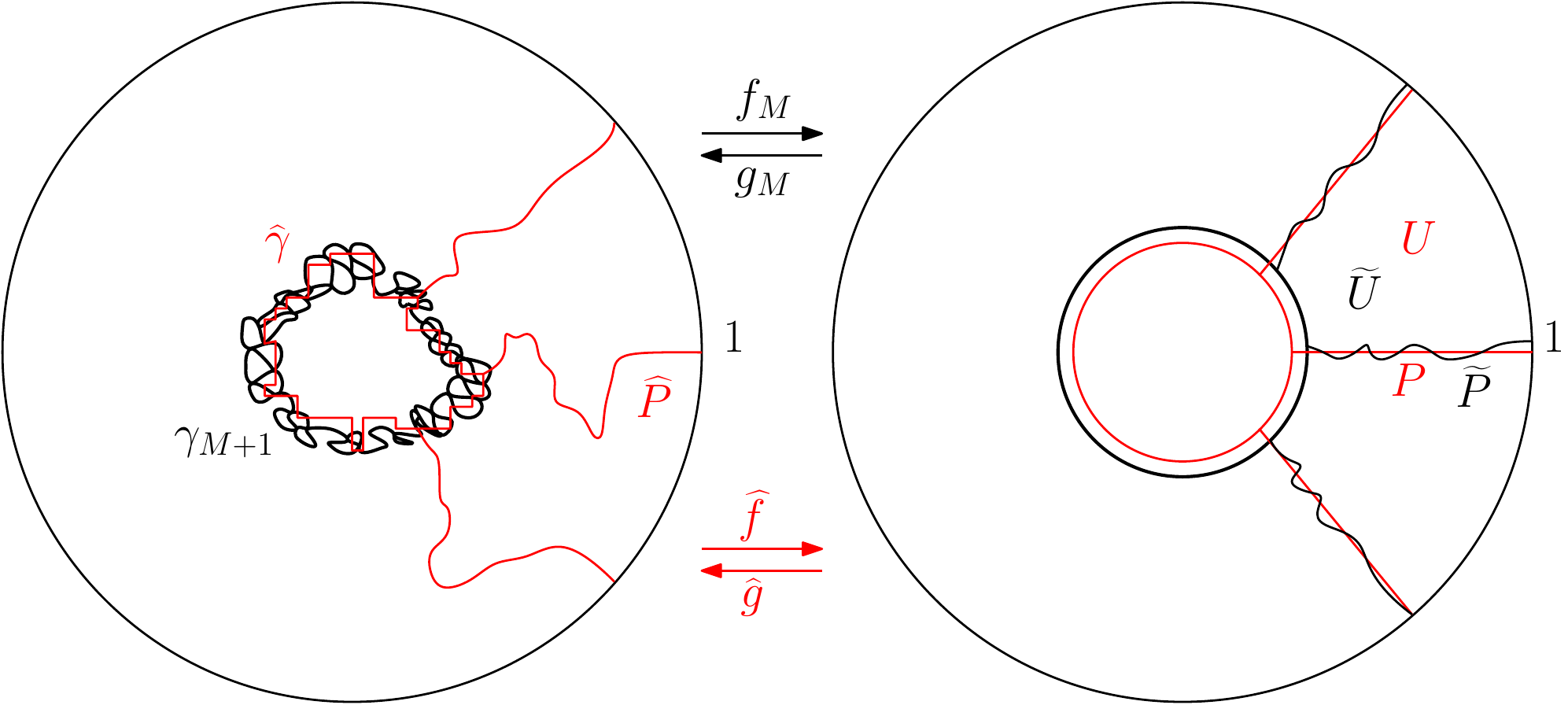}
\caption{Illustration of the proof of the annulus Markov property. We depict in black the loop $\gamma_{M+1}$ from the CLE and the conformal map $f_M$ from the doubly connected component of $\BB D\setminus \gamma_{M+1}$ onto an appropriate $\BB A_\rho$ and its inverse $g_M$. We depict in red the deterministic loop $\wh\gamma$ (which is supposed to approximate the outer boundary of $\gamma_{M+1}$) as well as the conformal map $\wh f$ from the doubly connected component of $\BB D\setminus \wh\gamma$ onto an appropriate $\BB A_{\wh\rho}$ and its inverse $\wh g$. We also depict the corresponding images of $U,P$ under successive conformal maps as defined in the text.}
\label{fg:UP}
\end{figure}

Recall that $\gamma_{M+1}$ is the $(M+1)$st outermost loop in $\Gamma$ surrounding~$0$, $\mcl A_M$ is the non-simply-connected component of $\BB D\setminus \gamma_{M+1}$, and $f_M : \mcl A_M \rta \BB A_\rho$ is a conformal map to an appropriate annulus.  
Let $(P,U)$ be as in the annulus Markov property. 
Recall that $U = \{r e^{i s} : r\in (\rho,1) , s \in (\theta-\pi/4 ,\theta+\pi/4)\}$. 

The idea of the proof is to apply Lemma~\ref{lem:nk} to the pair $(\wt P, \wt U)$ where $\wt P:=f_M^{-1}(P)$ and $\wt U:=f_M^{-1}(U)$.
The main difficulty is that $(\wt P, \wt U)$ is random: it depends on $\gamma_{M+1}$. 

To get around this difficulty, we will condition $\gamma_{M+1}$ to stay close to some deterministic $\eps$-lattice loop and argue as in the proof of the usual strong Markov property for stopping times.

Let $ \gamma_{M+1}^\eps$ be the outer boundary of the closure of the union of all $\eps$-lattice squares (i.e., squares with corners in $\epsilon \BB Z^2$) that are entirely contained in the domain encircled by the outer boundary of $\gamma_{M+1}$. 
Then $\gamma_{M+1}^\eps$ is a simple loop  that surrounds the origin. 
Let $\wh\gamma\subset{\BB D}$ be a deterministic loop encircling the origin which is the outer boundary of some union of connected $\eps$-lattice squares. 
Since there are only finitely many possible choices for $\gamma_{M+1}^\eps$, it holds for some choice of $\wh\gamma$ that $\BB P[\gamma_{M+1}^\eps = \wh\gamma] > 0$.
Let $\wh A$ be the annulus between $\wh \gamma$ and $\partial \BB D$ and let $\wh\rho$ be its conformal modulus. Let $\wh f$ be the conformal map from $\wh A$ onto $\BB A_{\wh \rho}$ which fixes 1.  
Let $\wh P:=\wh f^{-1}(P)$. 

By Corollary~\ref{cor:markov}, we know that $\Gamma$ satisfies the Markov property for the pair $(\wh P, \wh U)$. That is, a.s.\ $\wh P \subset \ol{\bigcup \Gamma(\wh P ,\wh U)}$ and if $x$ is the terminal endpoint of an element of $\mcl S_\Gamma(\wh P;\wh U)$ chosen in a measurable manner w.r.t.\ $\sigma(\Gamma(\wh P;\wh U), \mcl S_\Gamma(\wh P;\wh U))$, then conditionally on the $\sigma$-algebra $\Sigma_x$ of Definition~\ref{def:markovUP} for $(\wh P , \wh U)$, $\eta_x$ and $\Gamma|_{\BB D\setminus \ol{\bigcup\Gamma(\wh P)}}$ satisfy the properties~\ref{itm1} and~\ref{itm2} in Definition~\ref{def:markovUP}.  
Based on this, we will successively deduce the following properties for $\Gamma$:
\begin{enumerate}
\item Note that the event $\{\eta_x\not\subset \gamma_{M+1}\} \cap \{\gamma_{M+1}^\eps = \wh\gamma\}$ is measurable w.r.t.\  $\Sigma_x$, since it is equivalent to the event that among all the discovered loops, there exists a loop $\gamma$ such that $ \gamma^\eps =\wh\gamma$ and that $\gamma$ is the $(M+1)$st loop that surrounds the origin (this is determined by $\Sigma_x$, since one can see from the information in $\Sigma_x$ whether the loop containing $\eta_x$ encircles the origin based on the location of the endpoints of $\eta_x$).
Therefore, if we  condition on $\Sigma_x$ then on the event $\{\eta_x\not\subset \gamma_{M+1}\} \cap \{\gamma_{M+1}^\eps = \wh\gamma\}$, the conditional laws of $\eta_x$ and $\Gamma|_{\BB D\setminus \ol{\bigcup\Gamma(\wh P)}}$ still satisfy the properties~\ref{itm1} and~\ref{itm2} in Definition~\ref{def:markovUP}.

\item On the event $  \{\eta_x\not\subset \gamma_{M+1}\} \cap\{ \gamma^\eps_{M+1}  = \wh\gamma\}$, $\gamma_{M+1}$ is measurable w.r.t.\ $\Sigma_x$.
Therefore, if we condition on $\gamma_{M+1}, \Sigma_x$ and on the event $\{\eta_x\not\subset \gamma_{M+1}\} \cap \{\gamma_{M+1}^\eps = \wh\gamma\}$, the conditional laws of $\eta_x$ and $\Gamma|_{\BB D\setminus \ol{\bigcup\Gamma(\wh P)}}$ still satisfy properties~\ref{itm1} and~\ref{itm2} in Definition~\ref{def:markovUP}.

\item \label{itm:annulus}
We now change the order of conditioning and get the following statement. Conditionally on $\gamma_{M+1}$, on the event $\{\gamma_{M+1}^\eps = \wh\gamma\}$, the loop ensemble $\Gamma|_{\mcl A_M}$ satisfies the following property:  For any $\wh P$-excursion in  $\mcl S_\Gamma(\wh P;\wh U)$ chosen in  a measurable manner w.r.t.\ $\sigma(\Gamma(\wh P; \wh U), \mcl S_\Gamma(\wh P;\wh U))$ which does not trace a part of  $\gamma_{M+1}$, if its terminal endpoint is $x$ and we further condition on $\Sigma_x$, then  the conditional laws of $\eta_x$ and $\Gamma|_{\BB D\setminus \ol{\bigcup\Gamma(\wh P)}}$ still satisfy properties~\ref{itm1} and~\ref{itm2} in Definition~\ref{def:markovUP}. 

\item We condition on the loop $\gamma_{M+1}$ and the event $\{\gamma_{M+1}^\eps = \wh\gamma\}$ throughout the current paragraph.
Let $\wt P^\eps :=f_M(\wh P)$, $\wt U^\eps :=f_M(\wh U)$, and $\wt \Gamma  := f_M(\Gamma|_{\mcl A_M})$. 
For any $\wt P^\eps$-excursion in  $\mcl S_{\wt \Gamma^\eps}(\wt P^\eps;\wt U^\eps))$ chosen in a measurable manner w.r.t.\ $\sigma(\gamma_{M+1}  , \wt\Gamma(\wt P^\eps; \wt U^\eps), \mcl S_{\wt\Gamma}(\wt P^\eps ;\wt U^\eps ))$ with terminal endpoint $\wt x^\eps$, let $\wt \eta_{\wt x^\eps}^\eps$ be the complementary $\wt P^\eps$-excursion out of $\wt U^\eps$ started from $\wt x^\eps$. 
Note that $\sigma(\gamma_{M+1}  , \wt\Gamma(\wt P^\eps; \wt U^\eps), \mcl S_{\wt\Gamma}(\wt P^\eps ;\wt U^\eps ))$ is contained in $\sigma(\gamma_{M+1} ,  \Gamma(\wh P ; \wh U ) , \mcl S(\wh P ; \wh U ) )$.
Moreover, every $\wt P^\eps$-excursion in  $\mcl S_{\wt \Gamma^\eps }(\wt P^\eps ;\wt U^\eps))$ is the image under $f_M$ of some $\wh P^\eps$-excursion in $\mcl S_\Gamma(\wh P^\eps;\wh U^\eps))$ which does not trace a part of  $\gamma_{M+1}$.
Therefore, for any $\wt x^\eps$ and $\wt\eta_{\wt x^\eps}^\eps$ chosen as before, there exist a $\wh P$-excursion $ \mcl S_\Gamma(\wh P;\wh U)$ which is measurable w.r.t.\  $\sigma( \Gamma(\wh P; \wh U) , \mcl S(\wh P; \wh U) )$, with terminal endpoint $x$ and corresponding complementary $\wh P$-excursion $\eta_x$ such that  $\wt\eta_{\wt x^\eps}^\eps =f_M(\eta_x)$. Moreover, $e$ does not trace a part of  $\gamma_{M+1}$.

Now, if we apply $f_M$ to the objects in the statement in~\ref{itm:annulus}, then we can deduce the following statements for  $\wt\Gamma^\eps$.
Almost surely, $\wt P^\eps \subset \ol{\bigcup \wt\Gamma^\eps(\wt P^\eps,\wt U^\eps)}$. Moreover, conditionally on $\gamma_{M+1}$ and $\{\gamma_{M+1}^\eps = \wh\gamma\}$, for any $\wt P$-excursion in  $\mcl S_{\wt\Gamma^\eps}(\wt P^\eps;\wt U^\eps))$ chosen in a measurable manner w.r.t.\ $\sigma(\wt\Gamma^\eps(\wt P^\eps ; \wt U^\eps ), \mcl S_{\wt\Gamma^\eps}(\wt P^\eps;\wt U^\eps))$ with terminal endpoint $\wt x^\eps$, if we further condition on $\Sigma_{x}$,  the conditional laws of $\wt\eta_{\wt x^\eps}^\eps$ and $\wt\Gamma^\eps|_{\BB A_\rho\setminus \ol{\bigcup\wt\Gamma^\eps(\wt P^\eps)}}$ still satisfy properties~\ref{itm1} and~\ref{itm2} in Definition~\ref{def:markovUP}.  

The last step is to replace $\Sigma_x$ in the above statement by $\wt\Sigma_{\wt x^\eps}^\eps$, which is defined to be the sigma algebra generated by $\wt\Gamma^\eps( \wt P^\eps ; \wt U^\eps)$, $\mcl S_{\wt\Gamma^\eps}(\wt P^\eps; \wt U^\eps)$, and all of the complementary $ \wt P^\eps$-excursions of loops in $\wt\Gamma^\eps$ out of $ \wt U^\eps$ except for $\wt\eta_{\wt x^\eps}^\eps$.

Note that on the event $\{\gamma_{M+1}^\eps = \wh\gamma\}$, we have $\wt\Sigma_{\wt x^\eps}^\eps=\sigma(\Sigma_x, \gamma_{M+1})$. 
Therefore, we can replace $\Sigma_x$ by $\wt\Sigma_{\wt x^\eps}^\eps$  in the above statement.

We have therefore proved that conditionally on $\{\gamma_{M+1}^\eps = \wh\gamma\}$ and $\gamma_{M+1}$,  $f_M(\Gamma|_{  \mcl A_M})$ satisfies the annulus Markov property  for the pair $(\wt P, \wt U)$. 

\item
Let $g^\eps_M$ be the conformal map from $\BB A_{\rho^\eps}$ onto the doubly connected component of $\BB D\setminus \gamma_{M+1}^\eps$ which fixes $1$ where $\rho^\eps$ is the conformal radius of $\BB D\setminus \gamma_{M+1}^\eps$. Note that $g^\eps_M$ is a.s.\ determined by $\gamma_{M+1}$.
Now, if we look at the union of $\{\gamma_{M+1}^\eps = \wh\gamma\}$ for all $\wh\gamma$, then we get that conditionally on $\gamma_{M+1}$, $f_M(\Gamma|_{ \mcl A_M})$ satisfies the annulus Markov property for the pair $(f_M( g_M^\eps(P) ), f_M( g_M^\eps(U) ))$. As $\eps$ goes to zero, the pair $(f_M( g_M^\eps(P) ), f_M( g_M^\eps(U) ))$ converges to $(P, U)$, since $f_M\circ g^\eps_M $ converges uniformly to the identity.
Therefore, conditionally on $\gamma_{M+1}$, $f_M(\Gamma|_{ \mcl A_M})$ also satisfies the annulus Markov property for the pair $(P,U)$. 
Since this annulus Markov property itself does not depend on $\gamma_{M+1}$,  we can take away the conditioning and we get that $f_M(\Gamma|_{ \mcl A_M})$ satisfies the annulus Markov property for the pair $(P,U)$. \qed
\end{enumerate}

\section{The annulus Markov property uniquely characterizes CLE}
\label{sec-resampling}

In the preceding section, we showed that the construction of Theorem~\ref{thm-cle-markov} gives a loop ensemble on $\BB A_\rho$ with $M$ inner-boundary-surrounding loops which satisfies the annulus Markov property.
By Lemma~\ref{lem-miller-wu-dim} and the branching $\SLE_\kappa(\kappa-6)$ construction of $\CLE_\kappa$, we see that $\rho$ has positive probability to lie in any fixed open subset of $(0,1)$. By considering the regular conditional law given $\rho$ of the loop ensemble of Theorem~\ref{thm-cle-markov}, we get the existence part of Theorem~\ref{thm-cle-annulus} for a dense set of $\rho \in (0,1)$. 
By slightly perturbing the inner loop, it is easily seen that this regular conditional law depends continuously on $\rho$, so we can take limits to get the existence part of Theorem~\ref{thm-cle-annulus} in general.

The goal of this section is to establish the uniqueness part of Theorem~\ref{thm-cle-annulus}. 
To do this we will consider a Markov chain based on the annulus Markov property of Definition~\ref{def-annulus-markov}.
A law on loop ensembles satisfying the annulus Markov property will be a stationary measure for the Markov chain. 
We will then argue that the Markov chain has a unique stationary measure as follows.
We will show (Proposition~\ref{prop-coupling}) that the Markov chains started from any two initial configurations can be coupled together so that they agree with positive probability after finitely many steps. 
This will imply in particular that two stationary measures cannot be mutually singular. 
General ergodic theory considerations (as explained in Section~\ref{sec-stationary}) will then lead to the uniqueness of the stationary measure. 

A similar idea (but with a simpler Markov chain) is used in~\cite{ig2} to deduce the reversibility of $\SLE_\kappa(\rho_1;\rho_2)$ with $\rho_1,\rho_2 > -2$ from the reversibility of $\SLE_\kappa(\rho)$ for $\kappa \in (0,4)$ and $\rho > -2$. See also~\cite[Appendix A]{msw-sle-range} for an extension of this result proven using the same basic technique. 
 
Let us now define the Markov chain we will consider.  
Let $P_+ := [\rho,1]$ and let $P_- := [-1,-\rho]$. Also define the annular slices
\alb
U_+ &:= \{r e^{i\phi} : r\in (\rho,1) , \phi \in (-\pi/4 , \pi/4) \} \quad \text{and} \\
U_- &:= \{r e^{i\phi} : r\in (\rho,1) , \phi \in (3\pi/4 , 5\pi/4) \} .
\ale

The state space of our Markov chain will be the space of non-crossing, locally finite loop configurations on $\BB A_\rho$ which have exactly\footnote{Throughout most of this subsection we assume that $M\geq 1$ for convenience, which implies in particular that $\mcl S_\Gamma(P_\pm ; U_\pm)\not=\emptyset$. The case when $M = 0$ can be treated by a similar, but simpler argument.} $M \in \BB N $ inner-boundary-surrounding loops. 
Given such a loop configuration $\Gamma_0$, we define a new loop configuration $\Gamma_1$ as follows.
\begin{enumerate}
\item Sample a sign $\xi$ uniformly at random from $\{-,+\}$. 
\item Condition on $\xi$ and choose a uniformly random $P_\xi$-excursion into $U_\xi$ from the set $\mcl S_{\Gamma_0} (P_\xi; U_\xi)$. 
Let $x$ be its terminal endpoint. 
\item Condition on $\xi$ and $x$ and let $ \eta '$ be an independent chordal $\SLE_\kappa$ in the connected component of
\eqb \label{eqn-markov-chain-union}
\BB A_\rho \setminus  \ol{\bigcup \left( \Gamma_0(P_\xi) \setminus \{\gamma_x\} \right) \cup   \alpha_x} 
\eqe 
which has $x$ on its boundary, where here $\eta_x$ is the complementary $P_\xi$-excursion out of $U_\xi$ starting from $x$, $\gamma_x$ is the loop which contains $\eta_x$, and $\alpha_x$ is the complementary arc of $\eta_x$ in $\gamma_x$. 
\item The set of loops $\Gamma_1(P_\xi)$ is defined to be the same as $\Gamma_0(P)$ except that the loop segment $\eta_x$ is replaced by $ \eta '$. 
\item Conditioned on $\Gamma_1(P_\xi)$, sample the rest of $\Gamma_1$ by sampling an independent $\CLE_\kappa$ in each connected component of $\BB A_\rho \setminus \Gamma_1(P_\xi)$. 
\end{enumerate}
By definition, a probability measure on non-crossing, locally finite loop configurations that satisfies the annulus Markov property and has $M$ inner-boundary-surrounding loops is a stationary measure for the above Markov chain. 
Hence to prove the uniqueness part of Theorem~\ref{thm-cle-annulus} we only need to establish the following.

\begin{prop} \label{prop-stationary}
The above Markov chain has a unique stationary measure on locally finite, non-crossing loop configurations on $\BB A_\rho$. 
\end{prop}

To prove Proposition~\ref{prop-stationary}, fix two initial loop configurations $\Gamma_0$ and $\wt\Gamma_0$ (each of which is a deterministic, non-crossing, locally finite loop configuration on $\BB A_\rho$ with $M$ inner-boundary-surrounding loops)  and let $\{\Gamma_n\}_{n\in\BB N_0}$ and $\{\wt\Gamma_n\}_{n\in\BB N_0}$ be the Markov chains started from $\Gamma_0$ and $\wt\Gamma_0$, respectively. For $n\in\BB N$, we denote the objects in the definition of the Markov chain above with $\Gamma_{n-1}$ in place of $\Gamma_0$ and $\Gamma_n$ in place of $\Gamma_1$ with a subscript $n$ (so, e.g., $\xi_n \in \{-,+\}$ and $\eta_n'$ is the chordal $\SLE_\kappa$ curve above). We make a similar convention for $\{\wt\Gamma_n\}_{n\in\BB N}$ except that we also add a tilde to the notation. The main step in the proof of Proposition~\ref{prop-stationary}, and hence the uniqueness part of Theorem~\ref{thm-cle-annulus}, is the following statement. 
 
\begin{prop} \label{prop-coupling}
For any choice of initial configurations $(\Gamma_0, \wt\Gamma_0)$, there exists $n \in\BB N$ and a coupling of $\Gamma_n$ and $\wt\Gamma_n$ for which $\BB P[\Gamma_n = \wt\Gamma_n] > 0$. 
\end{prop}

We will explain how to extract Proposition~\ref{prop-stationary} from Proposition~\ref{prop-coupling} in Section~\ref{sec-stationary}. 

The basic idea of the proof of Proposition~\ref{prop-coupling} is to use the absolute continuity statements for SLE and CLE from Appendix~\ref{sec-sle-cle-lemmas} to couple together larger and larger pieces of $\Gamma_n$ and $\wt\Gamma_n$ with positive probability. This will be carried out in two steps. In Section~\ref{sec-simple-coupling}, we treat the case when the topology of $\Gamma_0$ and $\wt\Gamma_0$ is particularly simple: we require that all of the loops which intersect $P_+$ except for the inner-boundary-surrounding loops are contained in a neighborhood $U_+'$ of $U_+$ and the inner-boundary-surrounding loops each make only one excursion out of this neighborhood.
In Section~\ref{sec-full-coupling}, we reduce the general case to this case by using Lemma~\ref{lem-miller-wu-dim} to ``pull" the excursions which get far from $U_+$ back to $U_+'$ one at a time. See the start of each of the individual subsections for a more detailed overview of the arguments involved.

Before proceeding with the proof, we record the following basic topological lemma.

\begin{lem} \label{lem-excursion-count}
Let $\gamma $ be an arbitrary loop in $\BB A_\rho$ (not necessarily non-self-crossing). The following quantities are equal.
\begin{enumerate}
\item The number of crossings of $\gamma$ from $P_+$ to $P_-$. \label{item-down-crossing}
\item The number of crossings of $\gamma$ from $P_-$ to $P_+$.\label{item-up-crossing}
\item The number of complementary $P_+$-excursions of $\gamma$ out of $U_+$ which hit $P_-$. \label{item-down-excursion}
\item The number of complementary $P_-$ excursions of $\gamma$ out of $U_-$ which hit $P_+$. \label{item-up-excursion}
\end{enumerate}  
\end{lem}
\begin{proof}
The quantities~\ref{item-down-crossing} and~\ref{item-up-crossing} are equal since $\gamma$ is a loop.
The quantities~\ref{item-down-crossing} and~\ref{item-down-excursion} are equal since the concatenation of a complementary $P_+$-excursions of $\gamma$ out of $U_+$ and the $P_+$-excursion of $\gamma$ into $U_+$ immediately preceding it contains exactly one crossing from $P_+$ to $P_-$. 
Similarly, the quantities~\ref{item-up-crossing} and~\ref{item-up-excursion} are equal.
\end{proof}

\subsection{Initial configurations with simple topology}
\label{sec-simple-coupling}

In this subsection, we will establish Proposition~\ref{prop-coupling} in a special case when the topology of the initial configurations $\Gamma_0$ and $\wt\Gamma_0$ are particularly simple. 
We will need to work with a slightly larger annular slice which contains $U_+$ (the place where this is needed is Lemma~\ref{lem-base-coupling} below). 
To be concrete, we set
\eqb \label{eqn-bigger-slices}
U_+' := \{r e^{i\phi} : r\in (\rho,1) , \phi \in (-\pi/3 , \pi/3) \} . 
\eqe 
The main result of this subsection is the following proposition. 

\begin{prop} \label{prop-simple-coupling}
Suppose our initial configurations are such that $\#\mcl S_{\Gamma_0}(P_+ ; U_+') = \#\mcl S_{\wt\Gamma_0}(P_+ ; U_+') = M$. 
There is a coupling of $\{\Gamma_n\}_{n\in\BB N_0}$ and $\{\wt\Gamma_n\}_{n\in\BB N_0}$ such that $\BB P[\Gamma_{2M} = \wt\Gamma_{2M}] > 0$.
\end{prop}

Since each of the $M$ loops of $\Gamma_0$ which surround the inner boundary must have at least one complementary $P_+$-excursion out of $U_+'$, we always have $\#\mcl S_{\Gamma_0}(P_+ ; U_+') \geq M$. The hypothesis that $\#\mcl S_{\Gamma_0}(P_+ ; U_+') = M$ says that none of the loops of $\Gamma_0$ which intersect $P_+$ other than the inner-boundary-surrounding loops exit $U_+'$. Furthermore, each of the inner boundary surrounding loops has exactly one complementary $P_+$-excursion out of $U_+'$.  
Similar considerations hold for $\wt\Gamma_0$. See Figure~\ref{fig-simple-coupling} for an illustration of the setup. 

Proposition~\ref{prop-simple-coupling} is the main step in the proof of Proposition~\ref{prop-coupling}: once it is established, repeated applications of Lemma~\ref{lem-miller-wu-dim} will allow us to convert a general choice of $(\Gamma_0 ,\wt\Gamma_0)$ into one satisfying the hypotheses of Proposition~\ref{prop-simple-coupling} after finitely many iterations of the Markov chain.

\begin{defn}
Throughout this subsection, for $n\in\BB N_0$ we write $\gamma_n^1,\dots,\gamma_n^M$ for the inner-boundary-surrounding loops of $\Gamma_n$, labeled from outside in.
We similarly define $\wt\gamma_n^1,\dots,\wt\gamma_n^M$ with $\wt\Gamma_n$ in place of $\Gamma_n$. 
\end{defn}

\begin{figure}[t!]
 \begin{center}
\includegraphics[scale=.85]{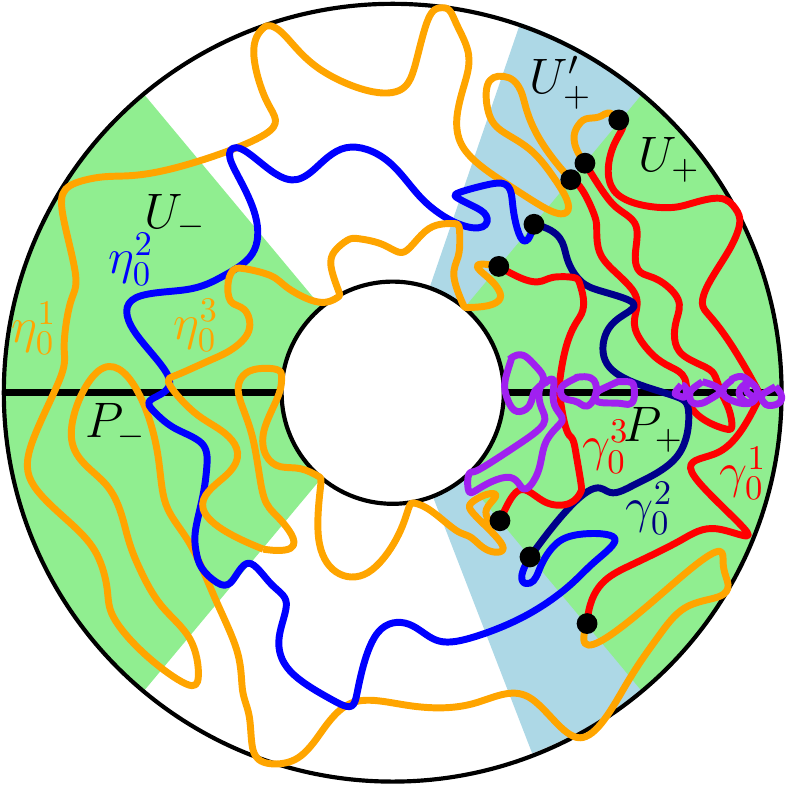}
\vspace{-0.01\textheight}
\caption{Illustration of the set of loops $\Gamma_0(P_+)$ in the setting of Proposition~\ref{prop-simple-coupling} in the case when $M=3$. Loops are shown without self-intersections for clarity, but in reality the loops intersect (but do not cross) themselves in a fractal set. The $P_+$-excursions into $U_+$ (resp.\ complementary $P_+$-excursions out of $U_+$) of the inner-boundary-surrounding loops $\gamma_0^1,\gamma_0^2,\gamma_0^3$ are shown in red or dark blue (resp.\ orange or blue). The other loops in $\Gamma_0(P_+)$ are shown in purple. 
The requirement that $\#\mcl S_{\Gamma_0}(P_+ ; U_+') = 3$ says that only the inner-boundary-surrounding loops are allowed to exit the light blue region $U_+'$, and furthermore each such loop has only one complementary $P_+$-excursion out of $U_+$ which exits $U_+'$. In the proof of Lemma~\ref{lem-outer-loop-coupling}, these excursions are called $\eta_0^1$, $\eta_0^2$, $\eta_0^3$. 
}\label{fig-simple-coupling}
\end{center}
\vspace{-1em}
\end{figure}

The proof of Proposition~\ref{prop-simple-coupling} has three main steps. 
\begin{enumerate}
\item We first show in Lemma~\ref{lem-outer-loop-coupling} that we can couple in such a way that after $M$ steps of the Markov chain, it holds with positive probability that the inner-boundary-surrounding loops of $\Gamma_M$ and $\wt\Gamma_M$ satisfy $\gamma_M^m \cap U_-  = \wt\gamma_M^m \cap U_- $ for each $m=1,\dots,M$ and moreover each of these loops has only one $P_-$-excursion into $U_-$. This is done by using the fact that the $P_+$-excursions of these loops out of $U_+ $ are re-sampled as $\SLE_\kappa$ curves in the Markov chain and applying Lemma~\ref{lem-sle-abs-cont} $M$ times, once for each pair of inner-boundary-surrounding loops.
\item  We next show in Lemma~\ref{lem-one-side-coupling} that we can modify our coupling in such a way that with positive probability, $\Gamma_M(P_- ; U_-) = \wt\Gamma_M(P_-; U_-)$ and $\mcl S_{\Gamma_M}(P_- ; U_-) = \mcl S_{\wt\Gamma_M}(P_- ; U_-)$. The equality $\mcl S_{\Gamma_M}(P_- ; U_-) = \mcl S_{\wt\Gamma_M}(P_- ; U_-)$ comes from the previous step, and the equality $\Gamma_M(P_- ; U_-) = \wt\Gamma_M(P_-; U_-)$ of the sets of ``small" loops intersecting $P_-$ comes from Lemma~\ref{lem-cle-abs-cont}. 
\item Finally, we show that after $M$ additional steps of the Markov chain, one can couple so that with positive probability, the complementary $P_-$-excursions of the inner-boundary-surrounding loops of $\Gamma_{2M}$ and $\wt\Gamma_{2M}$ out of $U_-$ agree. This is done using the fact that these excursions are re-sampled as $\SLE_\kappa$ curves in our Markov chain and (due to the previous step) these $\SLE_\kappa$ curves will be contained in domains which agree in a neighborhood of the initial and terminal points of the curves. This allows us to apply Lemma~\ref{lem-mw-sle-abs-cont} to couple the $\SLE_\kappa$ curves with positive probability. Once we have coupled so that all of the loops of $\Gamma_{2M}$ and $\wt\Gamma_{2M}$ which intersect $P_-$ agree, we are done by the definition of our Markov chain. 
\end{enumerate}

\begin{lem} \label{lem-outer-loop-coupling}
Suppose our initial configurations are such that $\#\mcl S_{\Gamma_0}(P_+ ; U_+') = \#\mcl S_{\wt\Gamma_0}(P_+ ; U_+') = M$. 
There is a coupling of $\Gamma_M$ and $\wt\Gamma_M$ such that with positive probability, the following is true.
\begin{enumerate}
\item $\Gamma_M(P_+ ; U_+) = \Gamma_0(P_+ ; U_+)$, $\mcl S_{\Gamma_M}(P_+ ; U_+) = \mcl S_{\Gamma_0}(P_+ ; U_+)$, and the same is true with $(\wt\Gamma_0 , \wt\Gamma_M)$ in place of $(\Gamma_0, \Gamma_M)$. \label{item-coupling-base}
\item For each $m=1,\dots,M$, the inner-boundary-surrounding loops satisfy $\gamma_M^m \cap U_-  = \wt\gamma_M^m \cap U_- $. \label{item-coupling-agree}
\item Each of the loops $\gamma_M^m$ (equivalently, each of the loops $\wt\gamma_M^m$) for $m=1,\dots,M$ has only one $P_-$-excursion into $U_- $. \label{item-coupling-excursions} 
\end{enumerate}
\end{lem}
\begin{proof} 
The idea of the proof is to apply Lemma~\ref{lem-sle-abs-cont} to couple the $P_+$-excursions out of $U_+'$ of the pairs of loops $(\gamma_0^m, \wt\gamma_0^m)$ one-by-one. 
We need to work from outside in since in order to apply Lemma~\ref{lem-sle-abs-cont}, we need to make sure that the $P_+$-excursions out of $U_+'$ for $\Gamma_0$ and $\wt\Gamma_0$ are contained in domains whose intersection includes a crossing between the two components of $\bdy U_-\setminus\bdy\BB A_\rho$. 

We will inductively construct for each $N = 1, \dots,  M$ a coupling of $\Gamma_N$ and $\wt\Gamma_N$ which satisfies the following conditions.
\begin{enumerate}
\item $\Gamma_N(P_+ ; U_+ ) = \Gamma_0(P_+ ; U_+ )$, $\mcl S_{\Gamma_N}(P_+ ; U_+ ) = \mcl S_{\Gamma_0}(P_+ ; U_+ )$, and $\gamma_N^m = \gamma_0^m$ for $m=N+1,\dots,M$; and the same is true with $(\wt\Gamma_0 , \wt\Gamma_N)$ in place of $(\Gamma_0, \Gamma_N)$. \label{item-coupling-base0}
\item For each $m=1,\dots,N$, the inner-boundary-surrounding loops satisfy $\gamma_N^m \cap U_-  = \wt\gamma_N^m \cap U_- $.  \label{item-coupling-agree0}
\item Each of the loops $\gamma_N^m$ (equivalently, each of the loops $\wt\gamma_N^m$) for $m=1,\dots,M$ has only one $P_-$-excursion into $U_- $. \label{item-coupling-excursions0}
\end{enumerate}
Taking $N=M$ concludes the proof.

For the construction, we will make use of the following notation. 
Let $\eta_0^1,\dots,\eta_0^M$ be the $M$ complementary $P_+$-excursions of $\Gamma_0$ out of $U_+$ which exit $U_+'$, enumerated so that $\eta_0^m$ is an arc of the $m$th outermost loop $\gamma_0^m$. Let $x_0^m$ and $x_0^{m,*}$ be the initial and terminal endpoints of $\eta_0^m$.  
Similarly define $\wt\eta_0^1,\dots,\wt\eta_0^M$ and $\wt x_0^1, \wt x_0^{1,*} , \dots,\wt x_0^M , \wt x_0^{M,*}$ with~$\wt\Gamma_0$ in place of~$\Gamma_0$. 
\medskip

\noindent\textit{Step 1: base case.} We will first construct a coupling of $\Gamma_1$ and $\wt\Gamma_1$ satisfying the above conditions for $N=1$.  
We first couple $(\xi_1 , x_1 )$ and $(\wt\xi_1 , \wt x_1  )$ so that with positive probability, $\xi_1 = \wt\xi_1 = +, x_1 = x_0^1 , \wt x_1 = \wt x_0^1$.  

Let $D$ be the connected component of $\BB A_\rho \setminus \ol{ \bigcup \left(  \Gamma_0(P_+ ) \setminus \{\gamma_0^1\} \right) \cup \alpha_0^1 }$ with $x_0^1$ and $x_0^{1,*}$ on its boundary, where here $\alpha_0^1$ is the complementary arc of $\eta_0^1$ in $\gamma_0^1$.  
Since $\Gamma_0$ is non-crossing, $\gamma_0^1$ is the \emph{outermost} loop of $\Gamma_0$ surrounding $\BB A_\rho$, and none of the loops in $\Gamma_0(P_+)$ exit $U_+'$ except for $\gamma_0^1,\dots,\gamma_0^M$, no loop of $\Gamma_0$ other than $\gamma_0^1$ can hit $\bdy\BB D\setminus \bdy U_+'$ and hence $\bdy\BB D\setminus \bdy U_+' \subset \bdy D$.  The definition of our Markov chain implies that the conditional law of $\eta_1'$ given $\{\xi_1 = + , x_1 =  x_0^1 \}$ is that of a chordal $\SLE_\kappa$ in $D$ from $x_0^1$ to $x_0^{1,*}$.  The analogous statements hold with $\wt\Gamma_0$ in place of $\Gamma_0$. 

If we let $\wt D$  be defined analogously to $D$ with $\wt\Gamma_0$ in place of $\Gamma_0$, then since $\bdy\BB D \setminus \bdy U_+'\subset \bdy D\cap \bdy \wt D$, it follows that $D\cap \wt D$ contains the closure of a connected subset $V$ of $U_- $ whose boundary intersects both connected components of $\bdy U_-  \setminus \bdy \BB A_\rho$.  By Lemma~\ref{lem-sle-abs-cont} (applied with this choice of $D$, $\wt D$, and $V$ and with $U = B_\ep(V)$ for a small enough $\ep > 0$), conditionally on $\{\xi_1 = \wt\xi_1 = +, x_1 = x_0^1 , \wt x_1 = \wt x_0^1\}$, we can further couple $\eta_1'$ and $ \wt\eta_1'$ in such a way that with positive probability, the segments of $\eta_1'$ and $\wt\eta_1'$ between their first entrance time of $V$ and their next subsequent exit time from $B_\ep(V)$ coincide, and neither of these segments hits $P_-$ between the first time after hitting $P_-$ at which it exits $U_- $ and the time when it exits $B_\ep(V)$. By Lemma~\ref{lem-miller-wu-dim} each of $\eta_1'$ and $\wt\eta_1'$ has positive probability not to return to $U_- $ after exiting $B_\ep(V)$. We have therefore proved that we can couple $(\xi_1 , x_1, \eta_1')$ and $(\wt\xi_1 , \wt x_1 , \wt\eta_1')$ in such a way that  with positive probability,
\eqb \label{eqn-one-curve-coupling}
 \xi_1 = \wt\xi_1 = +, \quad x_1 = x_0^1 , \quad \wt x_1 = \wt x_0^1,  \quad \eta_1' \cap U_-  = \wt\eta_1' \cap U_-   ,
\eqe 
and neither $\eta_1'$ nor $\wt\eta_1'$ hits $P_-$ again after the first time after hitting $P_-$ at which it exits $U_-$. 

By the definition of our coupling, if $\xi_1 = +$ and $x_1 = x_0^1$, then $\Gamma_1(P_+  ) \setminus \{\gamma_1^1\} = \Gamma_0(P_+ ) \setminus \{\gamma_0^1\}$, $\mcl S_{\Gamma_1}(P_+ ; U_+ ) = \mcl S_{\Gamma_0}(P_+ ; U_+ )$, and $\gamma_1^1$ is the concatenation of the arc $\gamma_0^1 \setminus \eta_0^1$ and the curve $\eta_1'$. 
The same is true for $\wt\Gamma_1$. Therefore, our desired conditions for $N=1$ hold whenever the event described in~\eqref{eqn-one-curve-coupling} occurs (note that the condition stated just after~\eqref{eqn-one-curve-coupling} is needed to obtain condition~\ref{item-coupling-excursions0} for $N=1$). 
\medskip

\noindent\textit{Step 2: inductive step.}
Suppose $N = 2,\dots, M$ and we have coupled $\Gamma_{N-1}$ and $\wt\Gamma_{N-1}$ so that the above conditions are satisfied with positive probability with $N-1$ in place of $N$. 
Suppose further that we are working on the positive probability event that these conditions are satisfied with $N-1$ in place of $N$.  
We will use a similar argument as in the case $N=1$.  
Recall that the $N$th outermost loops satisfy $\gamma_{N-1}^N = \gamma_0^N$ and let $D_N$ be the connected component $D_N$ of 
\eqbn
\BB A_\rho \setminus \ol{  \bigcup  \left( \Gamma_{N-1}(P_+ ) \setminus \{ \gamma_0^N \} \right) \cup \alpha_0^N  } 
\eqen
which has $x_0^N,x_0^{N,*}$ on its boundary, where here $\alpha_0^N$ is the complementary arc of $\eta_0^N$ in $\gamma_0^N$.
Since the loops $\gamma_{N-1}^1,\dots,\gamma_{N-1}^M$ are enumerated from outside in, $\gamma_{N-1}^N = \gamma_0^N$, and none of the loops in $\Gamma_{N-1}(P_+)$ exit $U_+'$ except for the inner-boundary-surrounding loops, we find that $D_N$ has a boundary arc which intersects both connected components of $\bdy U_- \setminus \bdy \BB A_\rho$ and is part of the loop $\gamma_{N-1}^{N-1}$. The same is true with $\wt\Gamma_{N-1}$ in place of $\Gamma_{N-1}$. 
If we let $\wt D_N$ be defined in the same manner as $D_N$ but with $\wt\Gamma_{N-1}$ in place of $\Gamma_{N-1}$, then since $\wt\gamma_{N-1}^{N-1} \cap U_-  =\gamma_{N-1}^{N-1} \cap U_- $, the set $D_N \cap \wt D_N$ contains the closure of a connected open subset $V_N$ of $U_- $ which intersects both connected components of $\bdy U_-  \setminus \bdy \BB A_\rho$. 
Using Lemmas~\ref{lem-sle-abs-cont} and~\ref{lem-miller-wu-dim} in exactly the same manner as in the case $N=1$, we can now obtain a coupling of $\Gamma_N$ and $\wt\Gamma_N$ satisfying our desired conditions. This completes the induction, hence the proof.  
\end{proof}

Building on Lemma~\ref{lem-outer-loop-coupling}, we now extend to a coupling of $\Gamma_M$ and $\wt\Gamma_M$ for which the (infinitely many) loops which intersect $P_-$ and are contained in $U_-$ agree.

\begin{lem} \label{lem-one-side-coupling}
Suppose our initial configurations are such that $\#\mcl S_{\Gamma_0}(P_+ ; U_+') = \#\mcl S_{\wt\Gamma_0}(P_+ ; U_+') = M$. 
There is a coupling of $\Gamma_M$ and $\wt\Gamma_M$ such that with positive probability, the following is true.
\begin{enumerate}
\item[$1'$.] $\Gamma_M(P_- ; U_- ) = \wt\Gamma_M(P_- ; U_- )$ and $\mcl S_{\wt\Gamma_M}(P_- ; U_- ) = \mcl S_{\wt\Gamma_M}(P_- ; U_- )$. \label{item-coupling-base'}
\item[$2'$.] $\#\mcl S_{\wt\Gamma_M}(P_- ; U_- ) = \# \mcl S_{\wt\Gamma_M}(P_- ; U_- ) = M$. 
\end{enumerate}
\end{lem}
\begin{proof}
Suppose we have coupled $\Gamma_M$ and $\wt\Gamma_M$ as in Lemma~\ref{lem-outer-loop-coupling}. We will use our coupling lemma for $\CLE_\kappa$'s on different domains (Lemma~\ref{lem-cle-abs-cont}) to modify this coupling to get a stronger coupling in which the statement of the lemma is satisfied.

Since none of the elements of $\Gamma_0(P_+ )$ except for the inner-boundary-surrounding loops intersect $U_- $, whenever the conditions of Lemma~\ref{lem-outer-loop-coupling} hold (which happens with positive probability),
\eqb \label{eqn-out-loop-agree}
\ol{\bigcup \Gamma_M(P_+)} \cap U_-  = \ol{\bigcup \wt\Gamma_M(P_+)} \cap U_-  .
\eqe 
By the definition of our Markov chain, the conditional law of $\Gamma_M$ given $\Gamma_M(P_+)$ on the event $\{\xi_M = +\}$ is given by the union of $\Gamma_M(P_+)$ and an independent $\CLE_\kappa$ in each connected component of $\BB A_\rho \setminus \ol{\bigcup \Gamma_M(P_+)}$. The analogous statement holds for $\wt\Gamma_M$. 
On the event that~\eqref{eqn-out-loop-agree} holds, there is a one-to-one correspondence between connected components of  $\BB A_\rho \setminus \ol{\bigcup\Gamma_M(P_+)}$ which intersect $P_-$ and connected components of $\BB A_\rho \setminus \ol{\bigcup\wt\Gamma_M(P_+)}$ which intersect $P_-$, wherein corresponding components share the same connected boundary arc of $\ol{\bigcup \Gamma_M(P_+)} \cap U_-  = \ol{\bigcup \wt\Gamma_M(P_+)} \cap U_- $. 
In fact, by the continuity of the loops $\gamma_M^1,\dots,\gamma_M^M$, a.s.\ all but finitely many corresponding pairs of such components have their boundaries entirely traced by $\gamma_M^m \cap U_-  = \wt\gamma_m^M\cap U_- $ for some $m=1,\dots,M$, in which case the two components are identical and contained in $U_- $. 
We may therefore apply Lemma~\ref{lem-cle-abs-cont} to each pair of non-identical corresponding components (with $X$ equal to the intersection of either of the components with $P_-$) to re-couple in such a way with positive probability, the conditions of Lemma~\ref{lem-outer-loop-coupling} are satisfied and the following additional conditions hold. 
\begin{enumerate} 
\setcounter{enumi}{3}
\item $\Gamma_M(P_-) \setminus \{\gamma_M^1,\dots,\gamma_M^M\} = \wt\Gamma_M(P_-) \setminus \{\wt\gamma_M^1,\dots,\wt\gamma_M^M\}$. \label{item-coupling-line}
\item Each loop of $\Gamma_M(P_-) \setminus  \{\gamma_M^1,\dots,\gamma_M^M\}$ (equivalently, each loop of $\wt\Gamma_M(P_-) \setminus \{\wt\gamma_M^1,\dots,\wt\gamma_M^M\}$) is contained in $U_- $. \label{item-coupling-contained}
\end{enumerate}

We will now argue that whenever the three conditions of Lemma~\ref{lem-outer-loop-coupling} plus the above two conditions are satisfied, the conditions in the statement of the lemma hold. Indeed, conditions~\ref{item-coupling-line} and~\ref{item-coupling-contained} above immediately imply that $\Gamma_M(P_- ; U_- ) = \wt\Gamma_M(P_- ; U_- )$. 
The relation $\mcl S_{\wt\Gamma_M}(P_- ; U_- ) = \mcl S_{\wt\Gamma_M}(P_- ; U_- )$ follows from condition~\ref{item-coupling-agree} of Lemma~\ref{lem-outer-loop-coupling} since conditions~\ref{item-coupling-line} and~\ref{item-coupling-contained} imply that no loops in $\Gamma_M(P_-)$ (resp.\ $\wt\Gamma_M(P_-)$) can exit $U_-$ except for $\gamma_M^1,\dots,\gamma_M^M$. 
The fact that $\#\mcl S_{\wt\Gamma_M}(P_- ; U_- ) = \# \mcl S_{\wt\Gamma_M}(P_- ; U_- ) = M$ follows from condition~\ref{item-coupling-excursions} of Lemma~\ref{lem-outer-loop-coupling}. 
\end{proof}

\begin{proof}[Proof of Proposition~\ref{prop-simple-coupling}]
It suffices to construct a coupling of $\Gamma_{2M}$ and $\wt\Gamma_{2M}$ such that $\Gamma_{2M}(P_-) = \wt\Gamma_{2M}(P_-)$ with positive probability: indeed, the definition of our Markov chain implies that conditionally on $\Gamma_{2M}(P_-)$, the conditional law of the rest of $\Gamma_{2M}$ is that of an independent $\CLE_\kappa$ in each connected component of $\BB D\setminus \ol{\bigcup \Gamma_{2M}(P_-)}$. To construct such a coupling, we will inductively construct for each $N  = 1,\dots , M$ a coupling of $\Gamma_{M+N}$ and $\wt\Gamma_{M+N}$ such that with positive probability, none of the loops in $\Gamma_{M+N}(P_-)$ or $\wt\Gamma_{M+N}(P_-)$ exit $U_- $ except for the inner-boundary-surrounding loops, and these loops satisfy $\gamma_{M+N}^m = \wt\gamma_{M+N}^m$ for each $m=1,\dots, N$. The argument is somewhat similar to that of Lemma~\ref{lem-outer-loop-coupling}, but simpler since we have a stronger relationship between the pairs of domains under consideration. 

Start by coupling $\Gamma_{ M}$ and $\wt\Gamma_{ M}$ as in Lemma~\ref{lem-one-side-coupling}. 
Throughout the proof we work on the (positive probability) event that the conditions of that lemma are satisfied.
For $m=1,\dots,N$ let $\beta_M^m$ be the unique $P_-$-excursion of $ \gamma_M^m$ into $U_- $ and let $x_M^m$ and $x_M^{m,*}$ be its terminal and initial endpoints, respectively.
Note that by the definition of the event in Lemma~\ref{lem-one-side-coupling}, these definitions are unaffected if we replace $\Gamma_M$ with $\wt\Gamma_M$. 

Let us now construct our desired coupling in the base case $N=1$. Let $D$ be the connected component of $\BB A_\rho\setminus \ol{\bigcup \left( \Gamma_M(P_-) \setminus \{\gamma_M^1\} \right) \setminus \alpha_M^1 }$ with $x_M^m$ and $x_M^{m,*}$ on its boundary (where here $\alpha_M^1$ is the complementary arc of $\beta_M^1$ in $\gamma_M^1$) and analogously define $\wt D$. 
Since $\gamma_M^1$ and $\wt\gamma_M^1$ are the outermost inner-boundary-surrounding loops in $\Gamma_M$ and $\wt\Gamma_M$, respectively, and since the $\beta_M^m$'s are the only elements of $\mcl S_{\Gamma_M}(P_- ; U_- )$, the set $\bdy(D\cap \wt D)$ contains a connected arc which includes $\bdy\BB D\setminus \bdy U_- $, $x_M^1$, and $x_M^{1,*}$ in its interior. The set $D\cap \wt D$ contains a neighborhood $U$ of this arc.

The conditional law of the curve $\eta_{M+1}'$ given $\Gamma_M$ and the event $\{\xi_{M+1} = -, x_{M+1} = x_M^1\}$ is that of an $\SLE_\kappa$ in $D$ from $x_M^1$ to $x_M^{1,*}$.
The analogous statement holds for $\wt\Gamma_M$.
Applying Lemma~\ref{lem-mw-sle-abs-cont} with the above choice of $D,\wt D$, and $U$ along with Lemma~\ref{lem-miller-wu-dim} with $P$ a path from $x_M^1$ to $x_M^{1,*}$ in $U$ shows that we can couple $\Gamma_{M+1}$ and $\wt\Gamma_{M+1}$ in such a way that with positive probability, $\eta_{M+1}' \subset U$ and $\eta_{M+1}' = \wt\eta_{M+1}'$.
If this is the case, then the definition of our Markov chain shows that our desired conditions are satisfied for $N=1$. This concludes the proof of the base case.

For the inductive step, we assume $N =2,\dots,M$ and the desired coupling of $\Gamma_{M+N-1}$ and $\wt\Gamma_{M+N-1}$ has been constructed. We then apply exactly the same argument as in the case $N=1$ in the connected component of $\BB A_\rho\setminus \gamma_{M+N-1}^{N-1} = \BB A_\rho\setminus \wt\gamma_{M+N-1}^{N-1}$ which has the inner boundary of $\BB A_\rho$ on its boundary. 
\end{proof}

\subsection{Proof of the coupling proposition in general}
\label{sec-full-coupling}

We will now deduce Proposition~\ref{prop-coupling} for a general choice of $(\Gamma_0 ,\wt\Gamma_0)$ from the special case given in Proposition~\ref{prop-simple-coupling}. 
Let $K = K(\Gamma_0,\wt\Gamma_0)$ be the total number of crossings from $P_+$ to $P_-$ by loops in $\Gamma_0$ plus the total number of crossings from $P_+$ to $P_-$ by loops in $\wt\Gamma_0$. We will prove Proposition~\ref{prop-coupling} by induction on $K$. 

Note that by Lemma~\ref{lem-excursion-count}, $K$ can equivalently be defined in terms of crossings from $P_-$ to $P_+$ or in terms of complementary $P_\pm$-excursions out of $U_\pm$ which hit $P_\mp$. We always have $K\geq 2M$ since each of $\Gamma_0$ and $\wt\Gamma_0$ has $M$ loops which surround the inner boundary of $\BB A_\rho$. 
 
We will now establish the base case of our inductive argument by reducing the case when $K=2M$ to the case when $\#\mcl S_{\Gamma}(P_+; U_+') = \#\mcl S_{\wt\Gamma}(P_+ ; U_+') = M$. This will be accomplished by using Lemma~\ref{lem-miller-wu-dim} to get rid of the $P_+$-excursions of $\Gamma_0$ and $\wt\Gamma_0$ out of $U_+'$ which do not hit $P_-$ one at a time. 

\begin{lem} \label{lem-base-coupling}
Suppose our initial configurations are such that $K=2M$.
There is an $N\in\BB N$ (depending only on $\Gamma_0$ and $\wt\Gamma_0$) and a coupling of $\Gamma_N$ and $\wt\Gamma_N$ such that $\BB P[\Gamma_N = \wt\Gamma_N] > 0$.
\end{lem} 
\begin{proof}
Recall the slightly larger annular slices $U_+ ' \supset U_+$ from~\eqref{eqn-bigger-slices}. 
Define
\eqbn
N = N(\Gamma_0, \wt\Gamma_0) :=  \#\mcl S_{\Gamma_0}(P_+ ; U_+') + \#\mcl S_{\wt\Gamma_0}(P_+ ; U_+')  
\eqen 
and note that $N\geq 2M$. 
We will prove by induction on $N$ that the statement of the lemma holds for this choice of $N$.
The base case $N = 2M$ is Proposition~\ref{prop-simple-coupling}. 
Suppose now that $N \geq 2M+1$ and the statement of the lemma has been established for configurations with $N( \Gamma_0, \wt\Gamma_0) \leq N-1$.  
Since $N\geq 2M+1$ and $K=2M$, either $\Gamma_0$ or $\wt\Gamma_0$ has a complementary $P_+$-excursion out of $U_+$ which exits $U_+'$ but do not hit $P_-$. 
The idea of the proof is to ``pull" this excursion into $U_+'$ and thereby reduce $N$ by at least 1.

Assume without loss of generality that $\Gamma_0$ has a complementary $P_+$-excursion out of $U_+$ which exits $U_+'$. Call this excursion $\eta_x$ and let $x$ and $x^*$ be its endpoints. 
Then $x$ and $x^*$ lie in the same connected component of $\bdy U_+ \setminus \bdy \BB A_\rho$ (otherwise the excursion would have to hit $P_-$). 

By possibly choosing a different excursion, we can assume that the segment $[x,x^*]$ from $x$ to $x^*$ does not contain the endpoints of any other complementary $P_+$-excursions out of $U_+$ which exits $U_+'$.
If this is the case, then the region enclosed by $[x,x^*]$ and the excursion $\eta_x$ does not contain any segment which intersects $\bdy U_+'$ of a loop in $\Gamma_0(P_+)$.
Consequently, there is a path in this region from $x$ to $x^*$ which is entirely contained in $U_+'$. 
On the positive probability event $\{\xi_1 = +, x_1 = x\}$, this same path is also contained in the connected component of the set~\eqref{eqn-markov-chain-union} for $\Gamma = \Gamma_0$ which has $x$ and $x^*$ on its boundary and the conditional law of $\eta_1'$ is that of a chordal $\SLE_\kappa$ from~$x$ to~$x^*$ in this connected component. Using Lemma~\ref{lem-miller-wu-dim}, we therefore find that with positive probability, $\eta_1'$ is contained in~$U_+'$.

If this is the case, then the definition of $\Gamma_1$ shows that $\#\mcl S_{\Gamma_1}(P_+ ; U_+') = \#\mcl S_{\Gamma_0}(P_+ ; U_+') - 1$. 
Trivially, $\BB P[ \#\mcl S_{\wt\Gamma_1}(P_+ ; U_+') \leq \#\mcl S_{\wt\Gamma_0}(P_+ ; U_+')  ]  >0 $.  
Therefore, with positive probability $N(\Gamma_1 , \wt\Gamma_1) \leq N-1$, so by the inductive hypothesis applied with $(\Gamma_1 , \wt\Gamma_1)$ in place of $ (\Gamma_0 , \wt\Gamma_0)$ we conclude the proof.
\end{proof}

To treat the case when $K > 2M$, we will need the following purely topological lemma which will allow us to unwind loops which wrap around the origin multiple times. See Figure~\ref{fig-loop-excursion} for an illustration of the statement and proof. We need the statement only for non-crossing loops, but we state it for general loops since the proof does not use the non-crossing property.

\begin{figure}[t!]
 \begin{center}
\includegraphics[scale=.7]{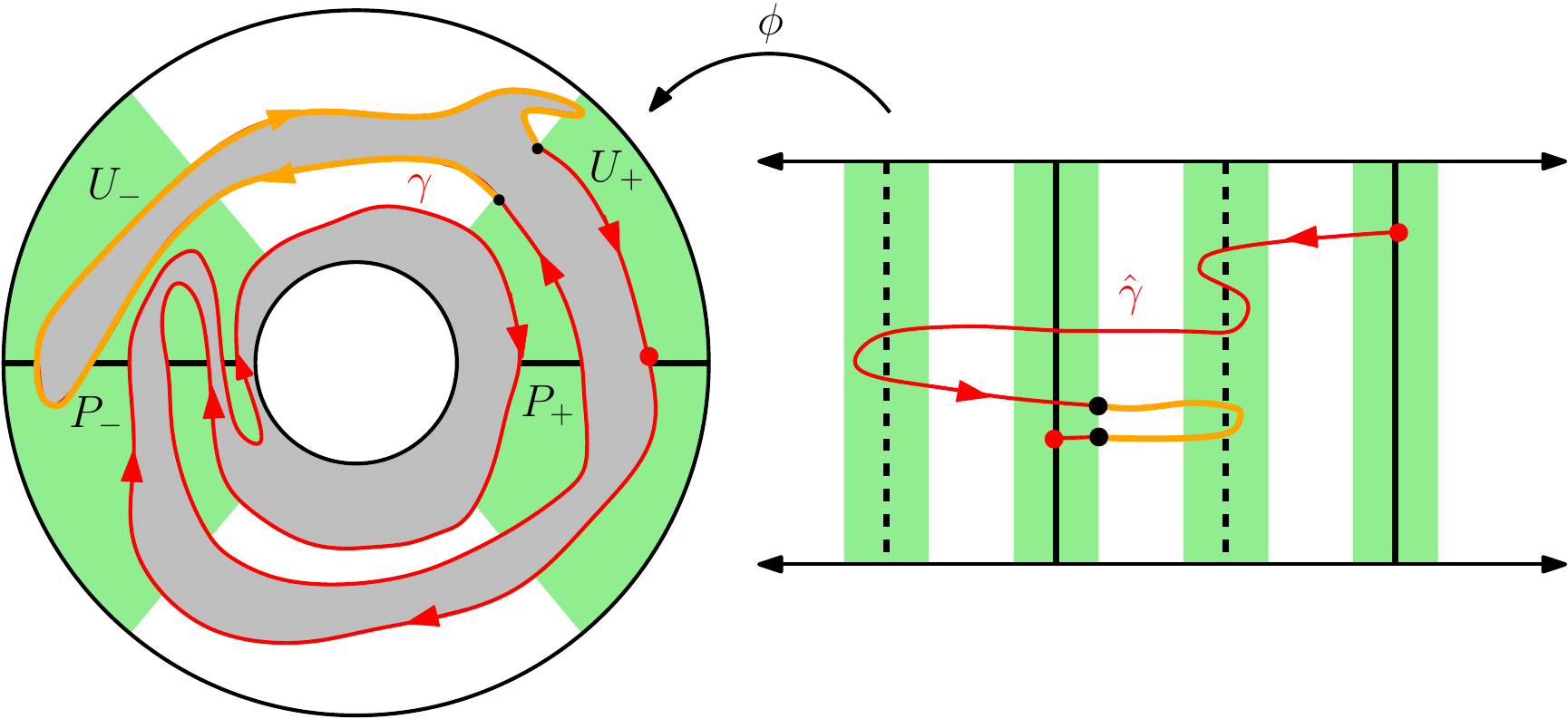}
\vspace{-0.01\textheight}
\caption{ Illustration of the statement and proof of Lemma~\ref{lem-loop-excursion}. \textbf{Left:} A loop $\gamma$ with winding number 1 around the inner boundary of $\BB A_\rho$. The interior of $\gamma$ is shown in grey. The orange arc of $\gamma$ is a complementary $P_+$-excursion out of $U_+$ which hits $P_-$ and has both of its endpoints in the same component of $\ol{\bdy U_+\setminus \bdy \BB A_\rho}$.
Here we have shown $\gamma$ as a simple loop for clarity, but in practice $\gamma$ will be a $\CLE_\kappa$-type loop.
\textbf{Right:} The lift $\wh\gamma$ of $\gamma$ to the universal cover of $\BB A_\rho$. Solid (resp.\ dashed) black segments are mapped to $P_+$ (resp.\ $P_-$). 
}\label{fig-loop-excursion}
\end{center}
\vspace{-1em}
\end{figure}

\begin{lem} \label{lem-loop-excursion}
Let $\gamma $ be an arbitrary loop in $\BB A_\rho$ (not necessarily non-self-crossing) and suppose that the winding number of $\gamma$ around the inner boundary of $\BB A_\rho$ is $N\in\BB N_0$. 
If $\gamma$ has at least $N+1$ complementary $P_+$-excursions out of $U_+$ which intersect $P_-$, then either $\gamma$ has a complementary $P_+$-excursion out of $U_+$ which intersects $P_-$ and has both of its endpoints in the same connected component of $\ol{\bdy U_+ \setminus \bdy \BB A_\rho}$; or the same is true with ``$+$" and ``$-$" interchanged.
\end{lem}
\begin{proof}
Choose a parameterization $\gamma :[0,1] \rta \BB A_\rho$ in such a way that $\gamma(0) \in P_+$. 
Let $\phi : \BB R \times [0,\rho] \rta \BB A_\rho$ be the conformal universal covering map normalized so that
\eqb \label{eqn-universal-cover}
\phi^{-1}(P_+) = \bigcup_{k\in\BB Z}\left(  \{2\pi k\} \times [0,\rho] \right)  \quad \text{and} \quad
\phi^{-1}(P_-) = \bigcup_{k\in\BB Z}\left(  \{2\pi k + \pi \} \times [0,\rho] \right) .
\eqe 
We note that the pre-images under $\phi$ of the two connected components of $\ol{\bdy U_- \setminus \bdy \BB A_\rho}$ are 
\eqb \label{eqn-universal-cover-component}
\bigcup_{k\in\BB Z}\left(   \{ 2\pi k + 3\pi/4\} \times [0,\rho] \right)  \quad \text{and} \quad
\bigcup_{k\in\BB Z}\left(  \{ 2\pi k + 5\pi/4\} \times [0,\rho] \right) , 
\eqe
and a similar statement holds for $\ol{\bdy U_+ \setminus \bdy \BB A_\rho}$. 
Let $\wh\gamma : [0,1] \rta \BB R\times [0,\rho]$ be the lift of $\gamma$ to $\BB R\times [0,\rho]$, so that $\phi\circ\wh\gamma = \gamma$, normalized so that $\re \wh\gamma(0)  = 0$. Since the winding number of $\gamma$ is $N$, we have $ \re\wh\gamma(1)  = 2\pi N$. 

By Lemma~\ref{lem-excursion-count} and our hypothesis on $\gamma$, the loop $\gamma$ has at least $N+1$ crossings from $P_+$ to $P_-$. 
Say that $[u,v]\subset [0,1]$ is a \emph{crossing interval} if $\gamma|_{[u,v]}$ is a crossing from $P_+$ to $P_-$. 
The crossing intervals are naturally ordered from left to right. 
If $[u,v]$ is a crossing interval, then by~\eqref{eqn-universal-cover} there is a $k\in\BB Z$ such that $\re\wh\gamma(u)  = 2\pi k$ and $\re\wh\gamma(v) \in \{2\pi k - \pi , 2\pi k + \pi\}$. 
Since $\re\wh\gamma(0) = 0$ and $\re\wh\gamma(1) = 2\pi N$ and there are at least $N+1$ crossing intervals, there must be two consecutive crossing intervals $[u_1,v_1]$ and $[u_2,v_2]$ (i.e., $u_1<v_1<u_2<v_2$ and there is no crossing interval in $[v_1,u_2]$) such that $\re\wh\gamma(v_1)-\re\wh\gamma(u_1)$ and $\re\wh\gamma(v_2) - \re\wh\gamma(u_2)$ have opposite signs.
Henceforth assume that $\re\wh\gamma(v_1)-\re\wh\gamma(u_1)  = \pi$ and $\re\wh\gamma(v_2)-\re\wh\gamma(u_2)  = -\pi$ (the other case is treated similarly).
Then for some $k\in\BB Z$, we have $\re\wh\gamma(u_1) = 2\pi k$ and $\re\wh\gamma(v_1) = 2\pi k + \pi$. 
Since the two crossing intervals are consecutive, there are two possibilities for $(\re\wh\gamma(u_2) , \re\wh\gamma(v_2))$: it is equal to either $(2  \pi k  ,  2\pi k -\pi ) $ or $(2\pi (k+1)   , 2\pi k + \pi) )$. Again, we assume that we are in the former case (the other case is treated similarly). 

We have $\re\wh\gamma(u_2) = 2\pi k$ and $\re\wh\gamma(t)  < 2\pi  k + \pi$ for $t \in [u_1,v_2]$. 
We will construct a complementary $P_-$-excursion out of $U_-$ which contains $\gamma(u_2)$ and satisfies the conditions in the statement of the lemma (two of the other four possible configurations above result in complementary $P_+$-excursions instead of complementary $P_-$-excursions).
We will first find the $P_-$-excursions of $\gamma$ into $U_-$ which come immediately before and after $\wh\gamma(u_2)$. 
Let $\ul s$ (resp.\ $\ol s$) be the last time $s$ before $u_2$ (resp.\ the first time $s$ after $u_2$) for which $\re\wh\gamma(s)  = 2\pi k -\pi $, so that $\gamma(\ul s), \gamma(\ol s) \in P_-$.
Let $\ul t$ (resp.\ $\ol t$) be the first time $s$ after $\ul s$ (resp.\ the last time $s$ before $\ol s$) with $\re\wh\gamma(s) = 2\pi k - 3\pi/4$.
By~\eqref{eqn-universal-cover-component}, $\gamma(\ul t)$ and $\gamma(\ol t)$ lie in the same connected component of $\ol{\bdy U_- \setminus \bdy \BB A_\rho}$ and $\gamma(\ul t)$ (resp.\ $\gamma(\ol t)$) is an endpoint of a $P_-$-excursion of $\gamma$ into $U_-$ which contains $\gamma(\ul s)$ (resp.\ $\gamma(\ol s)$). 
Furthermore, by the definition of $\ul s$ and $\ol s$ and since $\re\wh\gamma(t)  < 2\pi  k + \pi$ for $t \in [u_1,v_2] \supset [\ul s , \ol s]$, $\gamma$ does not hit $P_-$ between times $\ul s$ and $\ol s$, so $\gamma|_{[\ul t , \ol t]}$ is a complementary $P_-$-excursion of $\gamma$ out of $U_-$. 
Since $\gamma(u_2) \in P_+$, this excursion hits $P_+$ and by our choice of $\ul t , \ol t$, its endpoints lie in the same connected component of $\bdy U_+\cap \bdy U_-$.
\end{proof}

\begin{proof}[Proof of Proposition~\ref{prop-coupling}]
As explained at the beginning of this subsection, we will induct on $K$. The base case $K=2M$ was treated in Lemma~\ref{lem-base-coupling}. Suppose $K \geq 2M+1$ and we have established the proposition for all values of $K' \leq K$. 

Since $K\geq 2M+1$, either $\Gamma_0$ or $\wt\Gamma_0$ has at least $M+1$ complementary $P_+$-excursions out of $U_+$ which hit $P_-$. 
Suppose without loss of generality that $\Gamma_0$ has at least $M+1$ such excursions.
Since $\Gamma_0$ has $M$ loops with winding number 1 around the inner boundary of $\BB A_\rho$ and the rest of the loops have winding number zero, there must be a loop $\gamma \in \Gamma$ such that the following is true.
The number of complementary $P_+$-excursions of $\gamma$ out of $U_+$ which intersect $P_-$ exceeds the winding number of $\gamma$ around the inner boundary of $\BB A_\rho$ by at least 1.
By Lemma~\ref{lem-loop-excursion}, either $\gamma$ has a complementary $P_+$-excursion out of $U_+$ which intersects $P_-$ and has both of its endpoints in the same connected component of $\bdy U_+ \setminus \bdy \BB A_\rho$; or the same is true with ``$+$" and ``$-$" interchanged. Assume that the former condition (with complementary $P_+$-excursions) holds; the other case is treated identically. 

Let $x$ and $x^*$ be the endpoints of a complementary $P_+$-excursion of $\gamma$ out of $U_+$ which intersects $P_-$ and has both of its endpoints in the same connected component of $\bdy U_+ \setminus \bdy \BB A_\rho$. By possibly choosing a different excursion, we can assume that the segment $[x,x^*]$ from $x$ to $x^*$ does not contain the endpoints of any other complementary $P_+$-excursions out of $U_+$ which exits $P_-$. 
As in the proof of Lemma~\ref{lem-base-coupling}, this shows that there is a path from $x$ to $x^*$ in the set~\eqref{eqn-markov-chain-union} for $(\Gamma_0, \Gamma_1)$ with $x$ and $x^*$ on its boundary which does not hit $P_-$.  

On the event $\{\xi_1  = + ,\, x_1 = x\}$, the conditional law of $\eta_{x_1}'$ is that of a chordal $\SLE_\kappa$ from $x$ to $x^*$ in this component. 
By Lemma~\ref{lem-miller-wu-dim}, it holds with positive conditional probability given $\{\xi_1  = + ,\, x_1 = x\}$ that $\eta_{x_1}'$ does not hit $P_-$. 
In this case, the definition of the Markov chain implies that $\Gamma_1$ has at least one fewer complementary $P_+$-excursions out of $U_+$ which intersect $P_-$ than has $\Gamma_0$. It is easily seen that with positive probability $\wt\Gamma_0$ and $\wt\Gamma_1$ have the same number of complementary $P_+$-excursions out of $U_+$ which intersect $P_-$. 
This shows that with positive probability, the value of $K$ corresponding to $( \Gamma_1, \wt\Gamma_1)$ is strictly less than the value of $K$ corresponding to $( \Gamma_0, \wt\Gamma_0)$. 
By combining this with the inductive hypothesis, we conclude the proof.
\end{proof}

\subsection{Uniqueness of the stationary measure}
\label{sec-stationary}

We will now deduce Proposition~\ref{prop-stationary} from Proposition~\ref{prop-coupling}. 
This will be done using ergodic theory arguments similar to those in~\cite[Section 4]{ig2} or~\cite[Appendix A]{msw-sle-range}.
The key input in the argument is the following general theorem from Markov chain theory.

\begin{thm} \label{thm-ergodic-decomp}
Let $(\Omega,d  )$ be a separable metric space and suppose that $\Omega$ is a Borel measurable subset of its metric completion. Let $\Pi(x,dy)$ be the transition kernel of a Markov chain on $\Omega$ such that the measure $x\mapsto \Pi(x,\cdot)$ is a Borel measurable function from $\Omega$ to the space of probability measures on $\Omega$, when the latter is equipped with the Prokhorov distance. 
Let $\mu$ be a stationary probability measure for $\Pi$ (i.e., $\int_A\Pi(x,dy)\mu(dy) = \mu(A)$ for each Borel set $A\subset \Omega$).
Then $\mu$ is a convex combination of stationary ergodic measures for $\Pi$, i.e., there exists a probability measure $\pi_\mu$ on the space $\mcl M_{\op{e}}$ of stationary ergodic probability measures for $\Pi$ such that
\eqb \label{eqn-ergodic-decomp}
\mu = \int_{\mcl M_{\op{e}}} \nu \,  \pi_\mu(d\nu) .
\eqe 
Furthermore, any two distinct elements of $\mcl M_{\op{e}}$ are mutually singular.
\end{thm}
\begin{proof}
This is a classical result in Markov chain theory (see, e.g.,~\cite[Chapter 6]{varadhan-book}) but is usually stated with the stronger hypothesis that $(\Omega,d)$ is complete and separable. We will explain how to extract the given statement from the statement with this stronger hypothesis.\footnote{It is important for us to not require a completeness hypothesis; see Figure~\ref{fig-non-complete} and Lemma~\ref{lem-non-crossing-msrble}).}
To this end, let $(\wh\Omega,\wh d)$ be the metric completion of $(\Omega,d)$ and view $\Omega$ as a subset of $\wh\Omega$.
We define an extended Markov kernel $\wh\Pi$ on $\wh\Omega$ by
\eqbn
\wh\Pi(x,dy) 
:= \begin{cases}
\Pi(x,dy) \quad &x\in \Omega \\
\BB 1_x(dy) \quad &x\in \wh\Omega\setminus \Omega .
\end{cases}
\eqen
In other words, the Markov chain with transition kernel $\wh\Pi$ evolves according to $\Pi$ if we start in $\Omega$ and is constant if we start in $\wh\Omega\setminus \Omega$. 
Since $\Omega$ is a Borel measurable subset of $\wh\Omega$, it is easily seen that the Borel $\sigma$-algebra of $(\Omega,d)$ is contained in that of $(\wh\Omega,\wh d)$ (the latter $\sigma$-algebra is generated by $\wh d$-metric balls and their complements; the former $\sigma$-algebra is generated by the intersections of these sets with $\Omega$). 
Therefore, $x\mapsto \wh\Pi(x,\cdot)$ is a Borel measurable function on $\wh\Omega$. 

We identify Borel measures on $\Omega$ with Borel measures on $\wh\Omega$ which vanish on $\wh\Omega\setminus \Omega$. 
Such a measure $\mu$ is stationary (resp.\ ergodic) with respect to $\Pi$ if and only if it is stationary (resp.\ ergodic) with respect to $\wh\Pi$. 
Consequently, the version of the theorem for $\wh\Omega$ implies that for any $\Pi$-stationary probability measure $\mu$, there is a probability measure $\wh\pi_\mu$ on the space $\wh{\mcl M}_{\op{e}}$ of stationary ergodic probability measures for $\wh\Pi$ such that $\mu = \int_{\wh{\mcl M}_{\op{e}}} \wh\pi_\mu(d\nu)$. 
The measure $\wh\pi_\mu$ must assign full mass to elements $\nu$ of $\wh{\mcl M}_{\op{e}}$ with $\nu(\Omega)=1$. 
Each such stationary measure $\nu$ is stationary and ergodic for $\Pi$.   
This gives~\eqref{eqn-ergodic-decomp}. 

Since any two distinct elements of $\nu_{\wh{\mcl M}_{\op{e}}}$ are mutually singular and $\mcl M_{\op{e}}$ is a subset of $\wh{\mcl M}_{\op{e}}$ (under our identification), it follows that any two distinct elements of $\mcl M_{\op{e}}$ are mutually singular. 
\end{proof}

To apply Theorem~\ref{thm-ergodic-decomp} in our setting, we need to check some measurability statements, which we state now and prove (using standard arguments) at the end of this subsection.

\begin{figure}[t!]
 \begin{center}
\includegraphics[scale=.45]{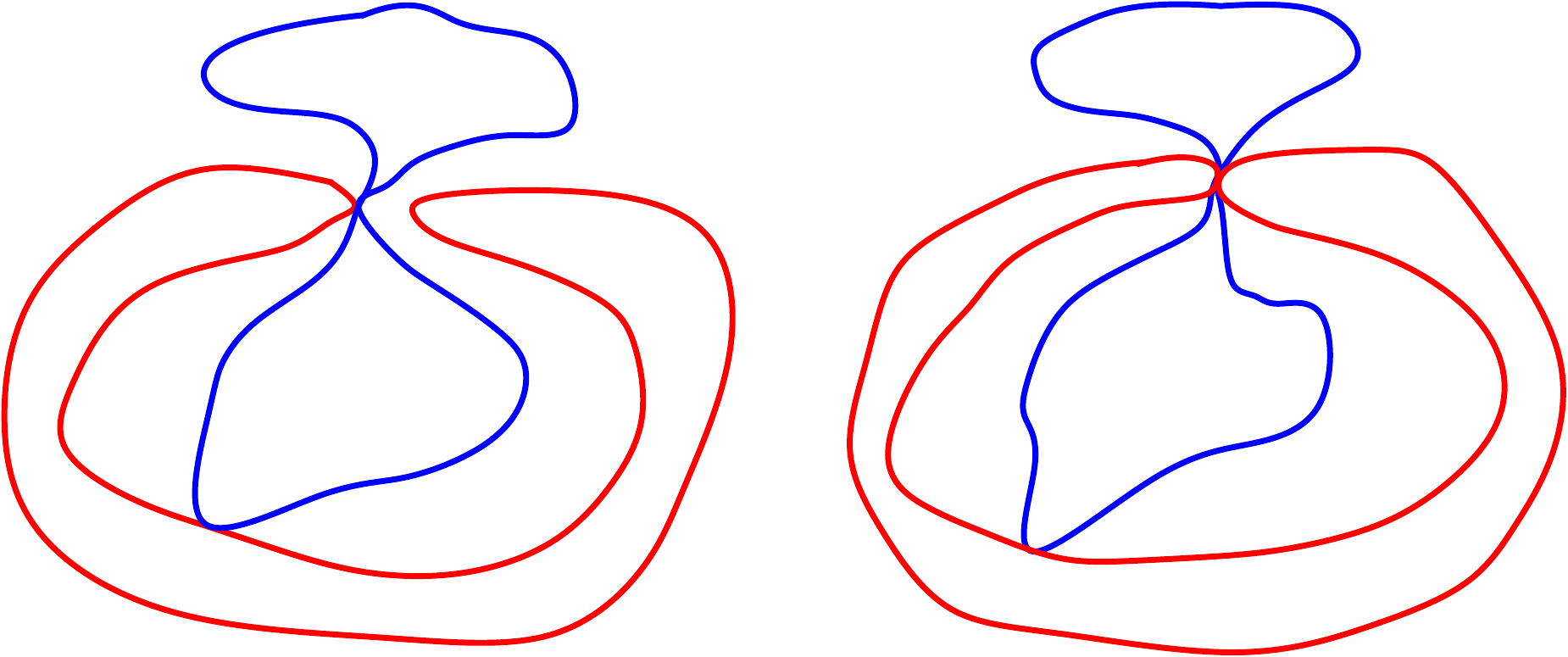}
\vspace{-0.01\textheight}
\caption{Illustration of why the set of locally finite, non-crossing loop configurations is \emph{not} a closed subset of the space of locally finite loop configurations. \textbf{Left:} A non-crossing configuration consisting of two loops. \textbf{Right:} A configuration consisting of two loops which is not non-crossing (since the blue loop intersects more than one complementary connected component of the red loop). The left picture can be made arbitrarily close to the right picture by making the red loop come arbitrarily close to hitting itself at its upper intersection point with the blue loop. }\label{fig-non-complete}
\end{center}
\vspace{-1em}
\end{figure} 

As illustrated in Figure~\ref{fig-non-complete}, the set of locally finite, non-crossing loop configurations is \emph{not} a closed subset of the space of locally finite loop configurations, so is not a complete metric space with respect to the metric of Section~\ref{sec-loop-prelim}. 
However, we do have the following much weaker statement. 

\begin{lem} \label{lem-non-crossing-msrble}
The space of non-crossing, locally finite loop ensembles on a domain $D\subset\BB C$ is a Borel measurable subset of the space of all locally finite loop ensembles with respect to the metric of~\eqref{eqn-lc-metric} (if $\ol D$ is compact) or~\eqref{eqn-lc-metric-loc} (if $\ol D$ is not compact).
\end{lem}

We note that Lemma~\ref{lem-non-crossing-msrble} together with the completeness of the space of locally finite loop ensembles (Lemma~\ref{lem-lc-metric}) shows that the space of non-crossing, locally finite loop ensembles is a Borel measurable subset of its completion. 
We also need the measurability of the transition kernel for our Markov chain.

\begin{lem} \label{lem-markov-msrble}
Let $\Phi$ be the operator which associates to each non-crossing, locally finite loop configuration $\Gamma_0$ the \emph{law} of the loop configuration $\Gamma_1$ produced by one step of the Markov chain introduced at the beginning of this section. 
If we endow the space of non-crossing, locally finite loop configurations with the topology of Section~\ref{sec-loop-config} and the space of probability measures on such loop configurations with the Prokhorov topology, then $\Phi$ is measurable.
\end{lem}

\begin{proof}[Proof of Proposition~\ref{prop-stationary}]
Lemmas~\ref{lem-non-crossing-msrble} and~\ref{lem-markov-msrble} show that the hypotheses of Theorem~\ref{thm-ergodic-decomp} are satisfied for the Markov chain in the proposition statement, defined on the space of non-crossing, locally finite loop configurations.
Consequently, any two distinct ergodic stationary probability measures for this Markov chain are mutually singular. Proposition~\ref{prop-coupling} implies that any two stationary probability measures for this Markov chain can be coupled together in such a way that they agree with positive probability.
Hence there can be only one ergodic stationary probability measure. 
On the other hand, Theorem~\ref{thm-ergodic-decomp} shows that any stationary probability measure for our Markov chain can be written as a convex mixture of ergodic stationary measures, hence there can be at most one stationary probability measure.
\end{proof}

Let us now check the measurability lemmas stated above. 
 
\begin{proof}[Proof of Lemma~\ref{lem-non-crossing-msrble}] 
For $n\in\BB N_0$, let $\mcl G_n$ be the set of $(n+2)$-tuples $\omega =  (\gamma_0^\omega , \gamma_1^\omega ,\dots,\gamma_n^\omega  ,\alpha_0^\omega)$  where $(\gamma_0^\omega, \gamma_1^\omega,\dots,\gamma_n^\omega)$ is an ordered collection of distinct loops and $\alpha_0^\omega$ is a proper arc of $\gamma_0^\omega$. We equip $\mcl G_n$ with the product metric corresponding to $n+1$ instances of the metric on loops and one instance of the metric on curves modulo time parameterization.  
We define sets $\mcl G_n^1, \mcl G_n^2, \mcl G_n^3 \subset \mcl G_n$ corresponding to the three conditions in Definition~\ref{def-non-crossing-ensemble}.
\begin{itemize}
\item $\mcl G_n^1$ is the set of $\omega\in\mcl G_n$ for which $\alpha_0^\omega$ does not have a non-trivial (i.e., more than a single point) sub-arc which is contained in $\ol{\wt\alpha_0^\omega \cup \gamma_1^\omega \cup\dots\cup\gamma_n^\omega}$, where $\wt\alpha_0^\omega$ denotes the complementary arc of $\alpha_0^\omega$ in $\gamma_0^\omega$.   
\item $\mcl G_n^2$ is the set of $\omega\in\mcl G_n$ for which $\alpha_0^\omega$ is contained in the closure of a single complementary connected component $U^\omega$ of $\BB C\setminus \ol{\wt\alpha_0^\omega \cup \gamma_1^\omega \cup\dots\cup\gamma_n^\omega}$.
\item $\mcl G_n^3 $ is the set of $\omega\in\mcl G_n^2$ for which the image of $\alpha_0^\omega$ under a conformal map $U^\omega \cup\bdy U^\omega \rta \ol{\BB D}$ is continuous. 
\end{itemize}   

We claim that each of $\mcl G_n^1$, $\mcl G_n^2$, and $\mcl G_n^3$ is a Borel measurable subset of $\mcl G_n$. 
Indeed, the compact set $ \ol{\wt\alpha_0^\omega \cup \gamma_1^\omega\cup\dots\cup\gamma_n^\omega}$ is a measurable function of $\omega$ (here we equip the space of compact subsets of $\BB C$ with the Hausdorff distance).
This immediately implies that $\mcl G_n^1$ and $\mcl G_n^2$ are measurable.  

We will now argue that $\mcl G_n^3$ is measurable. For $\omega \in \mcl G_n^2$, we can choose a conformal map $f^\omega :  \ol{\BB D} \rta U^\omega \cup \bdy U^\omega$ in such a way that $\omega \mapsto f^\omega$ is a Borel measurable function from $\mcl G_n^2$ to the space of continuous functions on $\BB D$ equipped with the topology of uniform convergence on compact subsets of $\BB D$.\footnote{For example, we can normalize $f^\omega$ as follows. Let $x \in \bdy U^\omega$ be the initial endpoint of $\alpha_0^\omega$. We require that $f^\omega(-i) = x$, $f^\omega(i)$ is the first point (prime end) on $\bdy U^\omega$ which lies at maximal distance from $f^\omega(-i)$ which we encounter when we traverse $\bdy U^\omega$ counterclockwise starting from $f^\omega(-i)$, and $f^\omega(1)$ is first point of $\bdy U$ which is equidistant from $f^\omega(i)$ and $f^\omega(-i)$ which we encounter when we traverse $\bdy U^\omega$ counterclockwise starting from $f^\omega(-i)$. }  

If $\omega \in \mcl G_n^1\cap \mcl G_n^2$, then $(f^\omega)^{-1}(\alpha_0^\omega)$ is the concatenation of countably many continuous curves (viewed modulo time parameterization) joining points of $\bdy\BB D$, corresponding to the images under $(f^\omega)^{-1}$ of the excursions of $\alpha_0^\omega$ away from $\bdy U$. 
By the continuity of $\alpha_0^\omega$, for each $\ep > 0$, there are only finitely many such curves in $\ol{\BB D}$ which have Euclidean diameter at least $\ep$.
Let $\beta_\ep^\omega$ be the continuous curve in $\ol{\BB D}$ (viewed modulo time parameterization) obtained by concatenating, in order, the curves with diameter at least $\ep$ along with the arcs of $\bdy\BB D$ which join the terminal and initial endpoints of the consecutive curves with diameter at least $\ep$. 
Then each $\beta_\ep^\omega$ is a continuous curve and is a measurable function of $\omega$. 
We have $\omega \in \mcl G_n^3$ iff $(f^\omega)^{-1}(\alpha_0^\omega)$ is continuous iff the curves $\beta_\ep^\omega$ converge to $(f^\omega)^{-1}(\alpha_0^\omega)$ modulo time parameterization as $\ep\rta 0$. This, in turn, is equivalent to the condition that the curves $\beta_\ep^\omega$ for $\ep \in \BB Q\cap (0,1) $ can be parameterized in such a way that they are equicontinuous and the maximal length of the time intervals on which these curves trace $\bdy\BB D$ tends to 0 as $\ep\rta 0$. 
Since each $\beta_\ep^\omega$ is a measurable function of $\omega$, this shows that $\mcl G_n^3$ is measurable. 

We will now deduce the measurability of the set of non-crossing loop ensembles from the measurability of each $\mcl G_n^1\cap\mcl G_n^2\cap \mcl G_n^3$. 
It is easy to see that there exists for each $n\in\BB N_0$ a countable collection of measurable functions $\{F_{n,m}\}_{m\in\BB N}$ from the space of locally finite loop ensembles into $\mcl G_n$ such that the set of non-crossing locally finite loop ensembles is precisely $\bigcap_{n=0}^\infty\bigcap_{m=1}^\infty F_{n,m}^{-1} (\mcl G_n^1\cap\mcl G_n^2\cap \mcl G_n^3)$.\footnote{To construct such functions, one can start by ordering the loops of $\Gamma$ according to their diameters, with ties broken by some measurable convention. By considering the possible ways of choosing $n+1$ loops of $\Gamma$, this gives a countable collection of measurable functions from $\Gamma$ to the set of ordered $n+1$-tuples of loops which output all of the possible ordered $n+1$-tuples of loops in $\Gamma$. 
We can then construct countably many measurable functions from the set of ordered $n+1$-tuples of loops to $\mcl G_n$ by choosing the arc $\alpha_0^\omega$ to be one of the $P$-excursions into $U$ (Definition~\ref{def-excursion}) of the first loop in the $n+1$-tuple, where $P$ ranges over all piecewise linear paths whose linear segments have rational endpoints and $U$ ranges over all open sets containing $P$ which are finite unions of Euclidean balls with rational centers and radii. }
Since each $\mcl G_n^1\cap \mcl G_n^2\cap \mcl G_n^3$ is measurable, this concludes the proof. 
\end{proof}

\begin{proof}[Proof of Lemma~\ref{lem-markov-msrble}]
Let us first observe that for any loop $\gamma$, compact set $P$, and open set $U\supset P$, the set of $P$-excursions of $\gamma$ into $U$ and the set of complementary $P$-excursions of $\gamma$ out of $U$, viewed as curves modulo time parameterization, is a measurable function of $\gamma$. 
Since $\Gamma_0$ is locally finite, it follows that the law of the point $x$ and hence also the law of the curve $\eta_x$ in the construction of $\Gamma_1$ is a measurable function of $\Gamma_0$ (recall that $(x,\eta_x)$ is chosen uniformly from a finite set of possibilities). 
Since also $\Gamma_0(P_+)$ and $\Gamma_0(P_-)$ are measurable functions of $\Gamma_0$, it follows that the set
\eqbn
 K  :=  \ol{\bigcup \left( \Gamma_0(P_\xi) \setminus \{\gamma_x \} \right) \cup \alpha_x }  ,
\eqen
defined as in~\eqref{eqn-markov-chain-union},
depends measurably on $\Gamma_0$, where here compact subsets of $\BB C$ are equipped with the Hausdorff distance. 

We will now argue that the law of the $\SLE_\kappa$ curve $\eta'$ is a measurable function of $\Gamma_0$. 
To this end, let $D$ be the connected component of $\BB A_\rho\setminus K$ with $x$ and $x^*$ on its boundary (so that $\eta'$ is an $\SLE_\kappa$ from $x$ to $x^*$ in $D$).  
Also let $z$ be the point of $P_-\cap D$ which is furthest from $\bdy D$ (with ties broken in some arbitrary measurable manner), and note that $z$ is a measurable function of $\Gamma_0$. 

For $n\in\BB N$, let $K^n$ be the closed union of the set of dyadic squares of side length $2^{-n}$ which intersect $K$ and let $D^n$ be the connected component of $\BB A_\rho\setminus K^n$ which contains $z$. 
Also let $x^n$ be the point of $\bdy D^n$ closest to $x$. 
Since there are only finitely many possibilities for $K^n$, and hence for $D^n$, and $x$ is a measurable functions of $\Gamma_0$, we see that $(D^n, x^n)$ is a measurable function of $\Gamma_0$. 

Since $\bdy D$ is a curve (which follows from the local finiteness of $\Gamma_0$), the conformal maps $f^n : \BB D\rta D^n$ taking $1$ to $x^n$ and 0 to $z$ converge uniformly to the conformal map $f : \BB D\rta D$ taking $1$ to $x$ and 0 to $z$ as $n\rta\infty$. Moreover, $f^n(f^{-1}(x^*)) \rta x^*$ uniformly. 
By the conformal invariance of the law of $\SLE_\kappa$ (viewed as a curve modulo time parameterization), we therefore get that the law of $\SLE_\kappa$ from $x$ to $x^*$ in $D^n$ converges to the law of $\SLE_\kappa$ from $x$ to $x^*$ in $D$. 
This gives the desired measurability of the law of $\eta'$. 

As a consequence, we find that the law of $\Gamma_1(P_\xi)$ is a measurable function of the law of $\Gamma_0$. 
By definition, the conditional law of $\Gamma_1$ given $\Gamma_1(P_\xi)$ is that of an independent $\SLE_\kappa$ in each of the connected components of $\BB A_\rho\setminus \ol{\bigcup \Gamma_1(P_\xi)}$. Each of these connected components is bounded by a curve and by local finiteness, only finitely many have diameter larger than each fixed $\ep > 0$. By approximating each such component by finite unions of small dyadic squares as above and using the conformal invariance of $\CLE_\kappa$, we find that the law of $\Gamma_1$ is a measurable function of $\Gamma_0$, as required.
\end{proof}

\subsection{Proof of inversion invariance}
\label{sec-inversion-proof}

We will now deduce Theorem~\ref{thm:main_result} from the results stated in Section~\ref{sec-cle-annulus}.
The basic idea is to use Theorem~\ref{thm-cle-markov} to find large annular sub-domains (regions between two loops) with the property that the restriction of the $\CLE_\kappa$ to the annular subdomain has the law of a $\CLE_\kappa$ in the annulus in the sense of Definition~\ref{def-cle-annulus}. 
We then apply Corollary~\ref{cor-annulus-inversion} to invert the $\CLE_\kappa$ in such an annular subdomain and take a limit as the domain increases to all of $\BB C \setminus \{0\}$. 
For technical reasons, it turns out to be more convenient to send the inner boundary of our domain to zero before sending the outer boundary to $\infty$, i.e., we first prove that the basic Markov property of Lemma~\ref{lem-cle-markov} holds for inverted $\CLE_\kappa$ (Lemma~\ref{lem-inverse-cle-markov}) then conclude the proof by looking at the origin-containing-components of larger and larger loops. 

Throughout this subsection, we let $\Gamma$ be a whole-plane $\CLE_\kappa$ and we let $\wh\Gamma$ have the law of the image of $\Gamma$ under $z\mapsto 1/z$. We seek to show that $\Gamma \eqD \wh\Gamma$. 

For $r  > 0$, let $  \gamma^r$ (resp.\ $ \gamma_r$) be the outermost (resp.\ innermost) origin-surrounding loop of $ \Gamma$ which intersects $\bdy B_r(0)$. 
If $r\in (0,1)$ and $ \gamma^1$ and $ \gamma_r$ do not intersect (which happens with probability tending to 1 as $r\rta 0$), then there is a unique connected component of $\BB C\setminus ( \gamma^1\cup  \gamma_r)$ which has the topology of an annulus. Let $  D_r$ be this connected component and otherwise (if $ \gamma^1\cap \gamma_r \not=\emptyset$) let $  D_r = \emptyset$. 
Define $\wh\gamma^r$, $\wh\gamma_r$, and $\wh D_r$ in an analogous manner but with the inverted $\CLE_\kappa$ $\wh\Gamma$ in place of $ \Gamma$.

\begin{lem} \label{lem-annulus-coupling}
In the notation introduced just above, we can find for each $r\in (0,1)$ a coupling of $\Gamma$ and $\wh\Gamma$ such that the following is true. We have $\{D_r\not=\emptyset\} = \{\wh D_r \not=\emptyset\}$ and on this event there is a.s.\ a conformal map $f_r : D_r\rta \wh D_r$ which takes $ \Gamma|_{ D_r}$ to $\wh\Gamma|_{\wh D_r}$ and takes the (a.s.\ unique) leftmost point of $\bdy D_r$ to the leftmost point of $\bdy\wh D_r$ (this last choice of normalization is arbitrary). 
Furthermore, the loop $\gamma^1$ and the loops of $\Gamma$ which it disconnects from 0 are independent from the loop $\wh\gamma^1$ and the loops of $\wh\Gamma$ which it disconnects from 0.
\end{lem}
\begin{proof}
Let $  M_r$ be the number of origin-surrounding loops of $ \Gamma$ which are contained in $D_r$ and similarly define $\wh M_r$. 
Theorem~\ref{thm-cle-markov} together with Lemma~\ref{lem-cle-markov} implies that if we condition on $D_r$ and $M_r$, then on the event $\{ D_r \not=\emptyset\}$, the conditional law of $ \Gamma|_{ D_r}$ is that of a $\CLE_\kappa$ on $ D_r$ with $M_r$ inner-boundary-surrounding loops, as in Definition~\ref{def-cle-annulus}.
Furthermore, $\Gamma|_{D_r}$ is conditionally independent from $\Gamma\setminus \Gamma|_{D_r}$. 
By the scale invariance of the law of whole-plane $\CLE_\kappa$, the law of $(\Gamma , D_r ,  M_r)$ is the same (modulo scaling of $(\Gamma , D_r)$) if we replace $\gamma^1$ and $\gamma_r$ by $\gamma^{1/r}$ and $\gamma_1$. 
By applying an inversion map along with Corollary~\ref{cor-annulus-inversion}, we therefore get the following for each $r\in (0,1)$.
\begin{enumerate}
\item The law of the conformal moduli of $D_r$ and $\wh D_r$ coincide (here we define the conformal modulus of the empty set to be 0). 
\item The laws of $M_r$ and $\wh M_r$ coincide.
\item The conditional law of $\wh\Gamma|_{\wh D_r}$ given $\wh D_r$ and $\wh M_r$ on the event $\{\wh D_r\not=\emptyset\}$ is that of a $\CLE_\kappa$ on $\wh D_r$ with $\wh M_r$ inner-boundary-surrounding loops. 
\end{enumerate}
The statement of the lemma follows.
\end{proof}

Let $D$ (resp.\ $\wh D$) be the connected component of $\BB C\setminus \gamma^1$ (resp.\ $\BB C\setminus \wh\gamma^1$) which contains 0.
If we have coupled $\Gamma$ and $\wh\Gamma$ as in Lemma~\ref{lem-annulus-coupling}, we let $f : D\rta\wh D$ be the unique conformal map which takes $D$ to $\wh D$, which fixes the origin and takes the (a.s.\ unique) leftmost point of $\bdy D$ to the leftmost point of $\bdy\wh D$. 
Note that the law of $f$ does not depend on $r$ since in our coupling $D$ and $\wh D$ are independent. 

\begin{lem} \label{lem-coupling-conv}
Suppose $r\in (0,1)$ and we have coupled $\Gamma$ and $\wh\Gamma$ as in Lemma~\ref{lem-annulus-coupling}.
As $r\rta 0$, we have that $\max_{z\in D} (f(z) -f_r(z) ) \rta 0$ in law. 
\end{lem}

To prove Lemma~\ref{lem-coupling-conv}, we will need the following basic complex analysis lemma.

\begin{lem} \label{lem-domain-invert}
Suppose that $\{\mcl A_n\}_{n \in \BB N}$ and $\{\wh{\mcl A}_n\}_{n\in\BB N}$ are two sequences of sub-domains of $\BB D$ such that $\mcl A_n$ and $\wh{\mcl A}_n$ are each conformally equivalent to an annulus with the same modulus. 
Suppose further that the outer boundaries of $\mcl A_n$ and $\wh{\mcl A}_n$ are each equal to $\bdy\BB D$; for each $n\in\BB N$, there exists $C_n >1$ and $\delta_n \in (0,1)$ such that 
\eqb \label{eqn-domain-invert}
\BB D\setminus B_{C_n \delta_n   }(0)  \subset  \mcl A_n \subset \BB D\setminus B_{\delta_n / C_n }(0), \quad \delta_n \rta0,  \quad \text{and} \quad C_n  = o_n(\delta_n^{-1}) \: \text{as $n\rta\infty$} ; 
\eqe
and there exists $\wh C_n >1$ and $\wh\delta_n \in (0,1)$ such that the same is true with $\wh{\mcl A}_n$ in place of $\mcl A_n$. 
For $n\in\BB N$, let $f_n : \mcl A_n \rta \wh{\mcl A}_n$ be the conformal map which takes the inner (resp.\ outer) boundary of $\mcl A_n$ to the inner (resp.\ outer) boundary of $\wh{\mcl A}_R$, normalized so that $f_n(1 ) = 1$. 
Then $f_n$ converges uniformly to the identity map on $\ol{\BB D} \setminus \{0\} $, at a rate depending only on $\delta_n$ and $C_n$.  
\end{lem}
\begin{proof}
By conformally mapping each of $\mcl A_n$ and $\wh{\mcl A}_n$ to a set of the form $\BB A_{\rho_n} = \BB D\setminus B_{\rho_n}(0)$ for appropriate $\rho_n > 0$, we see that it suffices to prove the statement of the lemma with $\BB A_{\rho_n}$ in place of $\wh{\mcl A}_n $. 

Let us first note that the hypothesis~\eqref{eqn-domain-invert} implies that $\delta_n/C_n \leq \rho_n \leq C_n\delta_n$, and hence that $\log \delta_n / \log \rho_n \rta 1$ as $n\rta\infty$: indeed, otherwise one of the annular domains $\mcl A_n$ or $\BB A_{\rho_n}$ would be conformally equivalent to a proper subdomain of itself.

By the Gambler's ruin formula, for $z\in \mcl A_n$, $\log| f_n(z)| / \log \rho_n $ is equal to the probability that a Brownian motion started from $z$ hits the inner boundary of $\mcl A_n$ before the outer boundary. By~\eqref{eqn-domain-invert},
\eqbn
 \frac{\log |z|  }{\log (\delta_n /  C_n) } \leq   \frac{\log| f_n(z)| }{ \log \rho_n} \leq  \frac{ \log |z|  }{\log ( \delta_n C_n) }  ,
 \eqen
so $\log |f_n(z)|  = (1 + o_n(1)) \log |z|$ and hence $\log |f_n(z)| \rta \log |z|$ uniformly on compact subsets of $\ol{\BB D} \setminus \{0\}$. 

We will now bound the gradient of $ \log |f_n(z)| - \log |z|  $ using the following standard estimate for harmonic functions: if  $u$ is harmonic on $B_r(z_0)$, then
\eqb \label{eqn-cac}
\sup_{z \in B_{r/2} (z_0)} |\nabla u(z)| \leq C \sup_{z\in B_r(z_0)} |u(z) -u(z_0)| 
\eqe
for a universal constant $C>0$. We can extend $f_n$ by Schwarz reflection to be conformal on the union of $\mcl A_n$ and its reflection across $\bdy\BB D$.
Applying~\eqref{eqn-cac} and the conclusion of the preceding paragraph to finitely many Euclidean balls contained in this union, with the harmonic function $u = \log |f_n(\cdot)| - \log|\cdot|$, shows that $|\nabla (\log |f_n(z)| - \log |z|) | \rta 0$ uniformly on compact subsets of $\ol{\BB D}\setminus \{0\}$. 

By the Cauchy-Riemann equations, $|\nabla(\op{arg} f_n(z)  - \op{arg} z)| \rta 0$ uniformly on compact subsets of $\ol{\BB D} \setminus \{0\}$. Since $f_n(1) = 1$, this shows that $f_n(z) \rta z$ uniformly on compact subsets of $\ol{\BB D} \setminus \{0\}$. 
In particular, $f_n(z) \rta z$ uniformly on $\ol{\BB D}\setminus B_\ep(0)$ for each $\ep >0$ and the diameter of $f_n(B_\ep(0))$ tends to zero as $n\rta\infty$ and then $\ep\rta 0$. This shows that $f_n \rta f$ uniformly on all of $\ol{\BB D} \setminus \{0\}$. 
\end{proof}

\begin{proof}[Proof of Lemma~\ref{lem-coupling-conv}]
Let $\phi : D\rta  \BB D$ and $\wh\phi : \wh D\rta \BB D$ be the conformal maps which take 0 to 0 and the leftmost points of $\bdy D$ and $\bdy \wh D$, respectively, to $1$. 

We claim that for each $r\in (0,1)$, there exists a random $C_r > 1$ and $\delta_r > 0$ such that with probability tending to 1 as $r\rta 0$, 
\eqbn
\BB D \setminus B_{C_r \delta_r  }(0)  \subset  \phi(D_r) \subset \BB D\setminus B_{\delta_r  / C_r }(0), \quad \delta_r  = o_r(1) ,  \quad \text{and} \quad C_r  = o_r(\delta_r^{-1}) ;
\eqen
and there exists a random $\wh C_r > 1$ and $\wh\delta_r > 0$ such that the same is true with $\wh\Gamma$ in place of $\Gamma$.
Given the claim, we can apply Lemma~\ref{lem-domain-invert} to the domains $\phi(D_r)$ and $\wh\phi(\wh D_r)$ to find that the law of $\wh\phi  \circ f_r \circ  \phi^{-1}$ under our coupling converges uniformly to the identity map. Since $f = \wh\phi^{-1} \circ \phi$, this shows that $\max_{z\in D} (f-f_r) \rta 0$ in law. 

It remains to prove the above claim. 
We will prove the statement for $\Gamma$; the argument for $\wh\Gamma$ is identical.
By the scale invariance of the law of $\CLE_\kappa$, it holds with probability tending to 1 as $C\rta\infty$, uniformly in $r$, that $\gamma_r \subset B_{C r}(0)\setminus B_{r/C}(0)$. 
The Koebe distortion theorem applied to the conformal map $\phi$ shows that with probability tending to 1 as $r\rta 0$ and then $C\rta\infty$, we can find $\delta_r > 0$ (in particular, $\delta_r =  |\phi'(0)| r$) such that $\BB D \setminus B_{C \delta_r  }(0)  \subset  \phi(D_r) \subset \BB D\setminus B_{\delta_r  / C }(0)$. Sending $C \rta \infty$ sufficiently slowly as $r\rta 0$ concludes the proof.
\end{proof}

We can now prove that the analog of Lemma~\ref{lem-cle-markov} holds for the inverted $\CLE_\kappa$ $\wh\Gamma$. 

\begin{lem} \label{lem-inverse-cle-markov}
For each $R>0$, if we condition on the outermost origin-surrounding loop $\wh\gamma^{R} \in\wh\Gamma$ which intersects $\bdy B_R(0)$, then the conditional law of the restriction of $\wh\Gamma$ to the connected component of $\BB C\setminus \wh\gamma^R$ containing 0 is that of a $\CLE_\kappa$ in this connected component. 
\end{lem}
\begin{proof}
The statement of the lemma for $R =1$ follows from Lemma~\ref{lem-coupling-conv} and the fact that the conditional law of $\Gamma|_D$ given $\gamma^1$ is that of a $\CLE_\kappa$ in $D$ (Lemma~\ref{lem-cle-markov}). 
The statement for general values of $R$ follows from the scale invariance of the law of $\wh\Gamma$. 
\end{proof}

We are now ready to prove the main theorem. 

\begin{proof}[Proof of Theorem~\ref{thm:main_result}]
Let $\wh U_R$ for $R > 0$ be the connected component of $\BB C\setminus \wh \gamma^R$ containing the origin. By Lemma~\ref{lem-inverse-cle-markov}, the conditional law of $ \wh\Gamma|_{\wh U_R}$ given $\wh\gamma^R$ is that of a $\CLE_\kappa$ in $\wh U_R$. Almost surely, domains $\wh U_R$ increase to all of $\BB C$. 
By~\cite[Theorem A.1]{mww-nesting}, $\wh\Gamma|_{\wh U_R}$ converges in law to whole-plane $\CLE_\kappa$, so $\wh \Gamma\eqD\Gamma$.  
\end{proof}

\appendix

\section{Basic lemmas for SLE and CLE}
\label{sec-sle-cle-lemmas}

In this appendix we prove a number of basic properties of SLE and CLE which are used in Sections~\ref{sec-markov} and~\ref{sec-resampling} and recorded here to avoid interrupting the main argument. 
This section does not use any of the results proven elsewhere in the paper, although we do use some of the notation from Section~\ref{sec-basic}.

\subsection{SLE stays close to a simple path with positive probability}
\label{sec-pos-prob}

We will make frequent use of the following slight extension of~\cite[Lemma 2.5]{miller-wu-dim}.  

\begin{lem} \label{lem-miller-wu-dim}
Let $\kappa > 0$, let $\rho^L , \rho^R   \in (-2) \vee (\kappa/2-4)$ (which is the range for which $\SLE_\kappa(\rho^L;\rho^R)$ does not fill the boundary of its domain~\cite{ig4,dubedat-duality}), and let $\eta$ be a chordal $\SLE_\kappa(\rho^L;\rho^R)$ from $-i$ to $i$ in $\BB D$ (with arbitrary force point locations). Let $P : [0,1]\rta \ol{\BB D}$ be a simple path from $-i$ to $i$. For each $\ep >0$, it holds with positive probability that the distance from $\eta$ to $P$ with respect to the metric on curves modulo time parameterization is at most $\ep$. 
\end{lem}  
\begin{proof}
Since $P$ is a simple path, we can find simply connected open sets $U_0,\dots,U_n$ which cover $P$, are entered in order by $P$, each have diameter at most $\ep/2$, and satisfy $U_i \cap U_j \not=\emptyset$ if and only if $|i-j| = 1$. Indeed, to construct such sets, one can set $t_0 = 0$ and inductively let $t_j = 1 \wedge \min\{t\geq t_{j-1} : |P(t_j) - P(t_{j-1})| \geq \ep/4\}$. Using that $P$ is a simple curve, one can then choose $\delta > 0$ small enough that the $\delta$-neighborhoods of the sets $P([t_{j-1} , t_j])$ satisfy the desired properties. 
By~\cite[Lemma 2.5]{miller-wu-dim}, it holds with positive probability that $\eta$ enters $U_1$ before exiting $U_0$. Iterating this and applying the domain Markov property of $\SLE_\kappa(\rho^L;\rho^R)$, we see that with positive probability, $\eta$ enters $U_j$ before exiting $U_{j-1}$ for each $j=1,\dots, n$.
Conditioned on this, it is easily seen from~\cite[Lemma 2.5]{miller-wu-dim} that with positive conditional probability, $\eta$ reaches its target point before exiting $U_n$ (this is the only place where we use that $\eta$ is not boundary filling). 
If this is the case, then by parameterizing $\eta$ and $P$ so that they each take the same amount of time between first entering $U_{j-1}$ and $U_j$ for each $j=1,\dots,n$, we get that the distance from $\eta$ to $P$ with respect to the metric on curves modulo time parameterization is at most $\ep$, as required.
\end{proof}

When studying CLE, Lemma~\ref{lem-miller-wu-dim} is often useful in conjunction with the following lemma, which is proven as part of the proof of~\cite[Theorem 5.4]{shef-cle}.

\begin{lem}[\!\! \cite{shef-cle}] \label{lem-cle-concatenate}
Fix $\kappa \in (4,8)$.  Let $\Gamma$ be a $\CLE_\kappa$ on $\BB H$ and let $x ,y \in \BB R$ with $x < y$. 
Let $\Gamma^{\op{out}}(I)$ be the set of loops in $\Gamma$ which intersect $I$ and which are maximal in the sense that the interval $[\inf (\gamma\cap I) , \sup(\gamma\cap I)]$ is not contained in $[\inf(\gamma'\cap I) , \sup(\gamma'\cap I)]$ for any $\gamma \in \Gamma\setminus \{\gamma'\}$. 
Let $\eta$ be the curve obtained by concatenating, in order, the clockwise arcs of the loops $\gamma \in \Gamma^{\op{out}}(I)$ from $\inf(\gamma \cap I)$ to $\sup(\gamma\cap I)$. 
Then $\eta$ is a chordal $\SLE_\kappa(\kappa-6)$ from $x$ to $y$ with the force point started immediately to the right of~$x$.
\end{lem}

We call the curve $\eta$ from Lemma~\ref{lem-cle-concatenate} the \emph{branch from $x$ to $y$ of the branching $\SLE_\kappa(\kappa-6)$ process associated with $\Gamma$}.

\subsection{Hitting lemmas for SLE and CLE}
\label{sec-hitting}

The following lemma is used in the proof of Lemma~\ref{lem:general-markov}. 

\begin{lem} \label{lem-cle-hit}
Fix $\kappa \in (4,8)$.  Let $\Gamma $ be a $\op{CLE}_{\kappa }$ on $\BB D$. For each $\ep > 0$, there exists $\delta = \delta(\ep) > 0$ such that for each Borel set $J\subset \partial\BB D $ with 1-dimensional Lebesgue measure at least $\ep$, it holds with probability at least $\delta$ that the outermost origin-surrounding loop in $\Gamma$ intersects $J$. 
\end{lem}

To prove Lemma~\ref{lem-cle-hit}, we first need an analogous statement for SLE.

\begin{lem} \label{lem-sle-hit}
Fix $\kappa \in (4,8)$.  Let $\eta$ be an $\op{SLE}_{\kappa}$ from $-i$ to $i$ in $\BB D$ and fix an arc $A\subset\bdy\BB D$ which lies at positive distance from $-i$. 
For each $\ep > 0$, there is a $\delta = \delta(\ep , A) > 0$ such that for any set $J\subset\bdy\BB D \setminus A$ with 1-dimensional Lebesgue measure at least $\ep$, it holds with probability at least $\delta$ that $\eta$ hits $J$ before $A$. 
\end{lem}
\begin{proof}
By possibly replacing $J$ by $J\setminus (B_{\ep/100}(i) \cup B_{\ep/100}(-i))$, we can assume without loss of generality that $J$ lies at distance at least $\ep/100$ from each of $-i$ and $i$. By reflection symmetry, we can also assume without loss of generality that $J$ is contained in the right semi-circle of $\bdy\BB D$. 
Let $f$ the conformal map from $\BB D$ to $\BB H$ which takes $-i$ to 0, $i$ to $\infty$, and the right endpoint of $B_{\ep/100}(-i)\cap\bdy\BB D$ to 1.
By our assumptions on $J$, we see that the 1-dimensional Lebesgue measure of $f(J)$ is bounded below by a parameter depending only on $\ep$.
Furthermore, if $R_\ep$ denotes the image under $f$ of the right endpoint of $B_{\ep/100}(i)\cap\bdy\BB D$, then $f(J)\subset [1, R_\ep]$. 

It is easily seen from a simple scaling argument that the law of the first place where $f(\eta)$ hits $[1,\infty)$ is mutually absolutely continuous w.r.t.\ Lebesgue measure on $[1,\infty)$.  

This implies that if we condition on the positive probability event that $f(\eta)$ hits $[1,R_\ep] \setminus f(A)$ before hitting $f(A)$, then the conditional probability that the first hitting location of $[1,\infty)$ lies in $f(J)$ is bounded below by a constant depending only on $\ep$ and $A$. 
The statement of the lemma follows.
\end{proof} 

\begin{proof}[Proof of Lemma~\ref{lem-cle-hit}]
Let $I\subset \bdy\BB D$ be the arc of length $\ep/100$ centered at $1$ and let $J'$ be the set of points in $J$ which lie at Euclidean distance at least $\ep/100$ from $J$. Then the Lebesgue measure of $J'$ is at least $\ep/2$. We will prove a lower bound for the probability that the outermost origin-surrounding loop hits $J'$ using Lemma~\ref{lem-sle-hit}. 
 
Let $x$ and $y$ be the endpoints of $I$ in clockwise order and let $\eta$ be the chordal $\SLE_\kappa(\kappa-6)$ from $x$ to $y$ obtained by concatenating arcs of $\Gamma$ which intersect $I$, as in Lemma~\ref{lem-cle-concatenate}.
By the definition in Lemma~\ref{lem-cle-concatenate}, each excursion of $\eta$ away from $I$ is contained in a single outermost loop of $\eta$.
If there is some such excursion which surrounds 0 in the clockwise direction and hits $J'$, then this excursion must be part of the outermost origin-surrounding loop and hence the origin-surrounding loop must intersect $J'$. 

We will now lower-bound the probability that an excursion as in the preceding sentence exists. 
Let $\tau$ be the first time that $\eta$ surrounds 0 in the clockwise direction (or $\tau = \infty$ if no such time exists). 
On the event $\{\tau < \infty\}$, let $z$ be the rightmost point of $\eta([0,\tau])\cap I$ (i.e., the point closest to $y$) and let $U$ be the connected component of $\BB D\setminus \eta([0,\tau])$ with $\eta(\tau)$ and $z$ on its boundary. By the strong Markov property and the target invariance of $\SLE_\kappa(\kappa-6)$~\cite{sw-coord}, the conditional law of the segment of $\eta$ from time $\tau$ until the first time after $\tau$ at which it hits $I$ agrees in law with a chordal $\SLE_\kappa$ from $\eta(\tau)$ to $z$ in $U$ run until it hits $I$. 

By Lemma~\ref{lem-miller-wu-dim}, we can find $\zeta  = \zeta(\ep ) >0$, such that with probability at least $\zeta$, it holds that $\eta([0,\tau])\cap\bdy\BB D$ is contained in the $\ep/100$-neighborhood of $I$ (so is disjoint from $J'$) and the inverse of the conformal map $f : U\rta \BB H$ which takes $\eta(\tau)$ to $-i$, $z$ to $i$, and $y$ to $-1$ is $1/\zeta$-Lipschitz on $f(\bdy \BB D\setminus B_{\ep/100}(I))$. 
If this is the case, then $f(J')$ has Lebesgue measure at least $\zeta \ep/2$. By the preceding paragraph and Lemma~\ref{lem-sle-hit} (applied to the image under $f$ of the segment of $\eta$ from $\tau$ until the first time after $\tau$ at which it hits $I$ and with $A$ the clockwise arc of $\bdy\BB D$ from $-1$ to $i$, which contains $f(I\cap\bdy U)$), there is a $\wt\delta = \wt\delta(\zeta \ep/2) > 0$ such that on the above event it holds with conditional probability at least $\wt\delta$ given $\eta|_{[0,\tau]}$, it holds that $\eta$ hits $J'$ after time $\tau$ and before returning to $I$. Hence the statement of the lemma holds with $\delta = \zeta\wt\delta$. 
\end{proof}

\subsection{Absolute continuity for SLE and CLE}
\label{sec-abs-cont}

Here we prove several lemmas about SLE and CLE which are used in Section~\ref{sec-resampling}. 
Our starting point is the a basic absolute continuity property of $\SLE_\kappa$, which is a consequence of, e.g.,~\cite[Lemma 2.8]{miller-wu-dim}. 
See Figure~\ref{fig-abs-cont}(a) for an illustration.

\begin{figure}[ht!] 
 \begin{center}
\includegraphics[scale=.85]{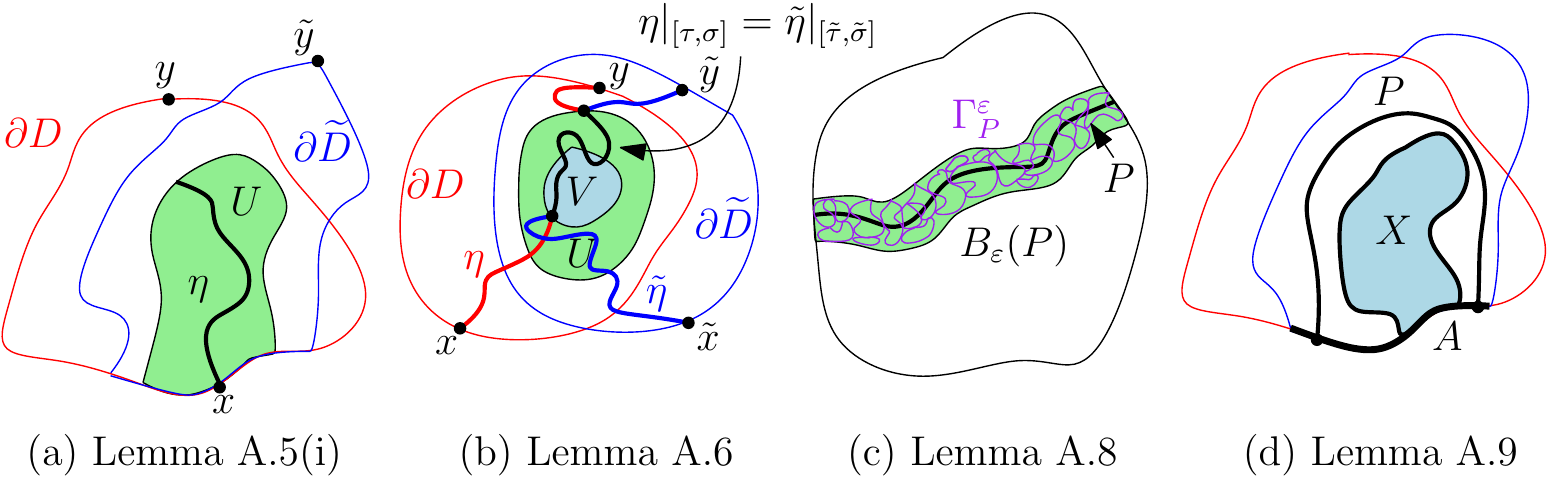} 
\caption{Illustrations of the setups for several of the lemmas in Section~\ref{sec-sle-cle-lemmas}. }\label{fig-abs-cont}

\end{center}
\end{figure}

\begin{lem} \label{lem-mw-sle-abs-cont}
Let $D,\wt D$ be simply connected domains, not all of $\BB C$, let $U$ be a connected open subset of $D\cap \wt D$ which is at positive distance from $\bdy D\setminus \bdy \wt D$ and $\bdy \wt D\setminus \bdy D$, and suppose $x \in \bdy D\cap \bdy \wt D\cap \bdy U$. Also let $\kappa > 0$ and $\rho^L , \rho^R > -2$. 
\begin{enumerate}
\item If $y\in \bdy D\setminus \bdy \wt D$, and $\wt y\in \bdy\wt D\setminus \bdy D$, the laws of chordal $\SLE_\kappa(\rho^L;\rho^R)$ (with force points immediately to the left and right of the starting point) from $x$ to $y$ in $D$ and from $x$ to $\wt y$ in $\wt D$ each run until it exits $U$ are mutually absolutely continuous.
\item Suppose $y \in \bdy D\cap \bdy \wt D\cap \bdy U$ and $y\not=x$.  
Then the laws of  chordal $\SLE_\kappa(\rho^L;\rho^R)$ (with force points immediately to the left and right of the starting point) from $x$ to $y$ in $D$ and from $x$ to $ y$ in $\wt D$ each run until it either hits $y$ or exits $U$ are mutually absolutely continuous.
\end{enumerate}
\end{lem}

Combining Lemmas~\ref{lem-miller-wu-dim} and~\ref{lem-mw-sle-abs-cont} with the reversibility of SLE yields the following coupling statement for a ``middle" segment of $\SLE_\kappa$ curves started from different points, which is illustrated in Figure~\ref{fig-abs-cont}(b).

\begin{lem} \label{lem-sle-abs-cont}
Let $\kappa \in (0,8)$. Let $D,\wt D\subset \BB C$ be simply connected domains, not all of $\BB C$. Let $ V \subset U \subset D \cap \wt D$ be connected open domains such that $V$ lies at positive distance from $\bdy U$ and $U$ lies at positive distance from $\bdy D\setminus \bdy \wt D$ and from $\bdy \wt D\setminus \bdy D$.  
Let $\eta$ (resp.\ $\wt\eta$) be a chordal $\SLE_\kappa$ in $D$ (resp.\ $\wt D$) between two points of $\bdy D $ (resp.\ $\bdy \wt D$) which lie at positive distance from $U$.
Let $\tau$ be the first time that $\eta$ enters $V$ and let $\sigma$ be the first time after $\tau$ at which $\eta$ exits $\ol U$. Define $\wt\tau$ and $\wt\sigma$ similarly with $\wt\eta$ in place of $\eta$. 
There is a coupling of $\eta$ and $\wt \eta$ such that 
\eqbn
\BB P\left[  \eta|_{[\tau,\sigma]} = \wt\eta|_{[\wt\tau,\wt\sigma]}  \right] > 0 .
\eqen
In fact, we can arrange that the following stronger statement is true. If $P$ is any simple path in $\ol U$ between a point of $\bdy V$ and a point of $\bdy U$, then
\eqb \label{eqn-sle-abs-cont}
\BB P\left[  \eta|_{[\tau,\sigma]} = \wt\eta|_{[\wt\tau,\wt\sigma]} ,\, \BB d(\eta|_{[\tau,\sigma]} ,P) \leq \ep \right] > 0 ,
\eqe 
where $\BB d$ is the metric on curves modulo time parameterization. 
\end{lem}
\begin{proof}
Let $x$ and $y$ (resp.\ $\wt x$ and $\wt y$) be the initial and terminal points of $\eta$ (resp.\ $\wt\eta$). 
We first treat the case when $\partial D\cap \partial \wt D$ contains a non-trivial connected arc and the starting points $x = \wt x$ is a point of such an arc. By Lemmas~\ref{lem-miller-wu-dim} and~\ref{lem-mw-sle-abs-cont}, we get the following stronger statement. Let $U' \supset U$ be any connected sub-domain of $D$ with $x$ on its boundary which lies at positive distance from $\bdy D\setminus \bdy \wt D$ and from $\bdy \wt D\setminus \bdy D$ and let $\sigma'$ and $\wt\sigma'$ denote the exit times of $\eta$ and $\wt\eta'$ from $U'$. Let $P'$ be a simple path from $x$ to a point of $\bdy U'$. Then the laws of $\eta|_{[0,\sigma']}$ and $\wt \eta |_{[0,\wt\sigma']}$ are mutually absolutely continuous and we can couple $\eta$ and $\wt\eta$ in such a way that
\eqb \label{eqn-sle-abs-cont'}
\BB P\left[  \eta|_{[0,\sigma']} = \wt\eta|_{[0,\wt\sigma']} ,\, \BB d(\eta|_{[0,\sigma']} ,P') \leq \ep  \right] > 0 .
\eqe 

Using Lemma~\ref{lem-miller-wu-dim} and the domain Markov property, we can arrange that with positive probability, the event in~\eqref{eqn-sle-abs-cont'} holds and neither $\eta$ nor $\wt\eta$ enters $V$ after exiting $U'$. This means that the segment of the time-reversal of $\eta$ (resp.\ $\wt\eta$) after the first time it enters $V$ is contained in $\eta([0,\sigma'])$ (resp.\ $\wt\eta([0,\wt\sigma'])$. 
Suppose now that we choose the path $P'$ so that the segment of $P'$ between its last exit time from $V$ and the last time before this time when it enters $U$ is equal to $P$. 
By the reversibility of $\SLE_\kappa$ (see~\cite{zhan-reversibility} or~\cite{ig2} for the case when $\kappa \leq 4$ and~\cite{ig3} for the case when $\kappa \in (4,8)$)
and~\eqref{eqn-sle-abs-cont'}, we find that in the case when $\partial D\cap \partial \wt D$ contains a non-trivial arc, $y= \wt y$ is a point of this arc, and $x$ and $\wt x$ are arbitrary, the laws of $\eta|_{[\tau,\sigma]}$ and $ \wt\eta|_{[\wt\tau,\wt\sigma]}$ are mutually absolutely continuous and we can couple $\eta$ and $\wt\eta$ so that~\eqref{eqn-sle-abs-cont} holds with positive probability.

Now suppose that $\bdy D\cap \bdy\wt D$ contains a non-trivial connected arc, but that $y$ and $\wt y$ need not be equal or contained in this arc. Choose a point $y'$ which lies in a non-trivial connected arc of $\bdy D\cap \bdy\wt D$.  By the preceding case, the statement of the lemma is true if we target our $\SLE_\kappa$ curves at $y'$ instead of at $y$ and $\wt y$. By Lemma~\ref{lem-mw-sle-abs-cont}, the laws of $\eta|_{[\tau,\sigma]}$ is mutually absolutely continuous w.r.t.\ the law of the corresponding segment of $\SLE_\kappa$ in $D$ from $x$ to $y'$. A similar statement holds with $\wt\eta$ in place of $\eta$ and $(\wt D , \wt x , y')$ in place of $(D, x , y')$.
Therefore, the statement of the lemma holds in this case. 

In general, we can modify $\bdy D$ and $\bdy \wt D$ away from $U$, $x, y , \wt x , \wt y$ in such a way that $\bdy D\cap \bdy \wt D$ contains a non-trivial connected boundary arc and apply Lemma~\ref{lem-mw-sle-abs-cont} to compare $\SLE_\kappa$ in the modified domains to $\SLE_\kappa$ in the original domains.  
\end{proof}

We next prove some basic properties of CLE which are analogous to those stated above for SLE. 
   
\begin{lem} \label{lem-cle-small}
Let $\kappa\in (4,8)$ and let $\Gamma$ be a $\CLE_\kappa$ on a simply connected domain $D\subset \BB C$ bounded by a Jordan curve. For each $\ep > 0$, it holds with positive probability that each loop in $\Gamma$ has diameter at most $\ep$. 
\end{lem}

To prove Lemma~\ref{lem-cle-small}, we will iteratively build a ``grid" of small loops of $\Gamma$ such that the complementary connected components of the union of the loops in the grid are small. The following lemma allows us to build a single path in this grid. See Figure~\ref{fig-abs-cont}(c) for an illustration.

\begin{lem} \label{lem-cle-small-path}
Let $\kappa \in (4,8)$ and let $\Gamma$ be a $\CLE_\kappa$ on $\BB D$. 
Let $P : [0,1] \rta \BB D$ be a simple path which does not hit $\bdy\BB D$ except at its endpoints and let $\ep > 0$. There is a random collection of loops $\Gamma_P^\ep \subset \Gamma$  such that a.s.\ each connected component of  $\BB D \setminus \ol{\bigcup \Gamma_P^\ep}$ is simply connected, if we condition on $\Gamma_P^\ep$, the conditional law of $\Gamma\setminus \Gamma_P^\ep$ is that of an independent $\CLE_\kappa$ in each connected component of $\BB D \setminus \ol{\bigcup \Gamma_P^\ep}$ and with positive probability, the following is true. 
\begin{enumerate}
\item Each loop in $\Gamma_P^\ep$ has diameter at most $\ep$ and is contained in the $\ep$-neighborhood of $P$.
\item The set $\ol{\bigcup \Gamma_P^\ep}$ is connected and intersects $B_\ep(P(0))\cap \bdy \BB D$ and $B_\ep(P(1)) \cap \bdy \BB D$. 
\end{enumerate}
\end{lem}
\begin{proof} 
By Lemma~\ref{lem-miller-wu-dim} and the branching $\SLE_\kappa(\kappa-6)$ construction of $\CLE_\kappa$, we find that the following is true. Let $I\subset \bdy \BB D$ be a connected boundary arc, let $U$ be a neighborhood of $I$ in $\BB D$, and let $K\subset U$ be a compact set. With positive probability, each loop of $\Gamma$ which intersects $I$ is contained in $U$ and the union of all such loops disconnects $K$ from $\bdy U \setminus \bdy \BB D$. 

We will construct $\Gamma_P^\ep$ via the following inductive procedure. Let $D_0 = \BB D$, let $t_0 = 0$, and let $I_0 := \bdy\BB D\cap B_\ep(P(0))$. 
Inductively, suppose $n\in\BB N$, a simply connected domain $D_{n-1}\subset \BB D$, and an arc $I_n\subset \bdy\BB D$ have been defined. 
If $D_{n-1}  \cap P =\emptyset$, set $D_n = I_n =  \emptyset$ and $t_n=1$. 
Otherwise, let $D_n$ be the connected component containing $P(1)$ on its boundary of $D_{n-1} \setminus \ol{\bigcup \Gamma(I_n)}$, with $\Gamma(I_n)$ as in~\eqref{eqn-loop-restrict}. 
Let $t_n$ be the largest $t \in [0,1]$ for which $P(t) \in \bdy D_n$ (or $t_n = 1$ of no such $t$ exists) and let $I_n$ be the connected component of $(\bdy D_n\setminus \bdy\BB D) \cap B_{\ep/4}(P(t_n))$ which contains $P(t_n)$. 
We set 
\eqbn
\Gamma_P^\ep := \bigcup_{n=1}^\infty \Gamma(I_n) .
\eqen

We will now argue that $\Gamma_P^\ep$ satisfies the conditions in the statement of the lemma. The description of the law of $\Gamma\setminus \Gamma_P^\ep$ is immediate from~\cite[Theorem 5.4, condition 5]{shef-cle}. 
To check that the two listed properties hold with positive probability, set $U_n := D_n \cap B_{\ep/2}(I_n)$, let $s_n$ be the first time after $t_n$ at which $P$ exits $B_{\ep/4}(P(t_n))$ (or $s_n = 1$ if no such time exists), and let $K_n := \eta([t_n,s_n]) \cap \ol{D_n}$. 
Applying the first paragraph and the Markov property of $\CLE_\kappa$~\cite[Theorem 5.4, condition 5]{shef-cle} shows that for each $n\in\BB N$, it holds with positive conditional probability given $(D_n,I_n,t_n)$ that each loop of $\Gamma|_{D_n}$ which intersects $I_n$ is contained in $U_n$ (so has diameter at most $\ep$) and the union of all such loops disconnects $\eta([t_n,s_n])$ from $\bdy U_n \setminus \bdy D_n$ in $D_n$ (which implies that $t_{n+1} \geq s_n$ and hence that $|P(t_{n+1}) - P(t_n)| \geq \ep/4$). 
Consequently, with positive probability there is a finite $n \in\BB N$ for which some loop in $\Gamma(I_n)$ intersects $\bdy\BB D\cap B_\ep(P(1))$.
\end{proof}

\begin{proof}[Proof of Lemma~\ref{lem-cle-small}]
The statement of the lemma follows by applying Lemma~\ref{lem-cle-small-path} finitely many times to build a ``grid" consisting of finitely many collections of loops of the form $\Gamma_P^{\ep/100}$ with the property that each connected component of $D$ minus the closed union of the loops in these collections has diameter at most $\ep$.

To be more precise, by applying a conformal map, we can assume without loss of generality that $D  = [0,1]^2$ is the Euclidean unit square.
Applying Lemma~\ref{lem-cle-small-path} $\lfloor 2\ep^{-1} \rfloor$ times shows that with positive probability, for each horizontal line segment of the form $[0,1] \times \{k\ep/2\}$ for $k = 1,\dots, \lfloor 2\ep^{-1} \rfloor$, there is a connected set of loops in $\Gamma$ with diameter at most $\ep/100$ which are each contained in the $\ep/100$ neighborhood of $[0,1] \times \{k\ep/2\}$ and whose union intersects the left and right boundaries of $D$. 
Furthermore, we can choose these collections in such a way that if $\Gamma^\ep$ denotes their union, then the conditional law of $\Gamma \setminus \Gamma^\ep$ given $\Gamma^\ep$ is that of an independent $\CLE_\kappa$ in each  of the connected components of $D\setminus \ol{\bigcup \Gamma^\ep}$. 

On the positive probability event that $\Gamma^\ep$ satisfies the conditions described above, each connected component of $D\setminus \ol{\bigcup \Gamma^\ep}$ is either contained in the $\ep$-neighborhood of one of the segments $[0,1] \times \{k\ep/2\}$ and has diameter at most $\ep$ or is a horizontal ``strip" between two connected components of $\ol{\bigcup\Gamma^\ep}$ corresponding to two consecutive vertical segments. Applying Lemma~\ref{lem-cle-small-path} to $\lfloor 2\ep^{-1} \rfloor$ evenly spaced vertical segments within each component of this latter type gives us a new collection of loops $\wh\Gamma^\ep \subset\Gamma$ such that with positive probability, each loop in this collection has diameter at most $\ep$ and each connected component of $\Gamma\setminus \ol{\bigcup \wh\Gamma^\ep}$ has diameter at most $\ep$. Since loops in $\Gamma$ do not cross one another, this concludes the proof. 
\end{proof}

The following is a partial analog of Lemma~\ref{lem-sle-abs-cont} for CLE. See Figure~\ref{fig-abs-cont}(d) for an illustration.

\begin{lem} \label{lem-cle-abs-cont}
Let $\kappa \in (4,8)$, let $D,\wt D\subset \BB C$ be simply connected domains, not all of $\BB C$, and let $\Gamma$ (resp.\ $\wt\Gamma$) be a $\CLE_\kappa$ on $D$ (resp.\ $\wt D$).
Let $X \subset \ol{ D\cap \wt D}$ be a closed set which lies at positive distance from $\bdy D\setminus \bdy\wt D$ and from $\bdy \wt D\setminus \bdy D$ and assume that there is a connected component $A$ of $\bdy D\cap \bdy \wt D$ with more than one point which contains $X\cap (\bdy D \cup \bdy\wt D) $. For each $\ep > 0$, there is a coupling of $\Gamma$ and $\wt\Gamma$ such that with positive probability, $\Gamma(X) = \wt\Gamma(X)$ and each loop in $\Gamma(X)$ has diameter at most~$\ep$. 
\end{lem}
\begin{proof}
Assume without loss of generality that $D\not=\wt D$. 
Let $A$ be a connected component of $\bdy D\cap \bdy \wt D$ as in the statement of the lemma. 
Since $X$ lies at positive distance from $(\bdy D\setminus \bdy\wt D)\cup (\bdy \wt D\setminus \bdy D)$, the set $X$ cannot disconnect all of $A$ from $(\bdy D\setminus \bdy\wt D)\cup (\bdy \wt D\setminus \bdy D)$. Therefore, we can find a curve $P$ in $\ol{D\cup \wt D}$ between two points of $A$ which disconnects $X$ from $(\bdy D\setminus \bdy\wt D)\cup (\bdy \wt D\setminus \bdy D)$ and which lies at positive distance from $(\bdy D\setminus \bdy\wt D)\cup (\bdy \wt D\setminus \bdy D)$.

By possibly replacing $P$ with its time-reversal, we can assume without loss of generality that $X$ lies to the right of $P$.
By slightly fattening $P$, we can find non-trivial arcs $I_0, I_1 \subset \bdy D\cap \bdy \wt D$ and a connected open set $V\subset D\cap\wt D$ such that $I_0,I_1$ are each contained in the interior of an arc of $\bdy V\cap \bdy D\cap \bdy\wt D$. 
Choose points $x_0, x_1$ in the interiors of $I_0, I_1$, respectively.
Let $\eta$ (resp.\ $\wt\eta$) be the clockwise branch from $x_0$ to $x_1$ of the branching $\SLE_\kappa(\kappa-6)$ process associated with $\Gamma$ (resp.\ $\wt\Gamma$) as in Lemma~\ref{lem-cle-concatenate}.
That is, if $J$ denotes the clockwise arc of $\bdy D$ from $x_0$ to $x_1$, then $\eta$ is an $\SLE_\kappa(\kappa-6)$ in $D$ from $x_0$ to $x_1$ obtained by concatenating appropriate arcs of the loops in $\Gamma(J)$, and similarly for $\wt\eta$. 
By~\cite[Theorem 5.4, condition 5]{shef-cle}, applied to the arc $J$, if we condition on $\eta$, then conditional law of the restriction of $\Gamma$ to each connected component of $D\setminus \eta$ lying to the right of $\eta$ is that of a $\CLE_\kappa$ in the component. 
Similar statements hold for $\wt\eta$. 

By Lemmas~\ref{lem-miller-wu-dim} and~\ref{lem-mw-sle-abs-cont}, we can couple $\eta$ and $\wt\eta$ such that with positive probability, $\eta = \wt \eta \subset V$. 
On this event, the conditional laws of $\Gamma(X)$ and $\wt\Gamma(X)$ given $\eta$ and $\wt\eta$ agree. 
Combining this with Lemma~\ref{lem-cle-small} now yields a coupling as in the lemma.
\end{proof}

\bibliography{cibiblong,cibib}

\bibliographystyle{hmralphaabbrv}

\bigskip

\filbreak
\begingroup
\small
\parindent=0pt

\bigskip
\vtop{
\hsize=5.3in
eg558@cam.ac.uk (Ewain Gwynne)\\
jpmiller@statslab.cam.ac.uk (Jason Miller, \emph{corresponding author})\\
wq214@cam.ac.uk (Wei Qian)\\

\smallskip

Statistical Laboratory\\
Center for Mathematical Sciences\\
Wilberforce Road\\
Cambridge CB3 0WB
United Kingdom} \endgroup \filbreak

\end{document}